\documentclass[a4paper]{article}
\usepackage[all]{xy}\usepackage[latin1]{inputenc}        %accents
\usepackage[dvips]{graphics,graphicx}
\usepackage{amsfonts,amssymb,amsmath,color,mathrsfs, amstext}
\usepackage{amsbsy, amsopn, amscd, amsxtra, amsthm,authblk}
\usepackage{enumerate,algorithmicx,algorithm}
\usepackage{algpseudocode}
\usepackage{upref}
\usepackage{geometry}
\usepackage{amsthm,amsmath,amssymb}
\usepackage{relsize}
\usepackage{mathrsfs}
\usepackage{bm}
\usepackage{exscale}
\usepackage{graphicx}
\usepackage{tabularx}
\usepackage{threeparttable,multirow,dcolumn,booktabs}
\usepackage{algorithmicx,algorithm}
\usepackage{tabularx}
\usepackage{caption}
\usepackage{float}
\usepackage{bm}% bold math
\usepackage{caption}
\usepackage{yhmath}
\usepackage{booktabs}
\usepackage{subcaption}
\usepackage{multirow}
\usepackage{makecell}
\usepackage[colorlinks,
            linkcolor=red,
            anchorcolor=red,
            citecolor=red
            ]{hyperref}

\numberwithin{equation}{section}

\numberwithin{equation}{section}

\DeclareMathOperator{\erf}{erf}
\DeclareMathOperator{\erfc}{erfc}

\usepackage{soul}
\usepackage{enumitem}
\usepackage{bbm}
\newtheorem{theorem}{Theorem}[section]%[section]
\newtheorem{lemma}[theorem]{Lemma}%[section]
\newtheorem{definition}[theorem]{Definition}%[section]
\newtheorem{remark}[theorem]{Remark}%[section]
\newtheorem{proposition}[theorem]{Proposition}%[section]
\newtheorem{assumption}[theorem]{Assumption}%[section]
\usepackage{comment}
\usepackage{exscale}
\usepackage{relsize}
\usepackage{booktabs}
\usepackage{longtable}
\usepackage{multirow}
\usepackage{makecell}

\numberwithin{equation}{section}

\begin{document}

\title{A fast spectral sum-of-Gaussians method for electrostatic summation in quasi-2D systems}

%\date{\today}

\author[1,2]{Xuanzhao Gao
\thanks{xz.gao@connect.ust.hk}}

\author[3]{Shidong Jiang
\thanks{sjiang@flatironinstitute.org}}

\author[3,4]{Jiuyang Liang
\thanks{jliang@flatironinstitute.org; liangjiuyang@sjtu.edu.cn}}

\author[4]{Zhenli Xu
\thanks{xuzl@sjtu.edu.cn}}

\author[4]{Qi Zhou
\thanks{zhouqi1729@sjtu.edu.cn}}
%\thanks{The work of J. L., Z. X. and Q. Z. are supported by the Natural Science Foundation of China (grants No. 12325113, 12401570 and 12426304) and the Science and Technology Commission of Shanghai Municipality (grant No. 21JC1403700). The work of J. L. is partially supported by the China Postdoctoral Science Foundation (grant No. 2024M751948). The work of X. G. is supported by the Natural Science Foundation of China (grant No. 12201146). This work was partially supported by SJTU Kunpeng \& Ascend Center of Excellence.} 

\affil[1]{Thrust of Advanced Materials, The Hong Kong University of Science and Technology (Guangzhou), 511458, Guangzhou, China}
\affil[2]{Department of Physics, The Hong Kong University of Science and Technology, 999077, Hong Kong SAR, China}
\affil[3]{Center for Computational Mathematics, Flatiron Institute, Simons Foundation, USA}
\affil[4]{School of Mathematical Sciences, MOE-LSC and CMA-Shanghai, Shanghai Jiao Tong University, 200240, Shanghai, China}

\date{}
\maketitle

%    Abstract is required.
\begin{abstract}
The quasi-2D electrostatic systems, characterized by periodicity in two dimensions with a free third dimension, have garnered significant interest in the fields of semiconductor physics, new energy technologies, and nanomaterials. We apply the sum-of-Gaussians (SOG) approximation to the Laplace kernel, dividing the interactions into near-field, mid-range, and long-range components. The near-field component, singular but compactly supported in a local domain, is directly calculated. The mid-range component is managed using a procedure similar to nonuniform fast Fourier transforms (NUFFTs) in three dimensions. The long-range component, which includes Gaussians of large variance, is treated with 
polynomial interpolation/anterpolation in the free dimension and Fourier spectral solver in the other two dimensions on proxy points. Unlike the fast Ewald summation, which requires extensive zero padding in the case of high aspect ratios, the separability of Gaussians allows us to handle such case without any zero padding in the free direction. Furthermore, while NUFFTs typically rely on certain upsampling in each dimension, and the truncated kernel method introduces an additional factor of upsampling due to kernel oscillation, our scheme eliminates the need for upsampling in any direction due to the smoothness of Gaussians, significantly reducing computational cost for large-scale problems.
Finally, whereas all periodic fast multipole methods require dividing the periodic tiling into a smooth far part and a near part containing its nearest neighboring cells, our scheme operates directly on the fundamental cell, resulting in better performance with simpler implementation. 
We provide a rigorous error analysis showing that upsampling is not required in NUFFT-like steps, and develop a careful parameter selection scheme to balance various parts of the whole scheme,  achieving $O(N\log N)$ complexity with a small prefactor. The performance of the scheme is demonstrated via extensive numerical experiments. The scheme can be readily extended to deal with many other kernels, owing to the general applicability of the SOG approximations.


\end{abstract}

\section{Introduction}
The study of many-body charged systems is fundamental to numerous scientific disciplines. In numerical simulations, these systems are often modeled using fully periodic boundary conditions in three dimensions (3D-PBC) to mitigate finite-volume effects~\cite{Frenkel2001Understanding}. 
However, this approach fails to accurately capture the unique physical properties that arise from interfaces between media and solid surfaces, which differ significantly from bulk properties. To overcome this limitation, ``quasi-2D" models~\cite{mazars2011long} are used, where two dimensions are periodic, and the third is confined between two parallel planes separated by nanometers or angstroms. Quasi-2D systems have gained considerable interest in studies of ultra-cold dilute gases~\cite{anderson1995observation}, nanofluidic devices~\cite{Robin2021Science}, dusty plasma~\cite{teng2003microscopic}, magnetic films~\cite{o1986magnetic}, and liquid crystal films~\cite{kawamoto2002history}.

In quasi-2D systems, pairwise interactions, particularly long-range electrostatics, pose computational challenges due to their
$O(N^2)$ complexity, where $N$ is the number of particles. Traditional Ewald splitting techniques~\cite{Ewald1921AnnPhys} fail to resolve this issue, as the resulting complexity of the Ewald2D summation method~\cite{parry1975electrostatic} remains $O(N^2)$, hindering large-scale simulations of such systems. 
To overcome this computational barrier, fast algorithms, primarily Fourier spectral methods~\cite{parry1975electrostatic,lindbo2012fast,nestler2015fast,shamshirgar2021fast,maxian2021fast} based on the Ewald splitting and solvers accelerated by the fast multipole method (FMM)~\cite{greengard1987fast,fmm2,fmm3,fmm4,fmm6,fmm7,fmm8}, have been developed. 
Fourier spectral methods distribute particles on a uniform grid and solve Poisson's equation in the Fourier domain using fast Fourier transform (FFT) pairs for acceleration. FMM-based solvers extend the free-space FMM approach to quasi-2D systems, adapting to partial periodic boundary conditions~\cite{yan2018flexibly,liang2020harmonic,jiang2023jcp}. 
Other methods have also been proposed, such as the MMM2D method~\cite{arnold2002novel}, the SOEwald2D method~\cite{gan2024fast}, and various correction approaches~\cite{yeh1999ewald,arnold2002electrostatics}. By combining with either FFT or FMM, these methods achieve $O(N\log N)$ or even $O(N)$ complexity. Many popular software packages~\cite{abraham2015gromacs,thompson2022lammps} use variants of the Ewald summation method due to its ease of implementation and the availability of high-quality FFT code on almost all modern computer architectures.

Despite significant advancements in this field, several substantial challenges persist. One major challenge is the large prefactor in
$O(N)$ or $O(N\log N)$ compared to 3D-PBC solvers~\cite{mazars2011long}. Specifically, spectral methods often use Ewald splitting~\cite{Ewald1921AnnPhys} to divide the Coulomb kernel into real and Fourier parts. However, the singularity introduced by the Laplacian in the Fourier integral along the free direction poses formidable obstacles to Fourier calculations. Current spectral solvers address this issue through techniques such as truncation~\cite{parry1975electrostatic}, regularization~\cite{nestler2015fast}, or periodic extension~\cite{lindbo2012fast} of special functions. However, these approaches often involve trade-offs, such as a reduction in convergence order to algebraic convergence or the need for additional zero-padding.
Recent advancements by Shamshirgar \emph{et al.}~\cite{shamshirgar2021fast} explore the integration of spectral solvers with truncated kernel methods (TKM)~\cite{vico2016fast}, 
highlighting their potential to reduce the zero-padding factor to an optimal value of $2$~\cite{liu2024optimal} with additional adaptive upsampling. 
This results in an electrostatic solving time for quasi-2D systems that is roughly twice as long as that for cubic systems with 3D-PBC.

Another challenge arises in scenarios where the free direction is strongly confined~\cite{mazars2011long}, i.e., $L_z \ll \min\{L_x, L_y\}$. In this context, algorithms based on FFT or FMM experience significant performance loss. Specifically, to achieve a given accuracy, the zero-padding issue of the former becomes worse~\cite{maxian2021fast}, while the latter requires the inclusion of more near-field contributions~\cite{yan2018flexibly}. The recently  developed methods like the ATKM~\cite{greengard2018anisotropic}, 
periodic FMM~\cite{jiang2023jcp}, and dual-space multilevel kernel-splitting method~\cite{greengard2023dual} offer potential solutions to this challenge, but these methods have not yet been extended to handle quasi-2D systems.%\tcr{add more discussion about ATKM}

In this study, we introduce an efficient spectral sum-of-Gaussians (SOG) method for electrostatic summation in quasi-2D systems. Instead of the traditional Ewald splitting, we apply the following SOG decomposition:
\begin{equation}\label{eq::SOGdecomp}
    \frac{1}{r}\approx\underbrace{\left(\frac{1}{r}-\sum_{\ell=0}^{M}w_{\ell}e^{-r^2/s_{\ell}^2}\right)\mathbbm{1}_{r\leq r_c}}_{\text{near-field}} +\underbrace{\sum_{\ell=0}^{m}w_{\ell}e^{-r^2/s_{\ell}^2}}_{\text{mid-range}}+\underbrace{\sum_{\ell=m+1}^{M}w_{\ell}e^{-r^2/s_{\ell}^2}}_{\text{long-range}}.
\end{equation}
Here, $s_{\ell}$ and $w_{\ell}$ are the nodes
and weights of the SOG approximation for the $1/r$ kernel~\cite{beylkin2010approximation}, $M$ is the number of Gaussians in the SOG approximation, which depends on the prescribed precision and the geometric parameters $L_x$, $L_y$, and $L_z$. The nodes $s_{\ell}$ are arranged in monotone increasing order, i.e., 
$s_{0}<s_1<\cdots<s_{M}$.
$\mathbbm{1}_{r\leq r_c}$ is the indicator function that equals $1$ when $r\leq r_c$ and $0$ otherwise. The cutoff parameter $r_c$ is chosen to balance the work between direct near-field and FFT-based far-field calculations. We also fine tune the value of $r_c$ so that the near-field is equal to zero at $r_c$, which ensures the continuity of our numerical scheme across the cutoff point. The parameter $m$ is chosen to balance the workload between the mid-range and long-range parts.   

As shown in \cite{beylkin2010approximation},
the SOG approximation of the $1/r$ kernel 
converges spectrally fast, with $M$ depending 
on the prescribed precision logarithmically.
The decomposition Eq.~\eqref{eq::SOGdecomp} effectively removes the singularity of $\Delta^{-1}$ in Fourier space without the need of kernel 
truncation, as the Fourier transform of a Gaussian remains a Gaussian. The first term on the right-hand side of Eq.~\eqref{eq::SOGdecomp} is computed directly within a cutoff $r_c$ in physical space. The second term, i.e., the mid-range part is handled by a 3D Fourier spectral solver which is similar to the NUFFT~\cite{nufft1,nufft2,nufft3}. The third term, i.e., the long-range part
uses Chebyshev interpolation and anterpolation
to deal with the $z$ direction (the free direction), and a Fourier-based approach on planes parallel to the
$xy$-plane to deal with periodic directions.
% The parameter $\rm mid$ is chosen 
% to balance the workload between these two parts. 
We show that the criterion of choosing $m$ is $s_{m}\sim O(L_z)$. Finally,
we prove rigorously that upsampling is not needed
in NUFFT-related steps and zero-padding is not needed in FFT-related steps, due to the rapid decay of
the Fourier transform of the Gaussian kernels.

The resulting spectral SOG method offers several advantages, demonstrated by rigorous error analysis and systematic experiments. It achieves a complexity of $O(N\log N)$, reducing the number of grids required across entire range of aspect ratios of the system considerably compared to Ewald-based spectral methods. In highly anisotropic scenarios, the improvement on performance could exceed one order of magnitude. This method exhibits broad applicability in simulation frameworks like molecular dynamics (MD) and Monte Carlo (MC) simulations, and can be seamlessly extended to other kernels through integration with the kernel-independent SOG method~\cite{greengard2018anisotropic,gao2022kernel}.

To summarize, the spectral SOG scheme developed in this paper has the following novel features:
\begin{enumerate}[label=(\arabic*)]
\item Unlike the schemes for quasi-2D systems implemented in polular software packages such as GROMACS and LAMMPS that are of low order and low accuracy and contain some heuristic arguments, our scheme is based on rigorous error analysis for each step with clear guidance to the selection of the parameters. Moreoever, our scheme is spectrally accurate and can achieve machine precision with mild increase in computational cost.
\item Unlike the fast Ewald method, our scheme does not need zero-padding for systems that are confined in a rectangular box of high aspect ratio. This is because each Gaussian is separable and we divide the Gaussians into two groups.
The smoothness and separability of the Gaussian also removes the need of kernel truncation in the free direction.
\begin{comment} 
\item Unlike recent random batch Ewald method that has certain accuracy only in the sense of statistical average across many steps in the molecular dynamics simulation, our scheme is spectrally accurate and can achieve machine precision for every step. On the other hand, the random batch Ewald method has demonstrated extremely impressive parallel efficiency for systems with distributed memory.
\end{comment}
\item Unlike the NUFFT based methods, our scheme does not require any upsampling in the gridding step. This is because we have taken advantage of the fact that the Fourier transform of the Gaussian decays quickly and it compensates the loss of accuracy in calculating the Fourier transform of the data. In other words, Fourier transform of the data is only an intermediate step and the loss of accuracy of the Fourier transform of the data for large wavenumber does not affect the accuracy of the final result.
\item Unlike the periodic FMMs, our scheme does not need dividing the periodic tiling into a smooth part and a local part containing nine nearest neighbors of the fundamental cell, resulting in simpler implementation and higher efficiency. Instead, all calculations are carried out in the fundamental cell itself.
\end{enumerate}

The structure of the remaining sections is as follows. Section~\ref{sec::2DElec} revisits the difficulties encountered in the classical Ewald decomposition when applied to quasi-2D systems and introduces an SOG decomposition as an alternative. Section~\ref{sec::FSSOG} elaborates on the spectral SOG solver in detail. Section~\ref{sec::erroranalysis} is dedicated to providing an analysis of errors and discussing the selection of parameters. Section~\ref{sec::numer} showcases numerical results that confirm both accuracy and efficiency of the proposed solver. Concluding remarks are provided in Section~\ref{sec::con}.

\section{Sum-of-Gaussians decompositions of electrostatic interactions for quasi-2D systems}\label{sec::2DElec}	
In this section, we begin by reviewing the basic electrostatic model, the classical Ewald summation for quasi-2D systems, and the associated computational challenges. We then introduce an SOG decomposition for the Coulomb kernel and discuss the related decomposition error.

\subsection{Quasi-2D electrostatic interactions and kernel decomposition} \label{sec::Diff}
Consider a system of $N$ point sources with charges $q_i\in\mathbb{R}$ located at positions $\bm{r}_{i}=(x_i,y_i,z_i)\in\Omega$, $i=1,2,\cdots,N$, in a rectangular box \begin{equation}
\Omega=[-L_x/2,L_x/2]\times[-L_y/2,L_y/2]\times[-L_z/2,L_z/2]\subset\mathbb{R}^{3}.
\end{equation}
To fix the discussion, we assume that the system is periodic in the $x$ and $y$ directions, while the $z$ direction remains free. It is straightforward to show that the electrostatic potential $\Phi(\bm{r})$ satisfies the free-space Poisson equation~\cite{jackson1999classical}
\begin{equation}\label{eq::Poisson}
    -\Delta\Phi(\bm{r})=f(\bm{r}), \quad \bm{r}\in\mathbb{R}^3,
\end{equation}
with the charges periodically tiling up
in the $x$ and $y$ directions
\begin{equation}
f(\bm{r})=4\pi\sum_{\bm{n}\in\mathcal{N}_2}\sum_{i=1}^{N}q_{i}\delta(\bm{r}-\bm{r}_i+\bm{n}\circ\bm{L}).
\end{equation}
Here, $\Delta$ is the Laplace operator, $\delta$ is the Dirac delta function, ``$\circ$'' represents the element-wise product of two vectors, and $\bm{L}=(L_x,L_y,L_z)$. The set $\mathcal{N}_2=\left\{\left(n_{x}, n_{y}, 0\right): n_{x}, n_{y} \in \mathbb{Z}\right\}$ imposes periodicity in the $x$ and $y$ directions. The system is assumed to be charge neutral, i.e.,
\begin{equation}\label{eq::chargeneutral}
    \sum_{i=1}^{N}q_i=0.
\end{equation} 
To ensure well-posedness, we also impose the conditions that the gradient $\nabla\Phi(\bm{r})$ vanishes at infinity in the free $z$-direction, and that the limit
\begin{equation}
\lim_{z\rightarrow \pm \infty}\Phi(\bm{r})=\pm\frac{2\pi}{L_xL_y}\sum_{j=1}^{N}q_jz_j
\end{equation}
must be a constant~\cite{lindbo2012fast}. Given the charge distribution, the electrostatic potential $\Phi$ at $\bm{r}_i$ is given by the discrete sum
\begin{equation}\label{eq::Phi}
    \Phi(\bm{r}_i)=\sum_{j=1}^{N} \sum'_{\boldsymbol{n} \in \mathcal{N}_{2}} \frac{q_{j}}{\left|\boldsymbol{r}_{ij}+\boldsymbol{n} \circ \boldsymbol{L}\right|}
\end{equation}
where $\bm{r}_{ij}:=\bm{r}_i-\bm{r}_j$, and the prime indicates that the self-interaction, i.e., the case $i=j$ when $\bm{n}=\bm{0}$, is excluded in the summation. Proposition~\ref{lem::summation} shows that Eq.~\eqref{eq::Phi} is well defined~\cite{de1980prsla,smith2008electrostatic}.

\begin{proposition}\label{lem::summation}
The infinite summation in Eq.~\eqref{eq::Phi} is absolutely convergent under the condition Eq.~\eqref{eq::chargeneutral}.
\end{proposition}

Throughout the remainder of this paper, we will make extensive use of the Fourier transform pair. To simplify the description, we introduce some non-standard notations, following ~\cite{bagge2022fast,shamshirgar2021fast}. 
	\begin{definition}
		\label{def::1}
		Let $\dot{\bm{r}}=(x, y)\in\mathcal{R}^2$ and $\dot{\bm{k}}=(k_x, k_y)\in \mathcal{K}^2$ denote the periodic part of position and Fourier mode, where 
		\begin{equation}
             \label{eq::R2}
			\mathcal{R}^2:=\left\{(x,y)\in\mathbb{R}^2 {\Big |}x \in \left[-\frac{L_x}{2},\frac{L_x}{2}\right], y\in \left[-\frac{L_y}{2},\frac{L_y}{2}\right]\right\},
		\end{equation}
		and 
		\begin{equation}
             \label{eq::K2}
			\mathcal{K}^{2}:=\left\{(k_x,k_y) \in \mathbb{R}^{2}{\Big |} k_{d} \in \frac{2 \pi}{L_d} \mathbb{Z}, \text{   }d=x, y\right\}.
		\end{equation}
	\end{definition} 
	\begin{definition}
		\label{def::2}
		For any function $g: \mathcal{K}^{2} \times \mathbb{R} \rightarrow \mathbb{C}$, we define the 2-periodic summation operator $\mathcal{L}$ by
		\begin{equation}
			\label{eq::MixingSum}
			\mathcal{L}[g(\boldsymbol{k})]=\mathcal{L}[g([\dot{\bm{k}}, k_z])]:=	\dfrac{1}{2 \pi L_xL_y} \sum\limits_{\dot{\boldsymbol{k}} \in \mathcal{K}^{2}} \mathlarger{\int}_{\mathbb{R}} g\left([\dot{\bm{k}}, k_z]\right) \mathrm{d} k_z.
		\end{equation}
	\end{definition} 
	\noindent Let $f\left([\dot{\bm{r}},z]\right)$ be a function that is periodic in $\dot{\bm{r}}$ and non-periodic in $z$, and its Fourier transform is defined by
	\begin{equation}
	 \label{eq::FT}
\hat{f}\left(\left[\dot{\boldsymbol{k}}, k_z\right]\right):=\int_{\mathcal{R}^2}\int_{\mathbb{R}}  f([\dot{\boldsymbol{r}}, z]) e^{-i \dot{\bm{k}} \cdot \dot{\bm{r}}} e^{-i k_z z} \mathrm{d} z\mathrm{~d} \dot{\boldsymbol{r}} .
	\end{equation}
Then, the inverse Fourier transform is given by
\begin{equation}
        \label{eq::IFT}
		f([\dot{\boldsymbol{r}}, z])=\mathcal{L}\left[\hat{f}\left(\left[\dot{\boldsymbol{k}}, k_z\right]\right) e^{i \dot{\bm{k}} \cdot \dot{\boldsymbol{r}}} e^{i k_zz}\right].
\end{equation} 

While the periodicity is limited to the first two dimensions, directly calculating the solution Eq.~\eqref{eq::Phi} for all particles has slow convergence and an unacceptable complexity of $O(N^2)$. In order to reduce the cost, Ewald~\cite{Ewald1921AnnPhys} began by introducing a screen function $\tau$ and decomposed the source term $f(\bm{r})$ in Eq.~\eqref{sec::Diff} as the contributions of near- and far-field parts:
\begin{equation}\label{eq::2.5}
f=f^{\mathcal{N}}+f^{\mathcal{F}},\quad f^{\mathcal{N}}=f-\left(f*\tau\right),\quad f^{\mathcal{F}}=f*\tau,
\end{equation}
where ``$*$'' denotes the convolution operator. Next, $\Phi$ can be obtained by solving the Poisson equation with the right-hand side equal to $f^{\mathcal{N}}$ and $f^{\mathcal{F}}$, respectively. We denote the solutions corresponding to these two right-hand sides by $\Phi^{\mathcal{N}}$ and $\Phi^{\mathcal{F}}$. Then the solution to the original problem in Eq.~\eqref{eq::Poisson} can be written as
\begin{equation}\label{eq::phisplit}
	\Phi(\bm{r}_i)=\Phi^{\mathcal{N}}(\bm{r}_i)+\Phi^{\mathcal{F}}(\bm{r}_i)-\Phi^{\text{self}}_i,
\end{equation}
where the self term $\Phi^{\text{self}}_i$ is subtracted to exclude the unwanted interaction of charges with themselves introduced in $\Phi^{\mathcal{F}}$. In the classical Ewald method, the screen function $\tau$, with its Fourier transform $\widehat{\tau}$, is selected as
\begin{equation}
    \tau_{\text{Ewald}}(\bm{r})=(2\pi \sigma^2)^{-3/2}e^{-r^2/(2\sigma)},\text{ }\quad \widehat{\tau}_{\text{Ewald}}(\bm{k})=e^{-\sigma^2k^2/2}, %\xi^3 \pi^{-3 / 2} e^{-\xi^2r^2},\text{ }\quad \widehat{\tau}_{\text{Ewald}}(\bm{k})=e^{-k^2/4\xi^2},
    \end{equation}
where $r=|\bm{r}|$, $k=|\bm{k}|$, and $\sigma$ is a positive parameter to be chosen to balance the workload on the evaluation of $\Phi^{\mathcal{N}}$ and $\Phi^{\mathcal{F}}$. The near-field part $\Phi^{\mathcal{N}}$ decays rapidly and thus is evaluated directly in physical space, since the corresponding kernel can be numerically truncated at a cutoff distance $r_c$. The far-field part $\Phi^{\mathcal{F}}$ is smooth and is evaluated in Fourier space. Using the 2-periodic summation operator in Eq.~\eqref{eq::MixingSum}, we rewrite $\Phi^{\mathcal{F}}$ as
\begin{equation}\label{eq::2.13}
\Phi^{\mathcal{F}}(\bm{r}_i)=
 4 \pi \sum_{j=1}^N q_{j} \mathcal{L}\left[\frac{e^{-\sigma^2(\dot{k}^2+k_z^2) / 2}}{\dot{k}^2+k_z^2} e^{\mathrm{i}\dot{\bm{k}}\cdot\dot{\bm{r}}_{ij}}e^{\mathrm{i}k_zz_{ij}}\right].
\end{equation}
We highlight some challenges associated with the factor $1/(\dot{k}^2+k_z^2)$
in Eq.~\eqref{eq::2.13}. In order to apply spectral methods and the FFT, we would like to discretize the integral (along the $z$-axis) in Eq.~\eqref{eq::2.13} using the trapezoidal rule on a uniform grid. 
However, when $\dot{\bm{k}}$ is equal to $\bm{0}$, this factor becomes singular at $k_z=0$.
Furthermore, for non-zero but small values of $\dot{\bm{k}}$, the integrand exhibits rapid variation when $k_z$ approaches 0, necessitating zero-padding by a factor of $6$ to achieve the desired precision in the cubic domain~\cite{lindbo2012fast}. Recently, the spectral Ewald (SE) method~\cite{shamshirgar2021fast} removes the singularity at the origin in Fourier space through the use of TKM~\cite{vico2016fast}, which unfortunately introduces additional oscillations and requires further upsampling. Moreover, it was shown recently~\cite{gan2024random} that the upsampling factor scales like $(\log\varepsilon^{-1})\cdot\max\{L_x,L_y\}/(2\pi L_z)$ for Ewald splitting-based methods, where $\varepsilon$ is the desired precision, indicating that the upsampling factor grows as the aspect ratio of the domain increases. These issues concerning quasi-2D systems remain far from settled.

\subsection{The SOG decomposition}\label{subsec::SOGdecomp}
To overcome the singularity in Fourier space in Eq.~\eqref{eq::2.13}, we replace the Ewald decomposition with the SOG decomposition of the $1/r$ kernel~\cite{beylkin2010approximation,DEShaw2020JCP}.
We start from the following integral representation:
\begin{equation}
\label{GammaExpansion}
\dfrac{1}{r}=\frac{1}{\sqrt{\pi}}\int_{-\infty}^{\infty}e^{-e^tr^2+\frac{1}{2} t}dt.
\end{equation}
Applying the change of variable $t=\log(x^2/2\sigma^2)$ and discretizing the resulting integral with the nodes $x_\ell=b^{-\ell}$, we obtain the so-called bilateral series approximation (BSA)~\cite{beylkin2010approximation} 
\begin{equation}
\label{BSA}
\dfrac{1}{r}\approx\dfrac{2\log b}{\sqrt{2\pi\sigma^2}}\sum_{\ell=-\infty}^{\infty}\dfrac{1}{b^{\ell}}e^{-\frac{r^2}{(\sqrt{2}b^{\ell}\sigma)^2}},
\end{equation}
with the relative error asymptotically bounded (as $b$ approaches $1$) by
\begin{equation}
\label{eq::pointwiseerror}
\left|1-\dfrac{2r\log b}{\sqrt{2\pi\sigma^2}}\sum_{\ell=-\infty}^{\infty}\dfrac{1}{b^{\ell}}e^{-\frac{r^2}{(\sqrt{2}b^{\ell}\sigma)^2}}\right|\lesssim 2\sqrt{2}e^{-\frac{\pi^2}{2\log b}},\quad \forall r>0.
\end{equation}
Here, $b>1$ is a parameter determining the distribution of the nodes and $\sigma$ controls the bandwidth of each Gaussian. Note that Eq.~\eqref{eq::pointwiseerror} is independent of $\sigma$. The so-called u-series \cite{DEShaw2020JCP} decomposes the $1/r$ kernel into the sum of a near part and a far part:
\begin{equation}\label{eq::SOGDEcomp0}
1/r=\mathcal{N}_{b}^{\sigma}(r)+\mathcal{F}_{b}^{\sigma}(r),
\end{equation}
where the near part
\begin{equation}
\label{eq::SOGDEcomp}
\mathcal{N}_{b}^{\sigma}(r)=\begin{cases}
1/r-\mathcal{F}_{b}^{\sigma}(r),\quad\text{if}~r<r_c\\\\0,\qquad\qquad\quad\,\,\, \text{if}~r\geq r_c.\end{cases}\quad\,
\end{equation}
is short-ranged, and the far part \begin{equation}
\label{eq::SOGField}
\mathcal{F}_{b}^{\sigma}(r)=\sum_{{\ell}=0}^{M}w_{\ell} e^{-r^2/s_{{\ell}}^2}
\end{equation}
is long-ranged. Here, 
\begin{equation}
w_{\ell}=(\pi/2)^{-1/2}b^{-\ell}\sigma^{-1}\log b \quad\text{  and  } \quad s_{\ell}=\sqrt{2}b^{\ell}\sigma.
\end{equation}
Note that only $\ell\geq 0$ terms in Eq.~\eqref{eq::pointwiseerror} are included in the far part of BSA, and the summation is truncated at $\ell=M$. The cutoff $r_c$ is selected to be the smallest root of the equation  
\begin{equation}\label{Eq::2.20}
    \frac{1}{r}-\mathcal{F}^{\sigma}_b(r)=0,
\end{equation}
which enforces the $C^0$ continuity across the cutoff point. For MD simulations, practitioners would like to ensure the $C^1$ continuity across the cutoff point,
i.e., the continuity of the force:
\begin{equation}\label{eq::2.22}
\frac{d}{dr}\left[\frac{1}{r}-\mathcal{F}_{b}^{\sigma}(r)\right]{\bigg|}_{r=r_c}=0.
\end{equation}
To achieve this, we introduce an
additional parameter $\omega$ in the weight of the narrowest Gaussian as follows:
\begin{equation}
w_0=\omega (\pi/2)^{-1/2}\sigma^{-1}\log b.
\end{equation}
By solving Eqs.~\eqref{Eq::2.20} and \eqref{eq::2.22} together, we find the values for $r_c$ and $\omega$. The u-series has been successfully adopted to accelerate electrostatic calculations in 3D-PBC systems~\cite{DEShaw2020JCP,liang2023random}, reducing the computational cost by half.
%FFT-based methods~\cite{DEShaw2020JCP} and random batch-type methods~\cite{liang2023random} in 3D-PBC systems, reducing half of the cost. 

In this paper, we apply the u-series to study quasi-2D systems. It is clear that the far part has no singularity in Fourier space since the Fourier transform of a Gaussian is again a Gaussian. We remark here that the SOG decomposition of the kernel is equivalent to the  following selection of the screen function acting on the data:
\begin{equation}
\tau_{\text{SOG}}(\bm{r})=\dfrac{1}{4\pi}\sum_{\ell=0}^{M}w_{\ell} \left(\dfrac{6}{s_{\ell}^2}-\dfrac{4r^2}{s_{\ell}^4}\right)e^{-r^2/s_{\ell}^2},\quad\text{  }\widehat{\tau}_{\text{SOG}}(\bm{k})=\dfrac{\sqrt{\pi}}{4}\sum_{\ell=0}^{M}w_{\ell}s_{\ell}^3k^2e^{-s_{\ell}^2k^2/4}.
\end{equation}
Similar to Eqs.~\eqref{eq::2.5}-\eqref{eq::2.13}, the far-field component of potential from the u-series decomposition is given by
\begin{equation}
\label{eq::phindf}
\Phi^{\mathcal{F}}(\bm{r}_i)=4\pi\sum_{j=1}^{N}q_j\mathcal{L}\left[\dfrac{\widehat{\tau}_{\text{SOG}}(\bm{k})}{k^2}e^{\mathrm{i}\bm{k}\cdot\bm{r}_{ij}}\right],
\end{equation}
and the singular term $1/k^2$, arising from the inverse of the Laplacian operator, is cancelled out by the $k^2$ factor in $\widehat{\tau}_{\text{SOG}}(\bm{k})$. The near-field term and self-interaction term of the u-series are given by 
    \begin{equation}
    	\label{eq::phiNDR}
    	\Phi^{\mathcal{N}}(\bm{r}_i)=\sum_{\bm{n}\in\mathcal{N}_2}'\sum_{j=1}^{N}q_j\mathcal{N}_{b}^{\sigma}(|\bm{r}_{ij}+\bm{n}\circ\bm{L}|),\qquad \Phi^{\text{self}}_i=q_i\sum_{\ell=0}^{M}w_{\ell}.
    \end{equation} 
The sum in $\Phi^{\mathcal{N}}(\bm{r}_i)$ converges absolutely and rapidly, 
requiring consideration of particles within a sphere of radius $r_c$ centered at $\bm{r}_i$. In practice, a cell list~\cite{verlet1967computer} is generated for each target point $\bm{r}_i$ to streamline the pair searching process, confining the calculations to this list. 
It is clear that with the cutoff~$r_c$ fixed, the total computational cost of evaluating $\Phi^{\mathcal{N}}(\bm{r}_i)$ and $\Phi^{\text{self}}_i$ 
for $i=1,\ldots, N$ is $O(N)$, when particles are distributed more or less 
uniformly.

\begin{remark}\label{rmk::sum}
The calculation of $\Phi^{\mathcal{N}}(\bm{r}_i)$ involves the evaluation of $M+1$ Gaussians for each target-source pair.
This is rather expensive. In practice, we accelerate this process by applying the Taylor expansions and bitmask-based table lookup techniques, as described in~\emph{\cite{liang2023random}}. These techniques accelerate the calculation of the near part significantly, and the associated computational cost is independent of number of Gaussians
in the SOG decomposition. 
\end{remark}

\subsection{Decomposition error and parameter selection in quasi-2D systems}\label{subsec::truncateerror}
In this section, we discuss error estimates and an associated parameter selection scheme for the energy and force calculations introduced by the SOG decomposition.

Since the SOG decomposition Eq.~\eqref{eq::SOGDEcomp0} is
exact for $r\le r_c$, the potential error is given by
\begin{equation}\label{eq::uerror}
	\Phi_{\text{err}}(\bm{r}_i):=\sum_{\boldsymbol{n} \in \mathcal{N}_{2}}\sum_{j=1}^{N}q_jK(|\bm{r}_j-\bm{r}_i+\bm{n}\circ \bm{L}|),
\end{equation}
where the kernel $K(r)$ is defined as
\begin{equation}\label{eq::errKer}
    K(r):=\frac{1}{r}-\mathcal{N}_{b}^{\sigma}(r)-\mathcal{F}_{b}^{\sigma}(r)=\left(\frac{1}{r}-\sum_{\ell=0}^M w_{\ell} e^{-r^2 / s_{\ell}^2}\right) \mathbbm{1}_{r>r_c}.
\end{equation}
By Parseval's theorem, the decomposition errors of energy and force can be written as
%given in Definition~\ref{def::2}
	\begin{equation}\label{eq::Errortol_P}
		U_{\text{err}}=\frac{1}{2}\sum_{i=1}^{N}q_{i}\phi_{\text{err}}(\bm{r}_i)=\frac{1}{2}\mathcal{L}\left[\left|\rho(\bm{k})\right|^2\widehat{K}(k)\right]
	\end{equation}
	and	\begin{equation}\label{eq::Errortol_F}
	    \bm{F}_{\text{err}}(\bm{r}_i)=-q_i \nabla_{\bm{r}_i}\phi(\bm{r}_i)=q_i\mathcal{L}\left[\bm{k}\left|\operatorname{Im}(e^{-i\bm{k}\cdot \bm{r}_i}\rho(\bm{k}))\right|\widehat{K}(k)\right],
	\end{equation}
	respectively, where $k=|\bm{k}|=|(\dot{\bm{k}},k_z)|$, and  
$\rho(\bm{k})=\sum_{j}q_je^{i\bm{k}\cdot\bm{r}_j}$.
We approximate the double summation in $\mathcal{L}$  by integrals~\cite{kolafa1992cutoff}: 	\begin{equation}\label{eq::approx}		\mathcal{L}\left[f(\dot{\bm{k}},k_z)\right]\approx\frac{1}{(2\pi)^3}\int_{0}^{\infty}\int_{0}^{2\pi}\int_{-\infty}^{\infty}\dot{k}f(\dot{\bm{k}},k_z)dk_zd\theta d\dot{k}.
\end{equation}
Note that standard error estimates for the trapezoidal rule~\cite{Trefen2014SIAMRev} can be used to rigorously
control the approximation error in Eq.~\eqref{eq::approx}.
For example, for a band-limited function $f$, the trapezoidal rule converges exponentially fast.
Thus, replacing $\mathcal{L}$ with the triple integral in Eq.~\eqref{eq::approx} does not change the leading order term in the error analysis of $U_{\text{err}}$ and $\bm{F}_{\text{err}}(\bm{r}_i)$~\cite{liang2023error}. Before presenting the main results, we introduce the following assumption on the bandwidths of the Gaussians and other parameters in the SOG decomposition.

\begin{assumption}\label{ass:SOG}
The bandwidth $s_{M+1}$ of the $(M+1)$th Gaussian in the far part of SOG decomposition is assumed to be much larger than the box size, that is, $s_{M+1}\gg \max\{L_x,L_y\}\geq 2r_c$. Here $L\geq 2r_c$ is due to the minimum image convention~\cite{Frenkel2001Understanding}. Moreover, the bandwidth of the zeroth Gaussian is assumed to be less than the cutoff radius, i.e., $s_0<r_c$.
\end{assumption}

Assumption~\ref{ass:SOG} is a natural setup for practical calculations. If $s_{M+1}$ is small enough to be comparable to $\max\{L_x,L_y\}$, then the range of $\mathcal{F}_b^{\sigma}(r)$ is only as large as the box, making it hard to capture long-range effects of $1/r$. Similarly, adjusting the weight of the zeroth Gaussian in BSA for $C^1$ continuity leads to an error of at least $O(1)$ if $s_{0}>r_c$. Under Assumption~\ref{ass:SOG}, it is straightforward to establish Theorem~\ref{thm:SOG}, whose proof is entirely parallel to the one in~\cite{liang2023error} for 3D periodic boundary conditions. 
  
\begin{theorem}\label{thm:SOG}
%In the mean-field region that $\min\{L_x,L_y\}\gg 1$, 
The decomposition error of energy and force for $C^1$ continuity SOG decomposition can be estimated by
% \begin{equation}\label{eq::convrate}
%     \begin{aligned}
%     \left|U_{\emph{err}}\right| & \simeq O\left((\log b)^{-3 / 2} e^{-\pi^2 / 2 \log b}+b^{-M}+w_{-1} e^{-r_c^2 / s_{-1}^2}\right), \\
%     \left|\boldsymbol{F}_{\emph{err}}\left(\boldsymbol{r}_i\right)\right| & \simeq O\left((\log b)^{-3 / 2} e^{-\pi^2 / 2 \log b}+b^{-3 M}+w_{-1}\left(s_{-1}\right)^{-2} e^{-r_c^2 / s_{-1}^2}\right),
%     \end{aligned}
% \end{equation}
\begin{equation}\label{eq::convrate}
	\begin{split}
	|U_{\emph{err}}|&\simeq O\left((\log b)^{-3/2}e^{-\pi^2/2\log b}+b^{-M}+w_{-1}e^{-r_c^2/s_{-1}^2}-(\omega-1)w_{0}e^{-r_c^2/s_0^2}\right),\\[1.em]
	|\bm{F}_{\emph{err}}(\bm{r}_i)|&\simeq O\left((\log b)^{-3/2}e^{-\pi^2/2\log b}+b^{-3M}+\frac{w_{-1}}{(s_{-1})^{2}}e^{-r_c^2/s_{-1}^2}-\frac{(\omega-1)w_{0}}{(s_0)^{2}}e^{-r_c^2/s_0^2}\right),
	\end{split}
\end{equation}
where $\simeq$ represents asymptotically equal as $b\rightarrow 1$.
\end{theorem}

\begin{remark} 
In Eq.~\eqref{eq::convrate}, the first term  corresponds to the ``aliasing error'' (or discretization error) of the trapezoidal rule 
in the SOG approximation. The second term is  the truncation error of the SOG approximation at large distance. In other words, the SOG approximation has to be valid to the distance $r\sim 1/\varepsilon$ for the given tolerance $\varepsilon$ in order to capture the effect of periodic conditions in $x$ and $y$ directions. The last two terms stand for the truncation error in the near-field and the modification of weight in the zeroth Gaussian, respectively.
\end{remark}	

% \begin{proof}
% To derive the estimation of $U_{\text{err}}$ and $\bm{F}_{\text{err}}(\bm{r}_i)$, one can safely approximate the summation functional by a triple integral~\cite{kolafa1992cutoff}:
% 	\begin{equation}\label{eq::approx}
% 		\mathcal{L}\left[f(\dot{\bm{k}},k_z)\right]=\frac{1}{(2\pi)^3}\int_{0}^{\infty}\int_{0}^{2\pi}\int_{-\infty}^{\infty}\dot{k}f(\dot{\bm{k}},k_z)dk_zd\theta d\dot{k}+O(e^{-L}),
% 	\end{equation}
% 	where $\simeq$ represents the asymptotically equal in the mean-field limit. By Eq.~\eqref{eq::approx}, the errors of energy and forces are approximated by
% 	\begin{equation}\label{eq::error}
% 		U_{\text{err}}\simeq\frac{1}{16 \pi^3} \int_{\mathbb{R}^3}|\rho(\boldsymbol{k})|^2 \widehat{K}(k)d\bm{k} \quad \text { and } \quad \boldsymbol{F}_{\text{err}}\left(\boldsymbol{r}_i\right)\simeq\int_{\mathbb{R}^3} \frac{q_i \boldsymbol{k}}{8\pi^3} \operatorname{Im}\left(e^{-i \boldsymbol{k} \cdot \boldsymbol{r}_i} \rho(\boldsymbol{k})\right) \widehat{K}(k)d\bm{k}.
% 	\end{equation}
% The integral form in Eq.~\eqref{eq::error} serves as a foundation for error analysis. By directly extending the framework proposed in \cite{liang2023error} from 3D-PBC to quasi-2D systems, one can deduced that both $U_{\text{err}}$ and $\boldsymbol{F}_{\text{err}}(\boldsymbol{r}_i)$ are uniformly bounded by the RHS of Eq.~\eqref{eq::error}. 
% \end{proof}

We now discuss the parameter selection for the SOG decomposition. Suppose that the error tolerance $\varepsilon$ is given. We first choose $b$ so that the contribution from the first term in the RHS of Eq.~\eqref{eq::convrate} is less than $\varepsilon$. To solve for $r_c$ and $\omega$, we observe the following key properties: 
\begin{equation}\label{eq::2.35z}
\mathcal{F}_{b}^{\sigma}(r_c)=\frac{1}{\sigma}\mathcal{F}_{b}^{1}\left(\frac{r_c}{\sigma}\right),\qquad \frac{d}{dr}\mathcal{F}_{b}^{\sigma}(r){\Big|}_{r=r_c}=\frac{1}{\sigma^2}\frac{d}{dr}\mathcal{F}_{b}^{1}(r){\Big|}_{r=\frac{r_c}{\sigma}},
\end{equation}
so that $r_c$ and $\omega$ are equivalently the roots of $r/\sigma\mathcal{F}_b^{1}(r/\sigma)-1$ and $r^2/\sigma^{2}d\mathcal{F}_b^{1}(r'=r/\sigma)/dr'+1$. Let $r_0=r_c/\sigma$. We solve for $r_0$ and $\omega$ from the equations $r_0\mathcal{F}_{b}^{1}(r_0)-1=0$ and $r_0^2d\mathcal{F}_b^{1}(r)/dr|_{r=r_0}+1=0$. Note that the solution for $r_0$ is not unique, and it holds that $r_c^2/s_0^2 = r_0^2/2$ and $s_{-1} = b s_0$. We select the smallest $r_0$ such that the last two terms in the RHS of Eq.~\eqref{eq::convrate} are less than $\varepsilon$. In practice, we choose a suitable $\sigma$ depending on the best approach for maximize the performance, and then set $r_c=r_0\sigma$. Finally, the second term on the RHS of Eq.~\eqref{eq::convrate} is used to determine the number of Gaussian truncation terms $M$. In Table~\ref{tabl:parameter}, the errors for six parameter sets are presented for ease of use, which is independent of cutoff $r_c$. We observe that $271$ and $116$ Gaussians are needed to achieve $10^{-14}$ accuracy for energy and force, respectively.

%$b\sim e^{-\pi^2/\log(\varepsilon^2)}$ based on the pointwise error Eq.~\eqref{eq::pointwiseerror} as well as the truncation error of erengy and force given in Theorem~\ref{thm:SOG}. To achieve a $C^1$ decomposition, the parameters $\sigma$ and $\omega$ are solved through the continuity equations Eqs.~\eqref{Eq::2.20} and \eqref{eq::2.22}. Finer tuning of $w_{\ell}$ and $s_{\ell}$ is conducted if a higher order of continuity is required. Subsequently, the second component is employed to ascertain the number of truncated Gaussians $M$.

%we can use the first and third components of Eq.~\eqref{eq::convrate} to determine the values of $b$ and $\sigma$. Subsequently, the second component is employed to ascertain the number of truncated Gaussians $M$.

\renewcommand\arraystretch{1.19}
\begin{table}[!ht]
\caption{Typical parameters and corresponding relative errors for the $C^1$ SOG decomposition with different $b$, $r_0$ and $\omega$ in quasi-2D periodic system. $M$ is the minimum number of Gaussian terms required to achieve the error.}
\centering
\scalebox{0.75}{
 \begin{tabular}{c|c|c|cc|cc}
		\hline \multirow{2}{*}{$b$}& \multirow{2}{*}{$r_0$} & 
        \multirow{2}{*}{$\omega$} &
  \multicolumn{2}{c|}{Energy} & \multicolumn{2}{c}{Force} \\
		 \cline{4-7} & & &  Error & $M$ &  Error& $M$ \\
		\hline 2 & $1.9892536839080267$ & $0.9944464927622323$&$3.12\times 10^{-2}$ & $16$ & $9.93\times 10^{-3}$ & $11$ \\
		\hline $1.62976708826776469$ & $2.7520026668023417$ &$1.0078069793438068$ & $2.33\times 10^{-3}$ & $31$ & $6.21\times 10^{-4}$ & $16$   \\
		\hline $1.48783512395703226$ & $3.7554672283554990$ &$0.9919117057598183$ & $2.29\times 10^{-4}$ & $46$ & $7.98\times 10^{-5}$ & $26$   \\
		\hline $1.32070036405934420$ & $4.3914554711638349$ &$1.0018891411481198$& $1.18\times 10^{-6}$ & $76$ & $5.76 \times 10^{-7}$ & $41$   \\
		\hline $1.21812525709410644$ & $5.6355288151271085$ &$1.0009014615603334$ &  $7.14\times 10^{-10}$ & $166$ & $5.14\times 10^{-10}$ & $71$   \\
        \hline $1.14878150173321925$ & $7.2956245490719404$ &$1.0000368348358225$& $1.30\times 10^{-15}$ & $271$ & $1.98\times 10^{-14}$ & $116$   \\
\hline
\end{tabular}
}
\label{tabl:parameter}
\end{table}

The error estimates in Theorem~\ref{thm:SOG}
are fairly accurate for a wide range of $b$.
As observed in the parameter selection process, the dominant term in $U_{\text{err}}$ and $\bm{F}_{\text{err}}(\bm{r}_i)$ is $b^{-M}$ and $b^{-3M}$, respectively. To numerically verify this, we perform numerical calculations on a cubic box with side length $1$ ($r_c=0.3$), with $100$ particles randomly distributed within the box. The reference solution is obtained via the Ewald2D summation~\cite{parry1975electrostatic}. Figure~\ref{fig:trunFour} shows close  agreement between the error estimate Eq.~\eqref{eq::convrate} and the actual errors.
	
\begin{figure}[!ht] 
    \centering
    \includegraphics[width=1\textwidth]{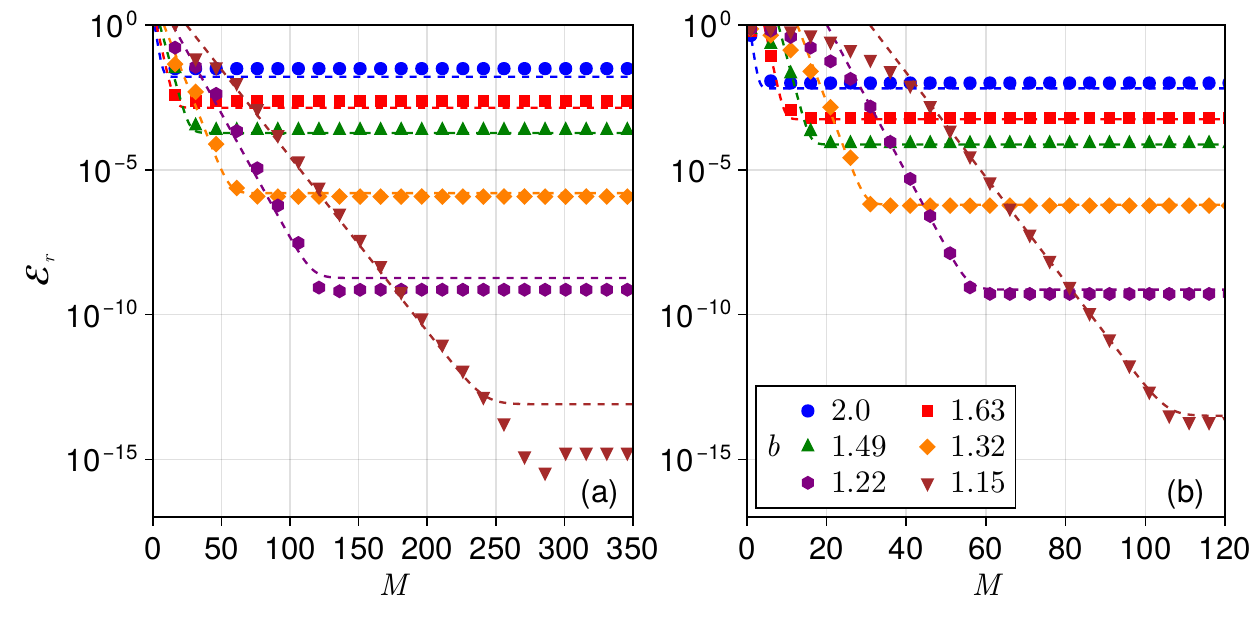}
    \caption{Relative error in (a) energy and (b) force due to the SOG decomposition, plotted against the number of truncated terms, $M$. 
    The data are presented for various $b$, with other associated parameters detailed in Table~\ref{tabl:parameter}. The dotted lines represent theoretical estimates as provided in Theorem~\ref{thm:SOG}.}
    \label{fig:trunFour}
\end{figure}
%Note that a large $M$ has only minor effects on the algorithm efficiency in each components by using advanced implementation techniques described in section 3.4.

%Note that a large $M$ has only minor effects on the algorithm efficiency in both real space and Fourier space by using advanced implementation techniques described in section 3

%we can use the first and third components of Eq.~\eqref{eq::convrate} to determine the values of $b$ and $\sigma$. Subsequently, the second component is employed to ascertain the number of truncated Gaussians $M$. 

%It worth noting that $\varepsilon$ are typically chosen as $10^{-3}$ and $10^{-4}$ in practical simulations for efficiency. 

\begin{remark}
The optimal parameter set depends on the specific application. In MD simulations, prioritizing force accuracy is crucial, whereas MC simulations require energy accuracy.
\end{remark}

%Detailed parameter sets for SOG decomposition are provided in Table \ref{tabl:parameter}. It is worth mentioning that approximately $150$ and $50$ Gaussians are typically required to achieve machine error for the energy and force, respectively.

\section{A fast spectral sum-of-Gaussians method}
% \label{sec::FSSOG}
% In this section, we introduce the concept of range splitting to decompose the far-field potential $\Phi^{\mathcal{F}}$ into summation problems involving Gaussian kernels of different variances. We then apply this idea to develop a fast spectral SOG method that effectively handles each component in $O(N\log N)$ or $O(N)$ operations. 
% \section{A two-level scheme for the far part}
\label{sec::FSSOG}
In this section, we present a simple yet efficient two-level scheme to evaluate the far part
$\Phi^{\mathcal{F}}$ whose interaction kernels consists of a collection of Gaussians.
Since Gaussians are completely separable
in Fourier space, it is tempting to use 
NUFFTs to compute the far part. However,
the Gaussian kernels in the far part have 
a very wide range of variances. As the $z$ direction is free and one needs to discretize the Fourier transform, it is clear that the first Gaussian with the smallest 
variance $s_0$ determines the cutoff value $K_{\text{max}}$ in Fourier space, and the last Gaussian with the largest variance $s_M$ determines the step size $h$ in Fourier space. Thus, treating all these Gaussians together will lead to exceedingly
large number of Fourier modes. This phenomenon is similar to the Fourier spectral analysis of the heat kernel on a large time interval in \cite{greengard2000acha}. However, the fix is rather straightforward. It is well known that Gaussians of very large variances are well approximated by polynomials on a fixed interval. This fact has been used in many 
fast Gauss transforms (see, for example, 
\cite{jiang2024sirev}). Thus, we split these Gaussians into two groups - a mid-range group efficiently approximated by Fourier 
series, and a long-range group efficiently approximated by polynomials along the $z$ 
direction. 
The mid-range group can be handled via a method similar to the NUFFT directly, where the upsampling in standard NUFFTs is not required due to the fast decay of the Gaussian kernels. For the long-range group we apply polynomial interpolation and anterpolation along the $z$ direction, then the Fourier-based method in $x$ and $y$ direction. Once again, the upsampling is unnecessary.

%We remark that this treatment for long-range group is similar to the one developed in \cite{jiang2023jcp}, though once again the upsampling in standard NUFFTs is not required due to the fast decay of the Gaussian kernels. \tcr{here}

%\tcr{Perhaps we should avoid discussing [48] here, as it is specifically developed for 2D fully-periodic systems. I am concerned that including it might lead reviewers to undervalue the novelty of this work.}

\subsection{Range splitting of the far-field potential}\label{subsec::rangesplitting}
    	
%Recent works~\cite{shamshirgar2021fast,bagge2022fast} argue that the integral form Eq.~\eqref{eq::PhiInt} is preferable due to its ability to construct a fast method in a similar manner of the fully-periodic SE method~\cite{lindbo2011spectral}. 
	
%As has been explored in earlier SE methods, it is possible to develop a fast method by employing techniques such as periodization~\cite{lindbo2012fast}, truncation~\cite{minary2002new}, or regularization~\cite{nestler2015fast} to Eq.~\eqref{eq::PhiClosed}. 
    
We first present the error estimate
on the polynomial approximation to Gaussians
of large variances.

\begin{theorem}\label{thm::Che}
Suppose that $s_{\ell}\geq \eta L_z$ with $\eta>0$. Then the Chebyshev interpolant of the Gaussian function $e^{-z^2/s_{\ell}^2}$ on the interval $[-L_z/2,L_z/2]$ satisfies the error estimate
	\begin{equation}\label{eq::gaussianapprox}
		\left|e^{-z^2/s_{\ell}^2}-\sum_{n=0}^{P-1}{}^{\prime}a_nT_{n}\left(\frac{2z}{L_z}\right)\right|\leq \frac{1}{\sqrt{P!}(2\sqrt{2}\eta)^{P}},
	\end{equation}
	where $T_{n}(z)$ represents Chebyshev polynomial of degree $n$ on $[-1,1]$, and the prime indicates that there is a factor of $1/2$ in front of $a_0$.
\end{theorem}
\begin{proof}
	Approximating a smooth function $f(x)$ on the interval $[a,b]$ by the Chebyshev interpolation, has the standard error estimate~\cite{Trefethen2019App}
	\begin{equation}\label{eq::Taylor3}
	\left|f(x)-\sum_{n=0}^{P-1}{}^{\prime}a_nT_{n}\left(\frac{2x-(a+b)}{b-a}\right)\right| \leq \frac{1}{2^{P} P !}\left(\frac{b-a}{2}\right)^{P}\left\|f^{(P)}\right\|_\infty.
	\end{equation}
    Note that the Gaussian function $f(x)=e^{-x^2/s_{\ell}^2}$ satisfies \cite{Jack1961Inequality}
    \begin{equation}\label{eq::2.38}
    	\|f^{(P)}\|_\infty\leq \sqrt{2^P P!}s_\ell^{-P}. 
    \end{equation}
    Substituting $a=-L_z/2$, $b=L_z/2$ and Eq.~\eqref{eq::2.38} into Eq.~\eqref{eq::Taylor3}, we obtain Eq.~\eqref{eq::gaussianapprox}.
\end{proof}
Theorem~\ref{thm::Che} demonstrates the rapid convergence of Chebyshev expansion for Gaussians with $s_{\ell}\sim O(L_z)$. For instance, when $\eta=1$, $P=6$, $10$, and $17$ yields approximately $4$, $8$, and $15$ decimal digits of accuracy, respectively. For $\eta=1/2$, the values of $P$ should be increased to $8$, $14$, and $23$ to attain the same level of accuracy. The aforementioned observations naturally suggest the idea of dividing $\mathcal{F}_{b}^{\sigma}(r)$, the far part Gaussians in u-series decomposition, into two categories, which we refer to as ``Range Splitting''. This can be precisely described through the following definition. %where it is worth noting that the bandwidth $s_{\ell}=\sqrt{2}b^{\ell}\sigma$ in u-series consistently increases with increasing $\ell$, given that $b>1$.

\begin{definition}
(Range Splitting) Let $\eta=O(1)$ be a positive parameter and let the Gaussian bandwidths satisfy $s_{0}<s_{1}<\cdots<s_{M}$. A Gaussian function $e^{-r^2/s_{\ell}^2}$ is categorized as ``long-range'' if $s_{\ell}>\eta L_z$, and as ``mid-range'' otherwise. $\eta$ is referred to as the factor of range splitting. 
\end{definition}

Using the range splitting, the Gaussians in the far field kernel of u-series decomposition, $\mathcal{F}_{b}^{\sigma}(r)$, are further divided into mid- and long-range based on their indices. The mid-range part consists of Gaussians with indices satisfying $0\leq\ell\leq m$, where $s_{m}\leq \eta L_z<s_{m+1}$. The long-range part consists of the remaining $M-m$ Gaussians with large variance, i.e., those with indices satisfying $\ell> m$. Consequently, the Coulomb kernel $1/r$ is decomposed as in Eq.~\eqref{eq::SOGdecomp}, where we recall the near-field kernel is simply the $1/r$ minus $\mathcal{F}_{b}^{\sigma}(r)$. The screening function can be split accordingly as $\tau_{\text{SOG}}=\tau_{\text{SOG}}^{\rm mid}+\tau_{\text{SOG}}^{\rm long}$, and the far-field potential can be expressed as  
\begin{equation}\label{eq::1.32}
\Phi^{\mathcal{F}}=\Phi^{\rm mid}_{\text{SOG}}+\Phi^{\rm long}_{\text{SOG}},
\end{equation}
corresponding to the solution obtained by substituting the Fourier transform of $\tau_{\text{SOG}}^{\rm mid}$ and $\tau_{\text{SOG}}^{\rm long}$ into Eq.~\eqref{eq::phindf}, respectively. In Eq.~\eqref{eq::1.32}, $\Phi^{\rm mid}_{\text{SOG}}$ can be accurately periodicized in the $z$-direction and effectively handled using a pure Fourier spectral solver. In contrast, the long-range component $\Phi^{\rm long}_{\text{SOG}}$ requires a large zero-padding factor if the same Fourier solver is applied. However, it is extremely smooth over $[-L_z/2,L_z/2]$, allowing for fast computation using a novel Fourier-Chebyshev spectral solver. These two solvers are described in Sections~\ref{sec::SSOG} and \ref{sec::Chebyshev}, respectively.

\begin{remark}
If $\eta L_z=o(1)$, such that $s_{\ell}>\eta L_z$ for all $\ell=0,\cdots,M$, all the Gaussians are categorized as long-range in the range splitting. Consequently, the entire far-field potential $\Phi^{\mathcal{F}}$ is handled by a single Fourier-Chebyshev spectral solver, and the mid-range solver is not invoked. 
\end{remark}

%the mid-ranged part consists of Gaussians with indices satisfying $\ell\leq \rm mid$, where $s_{\rm mid}\leq \eta L_z\leq s_{\rm mid+1}$, while the long-ranged part consists $\rm long=M-\rm mid$ Gaussians with large variance that do not meet this condition ($\ell>\rm mid$). As a result, the screening function are split by $\tau_{\text{SOG}}=\tau_{\text{SOG}}^{\rm mid}+\tau_{\text{SOG}}^{\rm long}$, and the far-field potential
% \begin{equation}\label{eq::1.32}
% \Phi^{\mathcal{F}}=\sum_{\ell=0}^{m}\Phi_\ell^{\mathcal{F}}+\sum_{\ell=m+1}^{M}\Phi_\ell^{\mathcal{F}}:=\Phi^{\rm mid}_{\text{SOG}}+\Phi^{\rm long}_{\text{SOG}},
% \end{equation}
% are written as the sum of mid-range and long-range parts accordingly. In Eq.~\eqref{eq::1.32}, $\Phi^{\rm mid}_{\text{SOG}}$ can be accurately periodicized in the $z$-direction and is effectively handled using a pure Fourier spectral solver~\cite{lindbo2011spectral}. In contrast, the long-range component $\Phi^{\rm long}_{\text{SOG}}$ requires a large zero-padding factor whereas being extremely smooth in real space, allowing for fast computation using a novel Fourier-Chebyshev spectral solver. These two solvers shall be described in Sections~\ref{sec::SSOG} and \ref{sec::Chebyshev}, respectively. 

\subsection{A Fourier spectral solver for the mid-range part}
	\label{sec::SSOG}
We employ a Fourier spectral solver to compute the mid-range component of potential, $\Phi^{\rm mid}_{\text{SOG}}$. This solver is formulated closely after the recent SE method~\cite{shamshirgar2021fast}. However, there are two notable differences: we replace the Ewald decomposition with the SOG decomposition described in Section~\ref{subsec::SOGdecomp}, and we do not need TKM to remove the singularity of the kernel at the origin.  
	
Let us start by introducing a window function $W(\bm{r})$ along with its Fourier transform $\widehat{W}(\bm{k})$ for smearing particles onto a uniform grid and collecting field values from these grids.  By inserting the identity $1\equiv \widehat{W}(\bm{k})\left|\widehat{W}(\bm{k})\right|^{-2}\widehat{W}(\bm{k})$ into the expression of $\Phi^{\rm mid}_{\text{SOG}}$, we obtain
\begin{equation}
\label{eq::PhiNDF}
\Phi^{\rm mid}_{\text{SOG}}(\bm{r}_i)=4\pi\mathcal{L}\left[\widehat{W}(\bm{k})e^{\mathrm{i}\bm{k}\cdot\bm{r}_i}\dfrac{\widehat{\tau}_{\text{SOG}}^{\rm mid}(\bm{k})}{|\bm{k}|^2}\left|\widehat{W}(\bm{k})\right|^{-2}\sum_{j=1}^{N}q_j\widehat{W}(\bm{k})e^{-\mathrm{i}\bm{k}\cdot\bm{r}_{j}}\right].
\end{equation}
We will use the last $\widehat{W}(\bm{k})$ for the gridding step, the first $\widehat{W}(\bm{k})$ for the gathering step, and $\left|\widehat{W}(\bm{k})\right|^{-2}$ for the scaling step, respectively. The details are provided as follows.
	
	Let us define
	\begin{equation}
        \label{eq::grid}
		\widehat{S}_{\text{grid}}(\bm{k}):=\sum_{j=1}^{N}q_j\widehat{W}(\bm{k})e^{-\mathrm{i}\bm{k}\cdot\bm{r}_{j}},
	\end{equation}
	which is the Fourier transform of
	\begin{equation}\label{eq::gridding}
	S_{\text{grid}}(\bm{r})=\sum_{j=1}^{N}q_jW(\bm{r}-\bm{r}_j)_{*}
	\end{equation}
	by the convolution theorem (see Appendix~\ref{subsec::conv}). Here, $(\cdot)_*$ denotes that periodicity is implied in the periodic directions. The evaluation of Eq.~\eqref{eq::gridding} on a uniform grid will become the gridding step in this solver, where the grid discretization will be discussed in Section~\ref{subsec::grid}. After this step, we obtain $\widehat{S}_{\text{grid}}(\bm{k})$ from $S_{\text{grid}}(\bm{r})$ via a forward 3D FFT. Next, we compute 
	\begin{equation}\label{eq::widehat-tilde}
		\widehat{S}_{\text{scal}}(\bm{k}):=\dfrac{\widehat{\tau}_{\text{SOG}}^{\rm mid}(\bm{k})}{k^2}\left|\widehat{W}(\bm{k})\right|^{-2}\widehat{S}_{\text{grid}}(\bm{k}),
	\end{equation}
	which will become the scaling step. By Eq.~\eqref{eq::widehat-tilde}, we can reformulate Eq.~\eqref{eq::PhiNDF} as 
	\begin{equation}\label{eq::eq33}
		\Phi^{\rm mid}_{\text{SOG}}(\bm{r}_i)=4\pi\mathcal{L}\left[\widehat{S}_{\text{scal}}(\bm{k})\widehat{W}(\bm{k})e^{\mathrm{i}\bm{k}\cdot\bm{r}_i}\right].
	\end{equation}
Applying the Plancherel's theorem and the convolution theorem (see Appendices~\ref{subsec::conv} and \ref{subsec::Plancherel}) to Eq.~\eqref{eq::eq33}, we obtain an equivalent integral form
\begin{equation}\label{eq::Phiconv}
	\Phi^{\rm mid}_{\text{SOG}}(\bm{r}_i)=4\pi\int_{\mathbb{R}}\int_{\mathcal{R}^2}S_{\text{scal}}([\dot{\bm{r}}_j,z_j])W(\dot{\bm{r}}_i-\dot{\bm{r}}_j)_{*}W(z_i-z_j)d\dot{\bm{r}}_jdz_j.
\end{equation}
We apply a backward 3D FFT to obtain $S_{\text{scal}}(\bm{r})$ from $\widehat{S}_{\text{scal}}(\bm{k})$ at each grid point. Once Eq.~\eqref{eq::Phiconv} is discretized on the grid by the trapezoidal rule, it will become the gathering step. Note that if the window function $W$ is smooth and has compact support, this step has spectral convergence with the number of grids. The procedure of this Fourier spectral solver is summarized in Algorithm~\ref{al::mid}. 
	
In the aforementioned procedure, we employ a pair of forward and backward 3D FFTs. These transforms are hybrid: a discrete transform is applied in each periodic direction, and an approximation to the continuous Fourier integral is performed in the free direction. The latter is equivalent to periodicizing the potential. Since the narrowest Gaussian necessitates the finest grid, we adopt a uniform grid size akin to that employed for the $\ell=0$ Gaussian across all mid-range Gaussians. This allows the precomputation of $\widehat{\tau}^{\rm mid}_{\text{SOG}}(\bm{k})$ so that the complexity is independent of the number of mid-range Gaussians $m$.

%To make the cost independent of $\rm mid$, the number of mid-range Gaussians, we use a uniform grid size as that for the $\ell=0$ Gaussian, as the narrowest Gaussian requires finest grid. This allows precomputation of $\widehat{\tau}^{\rm mid}_{\text{SOG}}(\bm{k})$ before evaluating Eq.~\eqref{eq::widehat-tilde}.

%The solver described above can be recast as an application of the non-uniform FFT (NUFFT)~\cite{dutt1993fast}. 

\begin{remark}
\label{rmk::PWindow}
Commonly used  window functions include the Gaussian window, the Kaiser-Bessel (KB) window~\cite{kaiser1980use}, and the ``exponential of semicircle'' (ES) window~\cite{barnett2019parallel}. The choice of window function influences both the accuracy and performance of the proposed solver. In our code, we designed a module that allows for adaptive constructions of polynomial approximations for these window functions using Chebyshev interpolation. %Results shown in Fig.~\ref{fig:window} show that both KB and ES windows exhibit promising efficiency.
\end{remark}
\begin{remark}
Compared to the particle-mesh Ewald (PME) method~\cite{Darden1993JCP}, our Fourier spectral solver inherits a key advantage from previous SE methods: the support of window functions can be varied independently of the grid size. This allows for separate control of discretization errors and window truncation errors. Additionally, unlike standard NUFFT-based solvers~\cite{nestler2015fast}, our method does not require any upsampling. These arguments will be proved in Section~\ref{subsec::analysis3DFF}.
\end{remark}

\begin{algorithm}[H]
\caption{~Fourier spectral solver for the mid-range potential $\Phi^{\rm mid}_{\text{SOG}}$}\label{al::mid}
\begin{algorithmic}[1]
\State (Gridding) Smear discrete particles onto uniform grids using Eq.~\eqref{eq::gridding}.
\State (Real to Fourier) Obtain $\widehat{S}_{\text{grid}}(\bm{k})$ from $S_{\text{grid}}(\bm{r})$ using a forward 3D FFT.
\State (Scaling) For each Fourier mode $\bm{k}$, obtain $\widehat{S}_{\text{scal}}(\bm{k})$ by a scaling as in Eq.~\eqref{eq::widehat-tilde}.
\State (Fourier to Real) Obtain $S_{\text{scal}}(\bm{r})$  from $\widehat{S}_{\text{scal}}(\bm{k})$ by a backward 3D FFT.
\State (Gathering) Compute $\Phi^{\rm mid}_{\text{SOG}}(\bm{r}_i)$ by gathering contributions from grids using Eq.~\eqref{eq::Phiconv}.
\end{algorithmic}
\end{algorithm}	

\subsection{Fourier-Chebyshev spectral solver for long-range part}
	\label{sec::Chebyshev}
In this section, we introduce an SOG-based Fourier-Chebyshev spectral solver specifically designed for handling the long-range component of potential, $\Phi^{\rm long}_{\text{SOG}}$. For clarity, we introduce some basic notations similar to those in Definition~\ref{def::2}, but using the Chebyshev expansion~(see Appendix~\ref{subsec::ChebyshevExpansion}) in the free direction.
\begin{definition}
(Fourier-Chebyshev transform) Let $f([\dot{\bm{r}},z])$ be a function that is periodic in $\dot{\bm{r}}$ and non-periodic in $z\in[-L_z/2,L_z/2]$. Its forward Fourier-Chebyshev transform is defined as:
\begin{equation}\label{eq::FCtransform}
\widetilde{f}([\dot{\bm{k}},n]):=\int_{\mathcal{R}^2}\int_{-\frac{L_z}{2}}^{\frac{L_z}{2}}\frac{f([\dot{\bm{r}},z])e^{-i\dot{\bm{k}}\cdot\dot{\bm{r}}}}{\sqrt{1-(2z/L_z)^2}}T_{n}\left(\frac{2z}{L_z}\right)dzd\dot{\bm{r}},
\end{equation}
where $T_n(\cdot)$ is the Chebyshev polynomial of degree $n$ and $\widetilde{f}([\dot{\bm{k}},n])$ is referred to as the Fourier-Chebyshev coefficient. The backward Fourier-Chebyshev transform is defined as:
\begin{equation}\label{eq::InvFCtransform}
f([\dot{\bm{r}},z])=\frac{2\pi}{L_xL_yL_z}\sum_{\dot{\bm{k}}\in\mathcal{K}^2}\sum_{n=0}^{\infty}{}^{\prime}\widetilde{f}([\dot{\bm{k}},n])e^{i\dot{\bm{k}}\cdot\dot{\bm{r}}}T_{n}\left(\frac{2z}{L_z}\right).
\end{equation}
\end{definition}
\noindent Once Eqs.~\eqref{eq::FCtransform} and \eqref{eq::InvFCtransform} are discretized and truncated, the calculations along periodic dimensions can be accelerated via the FFT. %and the calculations along periodic and free dimensions can be accelerated via the FFT and the fast cosine transform, respectively.

%In simple terms, our method centers on a crucial idea: restricting the use of window functions $W(\dot{\bm{r}})$ to periodic dimensions for efficient grid sampling, and applying Chebyshev interpolation and evaluation (see Appendix~\ref{subsec::ChebyshevInterpolation}) in the free dimension. 

Consider the closed form of $\Phi^{\rm long}_{\text{SOG}}$, which is similar to Eq.~\eqref{eq::phindf} except that only those Gaussians with $\ell> m$ are included and the integral along the $z$-direction is explicitly carried out. By introducing a window function $W(\dot{\bm{r}})$ with its Fourier transform $\widehat{W}(\dot{\bm{k}})$ and inserting the identity $1\equiv \widehat{W}(\dot{\bm{k}})|\widehat{W}(\dot{\bm{k}})|^{-2}\widehat{W}(\dot{\bm{k}})$, we obtain
\begin{equation}\label{eq::phi2M}
\Phi_{\text{SOG}}^{\rm long}([\dot{\bm{r}}_i,z_i])=\frac{\pi}{L_xL_y}\sum_{\dot{\bm{k}}\in\mathcal{K}^2}W(\dot{\bm{k}})e^{i\dot{\bm{k}}\cdot\dot{\bm{r}}_i}\left|W(\dot{\bm{k}})\right|^{-2}\sum_{\ell=m+1}^{M}w_{\ell}s_{\ell}^2e^{-s_{\ell}^2\dot{k}^2/4}\widehat{S}^{\ell}_{\text{grid}}([\dot{\bm{k}},z_i]),
\end{equation}
where 
\begin{equation}
\widehat{S}^{\ell}_{\text{grid}}([\dot{\bm{k}},z]):=\sum_{j=1}^{N}q_jW(\dot{\bm{k}})e^{-i\dot{\bm{k}}\cdot\dot{\bm{r}}_j}e^{-(z-z_j)^2/s_{\ell}^2}
\end{equation}
is the 2D Fourier transform of $S^{\ell}_{\text{grid}}$. Due to Theorem~\ref{thm::Che}, a key observation is that $\widehat{S}^{\ell}_{\text{grid}}([\dot{\bm{k}},z])$ can be well approximated by Chebyshev series on $z\in[-L_z/2,L_z/2]$:
\begin{equation}
\widehat{S}^{\ell}_{\text{grid}}([\dot{\bm{k}},z])=\sum_{n=0}^{\infty}{}^{\prime}\widetilde{S}^{\ell}_{\text{grid}}([\dot{\bm{k}},n])T_n\left(\frac{2z}{L_z}\right).
\end{equation}
Our solver centers on a fundamental idea for efficient gridding/gathering: restricting the use of window functions $W(\dot{\bm{r}})$ to periodic dimensions, and applying Chebyshev interpolation and evaluation (see Appendix~\ref{subsec::ChebyshevInterpolation}) in the free dimension. The details are provided as follows.

To begin with, we sample 
\begin{equation}\label{eq::Chebygrid}
S^{\ell}_{\text{grid}}([\dot{\bm{r}},z])=\sum_{j=1}^{N}q_jW(\dot{\bm{r}}-\dot{\bm{r}}_j)_{*}e^{-(z-z_j)^2/s_{\ell}^2}
\end{equation}
on tensor product proxy points $\mathscr{S}_{\text{R-C}}\subset\left\{(\dot{\bm{r}},z)\in\mathscr{R}^2\times \mathscr{C}\right\}$, where $\mathscr{R}^2$ and $\mathscr{C}$ denote the set of uniform nodes on periodic domain and Chebyshev nodes scaled on $[-L_z/2,L_z/2]$, respectively. This will become the gridding step. After this step, we obtain $\widetilde{S}^{\ell}_{\text{grid}}$ from $S^{\ell}_{\text{grid}}$ by a forward Fourier-Chebyshev transform as in Eq.~\eqref{eq::FCtransform}. During this calculation, the forward 2D FFT and Chebyshev interpolation will be used to get Fourier-Chebyshev coefficients from proxy points. Next, we directly evaluate
\begin{equation}\label{eq::overlineSscal}
\widetilde{S}_{\text{scal}}([\dot{\bm{k}},n]):=\left|\widehat{W}(\dot{\bm{k}})\right|^{-2}\sum_{\ell=m+1}^{M}w_{\ell}s_{\ell}^2e^{-s_{\ell}^2\dot{k}^2/4}\widetilde{S}_{\text{grid}}^{\ell}([\dot{\bm{k}},n])
\end{equation}
which will become the scaling step. The procedures outlined above, up to Eq.~\eqref{eq::overlineSscal}, can be further accelerated via techniques described in Appendix~\ref{subsec::Gridding} so that the complexity is independent of both $m$ and $M$. Finally, using Eq.~\eqref{eq::overlineSscal}, we rewrite Eq.~\eqref{eq::phi2M} in the form of backward Fourier-Chebyshev transform:
\begin{equation}
\Phi_{\text{SOG}}^{\rm long}([\dot{\bm{r}}_i,z_i])= \frac{2\pi}{V}\sum_{\dot{\bm{k}}\in\mathcal{K}^2}\sum_{n=0}^{\infty}{}^{\prime}\widetilde{S}_{\text{scal}}([\dot{\bm{k}},n])\widehat{W}(\dot{\bm{k}})e^{i\dot{\bm{k}}\cdot\dot{\bm{r}}_{i}}T_n\left(\frac{2z_i}{L_z}\right),
\end{equation}
where $V=L_xL_yL_z$ represents the volume of $\Omega$. By applying Plancherel's theorem and the convolution theorem in periodic dimensions, we arrive at an equivalent representation
\begin{equation}\label{eq::SOGLong}
\Phi_{\text{SOG}}^{\rm long}([\dot{\bm{r}}_i,z_i])=\frac{2\pi}{L_z}\sum_{n=0}^{\infty}{}^{\prime}\int_{\mathcal{R}^2}S_{\text{scal}}([\dot{\bm{r}},n])W(\dot{\bm{r}}_i-\dot{\bm{r}})_{*}d\dot{\bm{r}}T_n\left(\frac{2z_{i}}{L_z}\right).
\end{equation}
Once Eq.~\eqref{eq::SOGLong} is discretized using the trapezoidal rule, it becomes the gathering step. Here, we conduct backward 2D FFTs to obtain $S_{\text{scal}}$ from $\widehat{S}_{\text{scal}}$, followed by using window interpolation and Chebyshev evaluation in periodic and non-periodic dimensions, respectively, to recover the potential from proxy points to real charges. The process of this Fourier-Chebyshev spectral solver is outlined in Algorithm~\ref{al::long}.

In the described solver, the processes of obtaining $\widetilde{S}^{\ell}_{\text{grid}}$ and $S_{\text{scal}}$ are accelerated using FFTs. However, as discussed in Section~\ref{subsec::complexity}, this solver does not rely on the FFT to achieve its asymptotic complexity when $L_z$ is not significantly smaller than $L_x$ or $L_y$. In such scenarios, the solver can be significantly simplified, and the resulting version is outlined in Appendix~\ref{appendix::simplified}. Note that in Step $4$ of Algorithm~\ref{al::long}, we only require a ``Fourier to Real'' rather than a full ``Fourier-Chebyshev to Real'' procedure. This is because the
Chebyshev coefficients are already obtained through Steps $2-3$, allowing the Chebyshev expansion along $z$-direction to be evaluated directly, without the need of computation on the Chebyshev proxy points.

\begin{remark}
Maxian \emph{et al.}~\emph{\cite{maxian2021fast}} previously investigated Ewald-based Fourier-Chebyshev solvers. In comparison, the solver introduced in this paper offers several advantages: it eliminates the need to solve algebraic equations for individual Fourier modes~\emph{\cite{greengard1991spectral}}, and reduces grid sampling costs through the utilization of window functions.
\end{remark}

\begin{remark}
The calculation of Chebyshev coefficients can also benefit from acceleration using the FFT. The main idea is to transform the Chebyshev nodes into uniform nodes along the argument direction, a technique commonly known as the fast Fourier-Chebyshev transform (FFCT)~\emph{\cite{Trefethen2000}}.
\end{remark}

\begin{algorithm}[H]
	\caption{Fast spectral solver for the long-range potential $\Phi_{\text{SOG}}^{\rm long}$ (with FFT)}\label{al::long}
	\begin{algorithmic}[1]
	\State (Gridding) Sample $S^{\ell}_{\text{grid}}([\dot{\bm{r}},z])$ on tensor product proxy points for each $\ell$.
	\State (Real to Fourier-Chebyshev) Obtain $\widetilde{S}^{\ell}_{\text{grid}}([\dot{\bm{k}},n])$ from $S^{\ell}_{\text{grid}}([\dot{\bm{r}},z])$ by applying forward 2D FFTs and Chebyshev interpolation in periodic and free dimensions, respectively.
 \State (Scaling) For each Fourier mode $\dot{\bm{k}}$ and Chebyshev order $n$, evaluate $\widetilde{S}_{\text{scal}}([\dot{\bm{k}},n])$ by using Eq.~\eqref{eq::overlineSscal}.
\State (Fourier to Real) Obtain $S_{\text{scal}}([\dot{\bm{r}},n])$ from $\widetilde{S}_{\text{scal}}([\dot{\bm{k}},n])$ via backward 2D FFTs.
\State (Gathering) Recover $\Phi^{\rm long}_{\text{SOG}}([\dot{\bm{r}}_i,z_i])$ by gathering contributions from $S_{\text{scal}}([\dot{\bm{r}},n])$.
\end{algorithmic}
\end{algorithm}

\subsection{Grid discretization}\label{subsec::grid}
In the gridding step of  Algorithm~\ref{al::mid}, point sources must be distributed onto a uniform grid so that the FFT is applicable. We establish a uniform Cartesian grid within the box $\Omega$, divided into $I_{d}$ subintervals along each direction $d\in\{x,y,z\}$, with a mesh size defined as $h_{d}=L_d/I_d$. The \emph{half support} of the window function spans a length $H_d=(\mathcal{P}_d-1) h_d/2$ in each direction, where $\mathcal{P}_d$ represents an odd number of grid points within the window's support in each direction.  

Similar to previous SE methods~\cite{lindbo2012fast,shamshirgar2021fast}, the box $\Omega$ needs to be extended to a length
\begin{equation}
L_z^*=L_z+\delta_{z} ,\quad\text{with}\quad \delta_{z}\geq 2H_z,
\end{equation}
in the free direction to accommodate the support of the window function, particularly when some particles are situated near $\pm L_z/2$. As observed in~\cite{shamshirgar2021fast}, selecting $\delta_{z}=2H_z$ suffices for the Gaussian window. However, for the KB window, this factor should be increased to $\delta_{z}=2.6H_z$ due to the slow decay of the screening function towards the box boundaries.

Without loss of generality, we assume $L_z^*=\lambda_z (L_z+\delta_{z})$, where $\lambda_z\geq 1$ represents the ``zero-padding'' factor, consistent with \cite{shamshirgar2021fast}. The uniform grid must be also be extended to size $I^*_z=\lceil\lambda_z (I_z+\delta_{I_z})\rceil$, where $\delta_{I_z}=\delta_z/h_z$ is an integer. To ensure that $I^*_z$ becomes an even integer so that the FFT is beneficial, first set $I^*_z=2\lceil \lambda_z(I_z+\delta_{I_z})/2\rceil$ and then set $L_z^*=h_zI_z^*$. This defines the extended box \begin{equation}
\Omega^*:=\mathcal{R}^2\times \left[-\frac{L_z}{2}-\frac{\delta_z^*}{2},~\frac{L_z}{2}+\frac{\delta_z^*}{2}\right]
\end{equation}
where $\delta_z^*=L_z^*-L_z$. The optimal selection of $\lambda_z=L_z^*/(L_z+\delta_z)$ will be established through analysis and experiments in Sections~\ref{sec::erroranalysis} and \ref{sec::numer}, respectively.

Furthermore, spatial discretization along periodic directions is also necessary in Algorithm~\ref{al::long}. Similarly, we can introduce a uniform Cartesian grid on $\mathcal{R}^2$ with $I_x^{\rm long}$ and $I_y^{\rm long}$ subintervals in each direction, defining the mesh size as $h_x^{\rm long}=L_x/I_x^{\rm long}$ and $h_y^{\rm long}=L_y/I_{y}^{\rm long}$, respectively. 

\begin{remark}\label{rmk::meshsize}
It has been demonstrated that the mesh size required for a Gaussian $e^{-r^2/s_{\ell}^2}$ is on the order of $O(s_{\ell})$~\emph{\cite{shan2005gaussian}}. This indicates that we should set $h_x$, $h_y$, and $h_z$ as $O(s_{0})$ and set $h_x^{\rm long}$ and $h_y^{\rm long}$ as $O(s_{m+1})$. Since $s_{m+1}=b^{m+1}s_{0}$, the grid number required by the solver of long-range Gaussians is much smaller than that of mid-range Gaussians.    
\end{remark}

\subsection{Complexity analysis}\label{subsec::complexity}
Let us analyze the complexity of the proposed solver, which consists of four distinct components:
\begin{equation}\label{eq::splitting}
\Phi(\bm{r}_i)\approx \Phi_{\text{SOG}}(\bm{r}_i)=\Phi^{\mathcal{N}}(\bm{r}_i)+\Phi^{\rm mid}_{\text{SOG}}(\bm{r}_i)+\Phi^{\rm long}_{\text{SOG}}(\bm{r}_i)+\Phi^{\text{self}}_i.
\end{equation}
The computation of near-field term $\Phi^{\mathcal{N}}$ is done by combining a cell list with advanced techniques outlined in Remark~\ref{rmk::sum}, so that the cost is independent of the number of Gaussians $M$. The complexity of this part is $O(4\pi r_c^3\rho_r N)$, where $\rho_r=N/V$ denotes the density. 
 Recalling Eq.~\eqref{eq::phiNDR}, the self term $\Phi_{i}^{\text{self}}$ can be computed at a cost of $O(N)$. $\Phi^{\rm mid}_{\text{SOG}}$ and $\Phi^{\rm long}_{\text{SOG}}$ are generated via the mid- and long-range solvers, respectively, where the analysis of these two terms is presented below. 

As mentioned in Remark~\ref{rmk::meshsize}, the total number of Fourier modes required in the 3D Fourier solver for evaluating $\Phi^{\rm mid}_{\text{SOG}}$ is $O(s_{0}^{-3}V^*)$ per particle, where $V^*=L_xL_yL_z^*$ denotes the volume of $\Omega^*$, and $s_0\sim r_c$ as discussed around Eq.~\eqref{eq::2.35z}. Let $\rho_r^*=N/V^*$ represent the number density of the padding system. Since the factor in the scaling step (Eq.~\eqref{eq::widehat-tilde}) can be precomputed and stored, the cost of this solver is $O(\mathcal{P}_x\mathcal{P}_y\mathcal{P}_zN+(r_c^3\rho_r^*)^{-1}N\log N)$,
where the first term is the cost on gridding and gathering, and the second term is the cost on the FFT.

The cost required in the Fourier-Chebyshev solver for evaluating $\Phi^{\rm long}_{\text{SOG}}$ depends on whether the FFT is utilized. If we use the non-FFT version, Algorithm~\ref{al::long_nonfft} 
in Appendix~\ref{appendix::simplified}, for calculation, the cost is $O(s_{ m+1}^{-2}L_xL_yP)$ per particle, and $s_{m+1}\sim \eta L_z$ by the range splitting. In this case, the number of Fourier modes is $O(1)$, and direct calculation in Algorithm~\ref{al::long_nonfft} does not require any gridding or gathering, resulting in significant speedup compared to the FFT-accelerated version Algorithm~\ref{al::long}. The overall complexity is
\begin{equation}\label{eq::comp}
C_{\text{comp}}=O\left(4\pi r_c^3\rho_rN+\mathcal{P}_x\mathcal{P}_y\mathcal{P}_zN+\frac{\lambda_z(1+\delta_z/L_z)}{r_c^3\rho_r}N\log N+\frac{PL_xL_y}{\eta^2 L_z^2}N\right).
\end{equation}
Clearly, if we choose $r_c\sim O(1)$ and assume $L_z=O(\sqrt{L_xL_y})$, the spectral SOG solver has an $O(N\log N)$ complexity. Typical values of these parameters in Eq.~\eqref{eq::comp} (and Eq.~\eqref{eq::comp3} below) under different tolerance levels are provided in Table~\ref{tabl:CubicTime2}.

%\tcr{Include a table summarizing typical parameter choices used in equations (3.23) and (3.24) for 1 or 2 tolerance levels.}

In general, we can always use the FFT-accelerated version Algorithm~\ref{al::long} for the long-range part to achieve the $O(N\log N)$ complexity. And the computational complexity becomes
\begin{equation}\label{eq::comp3}
\begin{aligned}
C_{\text{comp}}&=O\left(4\pi r_c^3\rho_rN+\mathcal{P}_x\mathcal{P}_y\mathcal{P}_zN+\frac{\lambda_z(1+\delta_z/L_z)}{r_c^3\rho_r}N\log N\right)\\ &+O\left(P\mathcal{P}_x\mathcal{P}_yN+\frac{PL_xL_y}{s_{m+1}^2}N\log N\right).
\end{aligned}
\end{equation}
This is especially needed when the confinement is strong, i.e. $L_z=o(1)\ll \sqrt{L_xL_y}$, where the last term in Eq.~\eqref{eq::comp} is $O(L_xL_yN)\sim O(N^2)$. 
To summarize, the complexity of the spectral SOG solver is $O(N\log N)$, which depends on neither the number of Gaussians $M$ nor the aspect ratio of the computational rectangular box.

\section{Error analysis and parameter selection}\label{sec::erroranalysis}
In this section, we study the pointwise error of the proposed spectral SOG solver, expressed as
\begin{equation}
\mathcal{E}(\bm{r}_i):=\Phi(\bm{r}_i) - \Phi_{\text{SOG}}(\bm{r}_i),
\end{equation}
where $\Phi_{\text{SOG}}$ given in Eq.~\eqref{eq::splitting} is the solution generated by the solver.
%\begin{equation}\label{eq::splitting}
%\Phi_{\text{SOG}}(\bm{r}_i)=\Phi^{\mathcal{N}}(\bm{r}_i)+\Phi^{\rm mid}_{\text{SOG}}(\bm{r}_i)+\Phi^{\rm long}_{\text{SOG}}(\bm{r}_i)+\Phi^{\text{self}}_i
%\end{equation}
%represents the approximate solution generated by the solver. 
The errors in the spectral SOG solver arise from various sources: decomposition error $\mathcal{E}_{\text{deco}}$ due to the use of the SOG decomposition, and the approximation errors $\mathcal{E}_{\rm mid}$ and $\mathcal{E}_{\rm long}$ incurred during the computation of $\Phi^{\rm mid}_{\text{SOG}}$ and $\Phi^{\rm long}_{\text{SOG}}$, respectively.  It is worth noting that the truncation error $\mathcal{E}_{\text{deco}}$ has been analyzed in Section~\ref{subsec::truncateerror}. The remaining two components, $\mathcal{E}_{\rm mid}$ and $\mathcal{E}_{\rm long}$, will be estimated independently in Sections~\ref{subsec::analysis3DFF} and \ref{subsec::analysFFCT}. Furthermore, we provide a scheme for parameter selection in Section~\ref{subsec::parasele}.

\subsection{Approximation errors of the mid-range components}\label{subsec::analysis3DFF}
We begin by analyzing the error $\mathcal{E}_{\rm mid}$ originating from the computation of $\Phi^{\rm mid}_{\text{SOG}}$. Within the framework of Algorithm~\ref{al::mid}, $\mathcal{E}_{\rm mid}$ can be decomposed into four distinct components:
\begin{equation}
\mathcal{E}_{\rm mid}=  \mathcal{E}_{\rm mid}^{\text{trun}}+\mathcal{E}_{\rm mid}^{\text{scal}}+\mathcal{E}_{\rm mid}^{\text{padd}}+\mathcal{E}_{\rm mid}^{\text{gath}},
\end{equation}
where 
\begin{itemize}
    \item $\mathcal{E}_{\rm mid}^{\text{trun}}$ arises from spectrum truncation of the mixing summation functional in Eq.~\eqref{eq::PhiNDF}, enabling the use of a finite Fourier sum.
    \item $\mathcal{E}_{\rm mid}^{\text{scal}}$ stems from the procedure described in Eqs.~\eqref{eq::grid}-\eqref{eq::eq33} for approximating $\widehat{S}_{\text{scal}}$, which includes the gridding error in Algorithm~\ref{al::mid} and the discretization error of Fourier integrals. 
    \item $\mathcal{E}_{\rm mid}^{\text{padd}}$ relates to the truncation of the infinite integral in Eq.~\eqref{eq::Phiconv} within $[-L_z^*/2,L_z^*/2]$, influenced by the zero-padding.
    \item $\mathcal{E}_{\rm mid}^{\text{gath}}$ is associated with the gathering step in the evaluation of Eq.~\eqref{eq::Phiconv}, including both the discretization of the integral via the trapezoidal rule and the truncation of the support of window function.
\end{itemize}
These components will be analyzed individually.

\begin{remark}
The analysis of $\mathcal{E}_{\rm mid}^{\text{scal}}$ is critical in showing that unlike standard NUFFT algorithms~\cite{nufft1,nufft2,nufft3}, Algorithm~\ref{al::mid} does not require any upsampling due to the decay of the Gaussians
in the SOG approximation. Standard NUFFTs require a factor of two upsampling in each dimension~\cite{nufft3}. Thus, Algorithm~\ref{al::mid} gives us a factor of eight in the FFT step in both storage and work.
\end{remark}
%Let us consider the error $\mathcal{E}_{\rm mid}$ stem from the fast Fourier spectral solver, described in Algorithm~\ref{al::mid}, while evaluating $\Phi^{\rm mid}_{\text{SOG}}$. Similar to the analysis of previous SE methods~\cite{lindbo2011spectral,lindbo2012fast}, we can decompose $\mathcal{E}_{\rm mid}$ into five parts:
%\begin{equation}
%\mathcal{E}_{\rm mid}=  \mathcal{E}_{\rm mid}^{\text{trun}}+\mathcal{E}_{\rm mid}^{\text{FFT}}+\mathcal{E}_{\rm mid}^{\text{padd}}+\mathcal{E}_{\rm mid}^{\text{trap}}+\mathcal{E}_{\rm mid}^{\text{wind}},
%\end{equation}
%where $\mathcal{E}_{\rm mid}^{\text{trun}}$ is due to the spectrum truncation of mixing summation functional in Eq.~\eqref{eq::PhiNDF} so that a finite grid is applicable, $\mathcal{E}_{\rm mid}^{\text{FFT}}$ is due to approximating Fourier integrals in the free directions in the 3D forward/backward FFT steps, $\mathcal{E}_{\rm mid}^{\text{padd}}$ is due to truncating the infinite integral in Eq.~\eqref{eq::Phiconv} within $[-L_z^*/2,L_z^*/2]$ which depends on the zero-padding factor, $\mathcal{E}_{\rm mid}^{\text{trap}}$ is due to the discretization of Eq.~\eqref{eq::Phiconv} (after truncation) via the trapezoidal rule, and $\mathcal{E}_{\rm mid}^{\text{wind}}$ is due to the truncation of the support of window functions $W$ in evaluating Eq.~\eqref{eq::Phiconv}. We will analyze these parts separately.

Assuming truncation of the Fourier summation functional in Eq.~\eqref{eq::PhiNDF} up to $|\bm{k}|\leq K_{\max}^{\text{mid}}$, the Fourier truncation error $\mathcal{E}_{\rm mid}^{\text{trun}}$ can be estimated by Theorem~\ref{thm::trunFour} under the assumption of the Debye-H$\ddot{\text{u}}$ckel (DH) theory (see \emph{\cite{hansen2013theory}} and also Appendix~\ref{app::DH}), where the proof closely follows that for the classical Ewald summation~\cite{kolafa1992cutoff}. %\rev{Note that for systems where DH theory is not applicable, corresponding estimates can also be completed using simple inequalities.}

\begin{theorem}\label{thm::trunFour}
Under the DH assumption, the following estimate holds
\begin{equation}
\label{eq::M1trun}
\|\mathcal{E}_{\rm mid}^{\emph{trun}}\|_\infty = O(K_{\max}^{\mathrm{mid}}e^{-s_{0}^2(K_{\max}^{\rm{mid}})^2}). 
\end{equation}
To achieve a proscribed tolerance $\varepsilon$, we have $K_{\max}^{\mathrm{mid}} = O(\sqrt{\log(1/\varepsilon)}/s_0)$.
\end{theorem}
\begin{proof}
We replace the summation functional in Eq.~\eqref{eq::PhiNDF} by an integral as outlined in Eq.~\eqref{eq::approx}, and then truncate it at $|\bm{k}|\leq K_{\max}^{\mathrm{mid}}$, resulting in
\begin{equation}\label{eq::EtrunM1}
\begin{split}
\mathcal{E}_{\rm mid}^{\text{trun}}(\bm{r}_i)&\approx \frac{1}{4\sqrt{\pi}}\sum_{\ell=0}^{M}w_{\ell}s_{\ell}^3\sum_{j=1}^{N}q_{j}\int_{K_{\max}^{\text{mid}}}^{\infty}k^2e^{-s_{\ell}^2k^2/4}\int_{-1}^{1}e^{ikr_{ij}\cos\theta}d\cos\theta dk\\
&=\frac{1}{4\sqrt{\pi}}\sum_{\ell=0}^{M}w_{\ell}s_{\ell}^3\int_{K_{\max}^{\text{mid}}}^{\infty}k^2e^{-s_{\ell}^2k^2/4}\mathcal{G}(r_i,k)dk
\end{split}
\end{equation}
where the kernel reads
\begin{equation}
\mathcal{G}(r_i,k):=\sum_{j=1}^{N}q_{j}\frac{\sin(kr_{ij})}{kr_{ij}}.
\end{equation}
By the DH theory outlined in Appendix~\ref{app::DH}, one can derive $\|\mathcal{G}(r_i,k)\|_\infty\leq Cq_{\max}$ with $C$ a certain constant and $q_{\max}:=\max\limits_{j}|q_{j}|$. Substituting this result into Eq.~\eqref{eq::EtrunM1} yields
\begin{equation}
\|\mathcal{E}_{\rm mid}^{\text{trun}}(\bm{r}_i)\|_\infty\leq \frac{Cq_{\max}}{16\sqrt{\pi}}\sum_{\ell=0}^{M}w_{\ell} \left(\sqrt{\pi}+2s_{\ell}K_{\max}^{\text{mid}}\right)e^{-s_{\ell}^2(K_{\max}^{\text{mid}})^2}\sim O\left(K_{\max}^{\mathrm{mid}}e^{-s_{0}^2(K_{\max}^{\text{mid}})^2}\right).
\end{equation}
\end{proof}

In Steps $1-3$ of Algorithm~\ref{al::mid}, we sample $S_{\text{grid}}$ using a truncated window function, then obtain $\widehat{S}_{\text{grid}}$ from $S_{\text{grid}}$ via a forward FFT, where the Fourier integrals are uniformly discretized. Subsequently, $\widehat{S}_{\text{scal}}$ is derived by scaling $\widehat{S}_{\text{grid}}$ by a factor. The resulting error reflected in $\Phi_{\text{SOG}}^{\rm mid}$, denoted as $\mathcal{E}_{\rm mid}^{\text{scal}}$, depends on the chosen window. For instance, Theorem~\ref{thm::scal} provides an estimate specifically for the Gaussian window defined as:
\begin{equation}\label{eq::GS}
    W_{\text{GS}}(x)=\begin{cases}e^{-\mathcal{S}\left(x/H\right)^2},\quad& |x|\leq H,\\
    0,\quad&\text{otherwise},
    \end{cases}\,\quad \widehat{W}_{\text{GS}}(k_x)=\sqrt{\frac{\pi}{\mathcal{S}}}He^{-k_x^2H^2/(4\mathcal{S})},
 \end{equation}
 where $\mathcal{S}>0$ is a shape parameter and $H>0$ is the half-width of the window function. Note that $W_{\text{GS}}$ is compactly supported on $[-H, H]$, while its Fourier transform, $\widehat{W}_{\text{GS}}$, exhibits Gaussian decay.

\begin{theorem}
    \label{thm::scal}
    Given a grid size $h_d$ with $d\in\{x,y,z\}$ and a Gaussian window $\widehat{W}=\widehat{W}_{\emph{GS}}$ with half-width support $H_d$ in each dimension, the number of  grid points within the support of the window is $\mathcal{P}_d=2H_d/h_d+1$. If the bandwidth of the window is smaller than that of mid-range Gaussians, i.e. 
\begin{equation}\label{eq::condition}
H_d^2/\mathcal{S}<s_{\ell}^2\leq s_{m}^2,\quad \ell=0,1,\cdots,m,
\end{equation} 
we have the estimate
\begin{equation}
\left\|\mathcal{E}_{\rm mid}^{\emph{scal}}(\bm{r}_i)\right\|_\infty = O\left(\max_{d\in\{x,y,z\}}\left\{\sum_{\ell=0}^{m}w_{\ell}e^{-\frac{\pi^2\mu_{d}^{\ell}(\mathcal{P}_d-1)^2}{4s_{\ell}^2\mathcal{S}}}\right\}+\erfc\left(\sqrt{\mathcal{S}}\right)\right),
\end{equation}
where the coefficients are given by $\mu_{d}^{\ell} = s_{\ell}^2 - H_d^2 / \mathcal{S}$ for $\ell = 0, 1, \dots, m$ and $d\in\{x,y,z\}$.
\end{theorem}
\begin{proof}
To estimate $\mathcal{E}_{\rm mid}^{\text{scal}}$, let us denote 
\begin{equation}
\overline{\bm{r}}_j = \bm{r}_j - \bm{u}_j \circ \bm{h},\quad \bm{u}_j = [u_{j,x}, u_{j,y}, u_{j,z}],
\end{equation}
as the position of the nearest grid points of the $j$th particle, where $|u_{j,d}| \leq 1/2$ and $\bm{h} = [h_x, h_y, h_z]$. Following Steps $1-3$ of Algorithm~\ref{al::mid}, $\mathcal{E}_{\rm mid}^{\text{scal}}$ can be analytically expressed as
\begin{equation}\label{eq::4.11}
\mathcal{E}_{\rm mid}^{\text{scal}}(\bm{r}_i) = 4\pi\mathcal{L}\left[\left(\widehat{S}_{\text{scal}}(\bm{k}) - \widehat{S}_{\text{scal}}^{\bm{\mathcal{P}}}(\bm{k})\right)\widehat{W}(\bm{k})e^{\mathrm{i}\bm{k}\cdot\bm{r}_i}\right],
\end{equation}
where
\begin{equation}
\widehat{S}_{\text{scal}}^{\bm{\mathcal{P}}}(\bm{k}) := \dfrac{\widehat{\tau}_{\text{SOG}}^{\rm mid}(\bm{k})}{k^2|\widehat{W}(\bm{k})|^{2}}\sum_{j=1}^{N}q_j \left[V_h\sum_{d\in\{x,y,z\}}\sum_{\mathcal{N}_d=-\frac{\mathcal{P}_{d}-1}{2}}^{\frac{\mathcal{P}_{d}-1}{2}}W\left((\bm{u}_j+\bm{\mathcal{N}}\right)\circ \bm{h})e^{-i\bm{k}\cdot(\bm{r}_j+\bm{\mathcal{N}}\circ\bm{h})}\right],
\end{equation}
with $\bm{h} = [h_x, h_y, h_z]$ and $V_h := h_xh_yh_z$ denoting the spatial size and volume of each grid. Eq.~\eqref{eq::4.11} can be further decomposed as 
\begin{equation}
\mathcal{E}_{\rm mid}^{\text{scal}} = \mathcal{E}_{\rm mid}^{\text{scal},\bm{\infty}} + \mathcal{E}_{\rm mid}^{\text{scal},\bm{\mathcal{P}}},
\end{equation}
where
\begin{equation}
\mathcal{E}_{\rm mid}^{\text{scal},\bm{\infty}}(\bm{r}_i):=4\pi\mathcal{L}\left[\left(\widehat{S}_{\text{scal}}(\bm{k}) - \widehat{S}_{\text{scal}}^{\bm{\infty}}(\bm{k})\right)\widehat{W}(\bm{k})e^{\mathrm{i}\bm{k}\cdot\bm{r}_i}\right]
\end{equation}
with $\widehat{S}_{\text{scal}}^{\bm{\infty}}(\bm{k}) = \lim\limits_{\bm{\mathcal{P}} \rightarrow \bm{\infty}} \widehat{S}_{\text{scal}}^{\bm{\mathcal{P}}}(\bm{k})$, and
\begin{equation}
\mathcal{E}_{\rm mid}^{\text{scal},\bm{\mathcal{P}}}(\bm{r}_i) := 4\pi\mathcal{L}\left[\left(\widehat{S}_{\text{scal}}^{\bm{\infty}}(\bm{k}) - \widehat{S}_{\text{scal}}^{\bm{\mathcal{P}}}(\bm{k})\right)\widehat{W}(\bm{k})e^{\mathrm{i}\bm{k}\cdot\bm{r}_i}\right].
\end{equation}
Here, $\mathcal{E}_{\rm mid}^{\text{scal},\bm{\infty}}$ is the error due to discretization of Fourier integral, and $\mathcal{E}_{\rm mid}^{\text{scal},\bm{\mathcal{P}}}$ is the error due to the truncation of the discrete sum, which is influenced by the fast decay of the window functions. Next, we analyze these two components individually.

We bound the discretization error $\mathcal{E}_{\rm mid}^{\text{scal},\bm{\infty}}$ by leveraging the periodicity of $W$ and applying the Poisson summation formula (see Appendix~\ref{subsec::PoissonSummation}) to reformulate the difference between $\widehat{S}_{\text{scal}}$ and $\widehat{S}_{\text{scal}}^{\bm{\infty}}$:
\begin{equation}
\label{eq::scal_D}
\begin{split}
\mathcal{E}_{\rm mid}^{\text{scal},\bm{\infty}}(\bm{r}_i)
&=4\pi\mathcal{L}\left[\dfrac{\widehat{\tau}_{\text{SOG}}^{\rm mid}(\bm{k})}{k^2\widehat{W}(\bm{k})}\sum_{j=1}^{N}q_{j}e^{i\bm{k}\cdot\bm{r}_{ij}}\sum_{\bm{n}\in\mathbb{Z}^3\slash\{\bm{0}\}}\widehat{\mathscr{T}\,}(2\pi\bm{n}\circ \bm{h}_{\text{inv}})\right],
\end{split}
\end{equation}
where $\bm{h}_{\text{inv}}:=[1/h_x,1/h_y,1/h_z]$ and $\widehat{\mathscr{T}\, }$ is the Fourier transform of $\mathscr{T}(\bm{k})=\widehat{W}(\bm{k})e^{-i\bm{k}\cdot\bm{r}}$. By the DH theory again, we have $|\sum_{j}q_{j}e^{-i\bm{k}\cdot\bm{r}_{ij}}|\leq Cq_{\max}$, where $C$ is a constant. By this bound and substituting the expression of Gaussian window in Eq.~\eqref{eq::GS} into Eq.~\eqref{eq::scal_D},
we have
\begin{equation}
\begin{split}
\left\|\mathcal{E}_{\rm mid}^{\text{scal},\bm{\infty}}(\bm{r}_i)\right\|_{\infty}\leq 4\pi C q_{\text{max}}\mathcal{L}\left[\frac{\widehat{\tau}_{\text{SOG}}^{\rm mid}(\bm{k})}{k^2e^{-|\bm{k}\circ\bm{H}|^2/(4\mathcal{S})}}\sum\limits_{\bm{n}\in \mathbb{Z}^3\backslash\{0\}}e^{-|(\bm{k}+2\pi\bm{n}\circ\bm{h}_{\text{inv}})\circ\bm{H}|^2/(4\mathcal{S})}\right],
\end{split}
\end{equation}
where and $\bm{H}=[H_x, H_y, H_z]$. Replacing the summation functional with a triple integral similar to Eq.~\eqref{eq::approx}, we obtain
\begin{equation}
\begin{aligned}
\left\|\mathcal{E}_{\rm mid}^{\text{scal},\bm{\infty}}(\bm{r}_i)\right\|_{\infty}&\leq\frac{Cq_{\text{max}}}{8\pi^{3/2}}\sum_{\ell=0}^{m}w_{\ell}s_{\ell}^3\sum\limits_{\bm{n}\in \mathbb{Z}^3\backslash\{\bm{0}\}}\prod_{d\in\{x,y,z\}}\int_{\mathbb{R}}e^{-\left(\frac{\mu_d^{\ell}k_d^2}{4}+\frac{H_d^2(k_d+2\pi n_d/h_d)^2}{4\mathcal{S}}\right)}dk_d\\
&=Cq_{\max}\sum_{\ell=0}^{m}w_{\ell}\sum_{\bm{n}\in\mathbb{Z}^3\backslash\{\bm{0}\}}\prod_{d\in\{x,y,z\}}e^{-\frac{\pi^2\mu_{d}^{\ell}H_d^2n_d^2}{\mathcal{S}s_{\ell}^2h_d^2}}.
\end{aligned}
\end{equation}
It is observed that the main contribution to the summation comes from the terms corresponding to $|\bm{n}|=1$, leading to
\begin{equation}\label{eq::4.17}
\left\|\mathcal{E}_{\rm mid}^{\text{scal},\bm{\infty}}(\bm{r}_i)\right\|_{\infty} = O\left(\max_{d\in\{x,y,z\}}\left\{\sum_{\ell=0}^{m}w_{\ell}e^{-\frac{\pi^2\mu_{d}^{\ell}(\mathcal{P}_d-1)^2}{4s_{\ell}^2\mathcal{S}}}\right\}\right).
\end{equation}

Similarly, by applying the DH theory and replacing the summation functional with an integral, the remainder term $\mathcal{E}_{\rm mid}^{\text{scal},\bm{\mathcal{P}}}(\bm{r}_i)$ can be bounded by
\begin{equation}\label{eq::4.190}
\begin{split}
\left\|\mathcal{E}_{\rm mid}^{\text{scal},\bm{\mathcal{P}}}(\bm{r}_i)\right\|_{\infty}&\leq \frac{CC_{W}q_{\max}}{2\pi^2}\int_{\mathbb{R}^3}\frac{\widehat{\tau}_{\text{SOG}}^{\rm mid}(\bm{k})}{k^2\left|\widehat{W}(\bm{k})\right|}d\bm{k},
\end{split}
\end{equation}
where
\begin{equation}\label{eq::CW}
C_{W}:=\left|V_{h}\sum_{\mathbb{Z}^3\backslash\left\{|\mathcal{N}_d|
\leq\frac{\mathcal{P}_{d}-1}{2}\right\}}W\left((\bm{u}_j+\bm{\mathcal{N}}\right)\circ \bm{h})\right|
\leq \int_{\mathbb{R}^3\backslash \{|\mathcal{N}_d|\leq H_d\}}e^{-\sum\limits_{d}(\mathcal{N}_d/H_d)^2\mathcal{S}}d\bm{\mathcal{N}}
\end{equation}
is a constant with respect to the window function. In Eq.~\eqref{eq::CW}, we use the monotone decreasing property of the Gaussian window. Since the integration domain on the RHS of Eq.~\eqref{eq::CW} is the entire space subtracted by a cube, we have
\begin{equation}\label{eq::4.21}
\begin{split}
C_{W}&\leq \left(\sqrt{\frac{\pi}{\mathcal{S}}}\right)^3\erfc\left(\sqrt{\mathcal{S}}\right)\left[1+\erf\left(\sqrt{\mathcal{S}}\right)+\erf\left(\sqrt{\mathcal{S}}\right)^2\right]\prod_{d\in\{x,y,z\}}H_d\\
&\leq 3 \left(\sqrt{\frac{\pi}{\mathcal{S}}}\right)^3 \erfc\left(\sqrt{\mathcal{S}}\right)\prod_{d\in\{x,y,z\}}H_d.
\end{split}
\end{equation}
Substituting Eq.~\eqref{eq::4.21} into Eq.~\eqref{eq::4.190} yields  
\begin{equation}\label{eq::4.22}
\begin{split}
\left\|\mathcal{E}_{\rm mid}^{\text{scal},\bm{\mathcal{P}}}(\bm{r}_i)\right\|_{\infty}&\leq 3Cq_{\text{max}}\left[\sum_{\ell=0}^{m}w_{\ell}\prod_{d\in\{x,y,z\}}\frac{s_{\ell}}{\mu_{d}^{\ell}}\right]\erfc(\sqrt{\mathcal{S}})\sim O\left(\erfc\left(\sqrt{\mathcal{S}}\right)\right).
\end{split}
\end{equation}
By combining Eqs.~\eqref{eq::4.17} and \eqref{eq::4.22}, we conclude the proof.
\end{proof}

%From the regularity of the functions involved, the spectral properties of $\mathcal{E}_{\rm mid}^{\text{FFT}}$ shown in Theorem \ref{thm::FFT} is clear.
% \begin{theorem}\label{thm::scal}
% The forward/backward FFTs applied in Algorithm~\ref{al::mid} are both spectrally accuracte with respect to the number of grids, i.e. 
% \begin{equation}
% \label{eq::M1FFT}
% \|\mathcal{E}_{\rm mid}^{\emph{FFT}}(\bm{r}_i)\|_\infty\sim O(e^{-C_{\emph{FFT}}\min\{I_x,I_y,I_z\}}),
% \end{equation}
% where $C_{\emph{FFT}}$ is a certain constant and can be determined apriori by numerical experiments.
% \end{theorem}
% \begin{proof}
% Note that both $S_{\text{grid}}(\bm{r})$ and $\widehat{S}_{\text{scal}}(\bm{k})$ are $C^{\infty}$ by the definition, and are compactly supported in real and Fourier spaces, respectively. It is a well-known result that the trapezoidal rule applied to an integrand with these properties is spectrally accurate~\cite{Trefen2014SIAMRev}.
% \end{proof}

It worth noting that Theorem~\ref{thm::scal} can be readily extended to other windows with Gaussian-type decay tails like the KB and ES. In the gathering step of Algorithm~\ref{al::mid}, the infinite integral in Eq.~\eqref{eq::Phiconv} is confined to the interval $[-L_z^*/2, L_z^*/2]$. %Owing to the compactness of window function, periodicity can be introduced in the $z$-direction without compromising accuracy.  
Consequently, this modification renders Eq.~\eqref{eq::Phiconv} into a form amenable to discretization:
\begin{equation}\label{eq::phitrun}
\Phi_{\mathrm{SOG}}^{\rm mid}\left(\boldsymbol{r}_i\right)\approx 4 \pi \int_{-L_z^*/2}^{L_z^*/2} \int_{\mathcal{R}^2} S_{\mathrm{scal}}\left(\left[\dot{\boldsymbol{r}}_j, z_j\right]\right) W\left(\dot{\boldsymbol{r}}_i-\dot{\boldsymbol{r}}_j\right)_* W\left(z_i-z_j\right)_* d \dot{\boldsymbol{r}}_j d z_j.
\end{equation}
The integral truncation error, denoted as $\mathcal{E}_{\rm mid}^{\text{padd}}$, is characterized by the difference between Eq.~\eqref{eq::Phiconv} and Eq.~\eqref{eq::phitrun}. Next, we employ the trapezoidal rule to discretize Eq.~\eqref{eq::phitrun}, truncating the window to a certain support. %and both the quadrature error $\mathcal{E}_{\rm mid}^{\text{trap}}$ and window truncation error $\mathcal{E}_{\rm mid}^{\text{wind}}$ are involved. 
The involved error, $\mathcal{E}_{\rm mid}^{\text{gath}}$, also depends on the selected window. For instance, Theorem~\ref{thm::pad} provides an estimation of $\mathcal{E}_{\rm mid}^{\text{padd}}$ and $\mathcal{E}_{\rm mid}^{\text{gath}}$, tailored to the Gaussian window.

\begin{theorem}
\label{thm::pad}   
Under the same assumption as in Theorem \ref{thm::scal},
we have the following estimates:
\begin{equation}
\label{eq::M1padd}
\|\mathcal{E}_{\rm mid}^{\emph{padd}}(\bm{r}_i)\|_{\infty}
=O\left(\erfc\left[\frac{\lambda_z(L_{z}+\delta_z)}{2\sqrt{\mu_{z}^{m}}}\right]\right),
\end{equation}
and
\begin{equation}
\label{eq::M1trap}
\left\|\mathcal{E}_{\rm mid}^{\emph{gath}}(\bm{r}_i)\right\|_{\infty}= O\left(\max_{d\in\{x,y,z\}}\left\{\sum_{\ell=0}^{m}w_{\ell}e^{-\frac{\pi^2\mu_{d}^{\ell}(\mathcal{P}_d-1)^2}{4s_{\ell}^2\mathcal{S}}}\right\}+\erfc\left(\sqrt{\mathcal{S}}\right)\right).
\end{equation}
\end{theorem}
\begin{proof}
Comparing Eq.~\eqref{eq::phitrun} with Eq.~\eqref{eq::Phiconv} and shifting the periodicity into the interval, we obtain
\begin{equation}\label{eq::4.14}
\begin{split}
\|\mathcal{E}_{\rm mid}^{\text{padd}}(\bm{r}_i)\|_{\infty}&\leq 4\pi\int_{\mathbb{R}^3\backslash\{|z_j|\leq L_z^*/2\}}|S_{\text{scal}}(\bm{r}_j)|W(\bm{r}_i-\bm{r}_j)d\bm{r}_j.
\end{split}
\end{equation}
By the definitions of $S_{\text{scal}}$ and the DH theory, we have
\begin{equation}\label{eq::4.26}
\begin{split}
|S_{\text{scal}}(\bm{r})|
&\leq \frac{C_{\text{scal}}q_{\max}}{32\pi^{5/2}}\sum_{\ell=0}^{m}w_{\ell}s_{\ell}^3\int_{\mathbb{R}^3}e^{-s_{\ell}^2k^2/4}\widehat{W}(\bm{k})^{-1}e^{i\bm{k}\cdot\bm{r}}d\bm{k}\\
&=\frac{C_{\text{scal}}q_{\max}}{4\pi^{5/2}}\sum_{\ell=0}^{m}w_{\ell}s_{\ell}^3\prod_{d\in\{x,y,z\}}\sqrt{\frac{\mathcal{S}}{\mu_{d}^{\ell}H_d^2}}e^{-r_d^2/\mu_{d}^{\ell}},
\end{split}
\end{equation}
where the positive constant $C_{\text{scal}}$ is due to the application of DH and integral transform. Eq.~\eqref{eq::4.26} gives an estimate of $S_{\text{scal}}$, and an elementary calculation shows that
\begin{equation}\label{eq::Wz}
\left|\int_{\mathbb{R}\backslash\{|z_j|\leq L_z^*/2\}}e^{-z_j^2/\mu_{z}^{\ell}}W(z_i-z_j)dz_j\right|\leq \sqrt{\pi\mu_{z}^{\ell}}\erfc\left(\frac{L_{z}^*}{2\sqrt{\mu_{z}^{\ell}}}\right).
\end{equation}
Substituting Eqs.~\eqref{eq::4.26} and \eqref{eq::Wz} into Eq.~\eqref{eq::4.14}, we have 
\begin{equation}\label{eq::padd}
\begin{split}
\|\mathcal{E}_{\rm mid}^{\text{padd}}(\bm{r}_i)\|_{\infty}&\leq C_{\text{scal}}q_{\max}\sum_{\ell=0}^{m}w_{\ell}s_{\ell}^3\frac{\mathcal{S}^{3/2}}{H_xH_yH_z}\erfc\left(\frac{L_{z}^*}{2\sqrt{\mu_{z}^{\ell}}}\right)\\
&= O\left(\erfc\left[\frac{\lambda_z(L_{z}+\delta_z)}{2\sqrt{\mu_{z}^{m}}}\right]\right)
\end{split}
\end{equation}
where we recall $\mu^{\ell}_{z}=s_{\ell}^2-H_z^2/\mathcal{S}$ and $s_{0}<s_{1}\cdots<s_{m}$.

The estimate of $\mathcal{E}_{\rm mid}^{\text{gath}}$ closely follows the approach we used in the proof of Theorem~\ref{thm::scal}. Recall that $\overline{\bm{r}}_{i}$ denotes the position of the nearest grid points of $i$th particle. One can write $\mathcal{E}_{\rm mid}^{\text{gath}}$ as the sum of two components, 
\begin{equation}
\mathcal{E}_{\rm mid}^{\text{gath},\bm{\infty}}(\bm{r}_i)=\Phi_{\text{SOG}}^{\rm mid}(\bm{r}_i)-\mathcal{T}_{\infty}(\bm{r}_i),\quad \text{and}\quad \mathcal{E}_{\rm mid}^{\text{gath},\bm{\mathcal{P}}}(\bm{r}_i)=\mathcal{T}_{\infty}(\bm{r}_i)-\mathcal{T}_{\bm{\mathcal{P}}}(\bm{r}_i),
\end{equation}
where 
\begin{equation}
\mathcal{T}_{\bm{\mathcal{P}}}(\bm{r}_i):=4\pi V_{h}\sum_{\mathcal{N}_d=-\frac{\mathcal{P}_{d}-1}{2}}^{\frac{\mathcal{P}_{d}-1}{2}}S_\text{scal}\left(\bar{\bm{r}}_i+\bm{\mathcal{N}}\circ\bm{h}\right)W(\bm{u}_{i}\circ\bm{h}+\bm{\mathcal{N}}\circ \bm{h})_*
\end{equation}
with $\mathcal{T}_{\bm{\infty}}(\bm{r}_i)=\lim\limits_{\bm{\mathcal{P}}\rightarrow\bm{\infty}}\mathcal{T}_{\bm{\mathcal{P}}}(\bm{r}_i)$. Here, $\mathcal{E}_{\rm mid}^{\text{gath},\bm{\infty}}$ and $\mathcal{E}_{\rm mid}^{\text{gath},\bm{\mathcal{P}}}$ arise from the application of the trapezoidal rule and the truncation of window functions, respectively. By the Poisson summation formula, we bound $\mathcal{E}_{\rm mid}^{\text{gath},\bm{\infty}}$ by 
\begin{equation}\label{eq::trap}
\begin{split}
\left\|\mathcal{E}_{\rm mid}^{\text{gath},\bm{\infty}}(\bm{r}_i)\right\|_{\infty}&\leq 4\pi \sum_{\bm{n}\in\mathbb{Z}^3\backslash\{\bm{0}\}} \left|(\widehat{S}_{\text{scal}}*\widehat{W})(2\pi\bm{n}\circ \bm{h}_{\text{inv}})\right|.
\end{split}
\end{equation}
By the definitions of $\widehat{S}_{\text{scal}}$, the Gaussian window, and also by the DH theory, we obtain
\begin{equation}\label{eq::Sscal}
\begin{split}
\left|(\widehat{S}_{\text{scal}}*\widehat{W})(\bm{k})\right|&\leq Cq_{\max}\sum_{\ell=0}^{m}w_{\ell}s_{\ell}^3 \int_{\mathbb{R}^3}e^{-\frac{s_{\ell}^2\mathcal{S}k'^2-|\bm{H}\circ\bm{k}'|^2+|\bm{H}\circ(\bm{k}-\bm{k}')|^2}{4\mathcal{S}}}d\bm{k}'\\
&=Cq_{\max}\sum_{\ell=0}^{m}w_{\ell}\prod\limits_{d\in\{x,y,z\}}e^{-\frac{k_d^2H_d^2\mu_{d}^{\ell}}{4s_{\ell}^2\mathcal{S}}}
\end{split}
\end{equation}
where $C$ is a certain constant. Substituting Eq.~\eqref{eq::Sscal} into Eq.~\eqref{eq::trap} yields 
\begin{equation}
\left\|\mathcal{E}_{\rm mid}^{\text{gath},\bm{\infty}}(\bm{r}_i)\right\|_{\infty}= O\left(\max_{d\in\{x,y,z\}}\left\{\sum_{\ell=0}^{m}w_{\ell}e^{-\frac{\pi^2\mu_{d}^{\ell}(\mathcal{P}_d-1)^2}{4s_{\ell}^2\mathcal{S}}}\right\}\right),
\end{equation}
as the main contribution to Eq.~\eqref{eq::trap} comes from the terms corresponding to $|\bm{n}|=1$. The remainder term, $\mathcal{E}_{\rm mid}^{\text{gath},\bm{\mathcal{P}}}$, can be  bounded by the integral 
\begin{equation}\label{eq::4.23}
\left\|\mathcal{E}_{\rm mid}^{\text{gath},\bm{\mathcal{P}}}(\bm{r}_i)\right\|_{\infty}\leq 4\pi \int_{\mathbb{R}^3\backslash \{|t_{d}|\leq H_d\}}S_\text{scal}\left(\bar{\bm{r}}_i+\bm{t}\right)W(\bm{u}_{i}\circ\bm{h}+\bm{t}) d\bm{t}.
\end{equation}
By Eq.~\eqref{eq::4.22} and Eq.~\eqref{eq::4.26}, we have
\begin{equation}
\begin{split}
\left\|\mathcal{E}_{\rm mid}^{\text{gath},\bm{\mathcal{P}}}(\bm{r}_i)\right\|_{\infty}&\leq 4\pi \|S_{\text{scal}}(\bm{r})\|_{\infty}\left|\int_{\mathbb{R}^3\backslash \{|t_{d}|\leq H_d\}}W(\bm{u}_{i}\circ\bm{h}+\bm{t}) d\bm{t}\right|\\
&\leq 3C_{\text{scal}}q_{\text{max}}\erfc\left(\sqrt{\mathcal{S}}\right)\sum_{\ell=0}^{m}w_{\ell}s_{\ell}^3\prod_{d\in\{x,y,z\}}\left(\mu_{d}^{\ell}\right)^{-1/2}\\
&= O\left(\erfc\left(\sqrt{\mathcal{S}}\right)\right).
\end{split}
\end{equation}
\end{proof}

In Theorems~\ref{thm::scal} and \ref{thm::pad}, a crucial assumption is imposed: the support of $W$ should not exceed the interaction kernel, otherwise it would render the solver inefficient. While this condition is inherently satisfied by certain previous SE methods through their parameter choices (e.g., Eq.~(16) in~\cite{lindbo2011spectral} and Eq.~(22) in~\cite{lindbo2012fast}), this appears to be the first place where it is explicitly stated. To match $\mathcal{E}_{\rm mid}^{\text{scal}}$ and $\mathcal{E}_{\rm mid}^{\text{gath}}$, we choose
\begin{equation}
\label{eq::S}
\mathcal{S}\sim \min_{d\in\{x,y,z\}}\left\{\frac{\pi\left(\mathcal{P}_d-1\right)}{2}\right\}
\end{equation}
since $\erfc(x)\leq e^{-x^2}$ holds for all $x \geq 0$ and $\mu_{d}^{\ell}\sim s_{\ell}^2$. Theorem~\ref{thm::pad} also indicates that $\mathcal{E}_{\rm mid}^{\text{padd}}$ decays exponentially with respect to $\lambda_z^2$.

\begin{remark}
The shape parameter $\beta$ in KB or ES windows~\cite{barnett2019parallel} serves a role similar to $\mathcal{S}$ in the Gaussian window. Given that the shapes of these two windows are qualitatively similar to that of the Gaussian window, the approximation error estimate is expected to be close to Theorems~\ref{thm::scal} and \ref{thm::pad}. However, conducting a rigorous analysis remains challenging.
\end{remark}

\begin{remark}
In NUFFT-based solvers~\cite{nestler2015fast,nestler2016parameter}, a two-fold upsampling is typically employed to enhance accuracy when discretizing the Fourier integral. This process involves sampling the sources on a finer grid with spacing $h_d$ divided by $\sigma_d\approx 2$ along each dimension $d\in\{x,y,z\}$, which results in the Fourier grid spacing being scaled as $2\pi/(\sigma_d L_d)$. However, based on the results in Theorems~\ref{thm::scal} and \ref{thm::pad}, the discretization error in our method is insensitive to $h_d$ as long as the condition $H_d^2/\mathcal{S} < s_{\ell}^2$ is met. Consequently, our method does not require upsampling.
\end{remark}

\subsection{Approximation errors of the long-ranged components}\label{subsec::analysFFCT}
Next, we consider the error $\mathcal{E}_{\rm long}$ due to the use of Fourier-Chebyshev solver, described in Algorithm~\ref{al::long}, while computing $\Phi_{\text{SOG}}^{\rm long}$. We decompose $\mathcal{E}_{\rm long}$ into three parts: 
\begin{equation}
\mathcal{E}_{\rm long}=\mathcal{E}_{\rm long}^{\text{Four}}+\mathcal{E}_{\rm long}^{\text{Cheb}}+\mathcal{E}_{\rm long}^{\text{FFT}},
\end{equation}
where 
\begin{itemize}
    \item $\mathcal{E}_{\rm long}^{\text{Four}}$ is due to the truncation of Fourier expansion in periodic dimensions.
    \item $\mathcal{E}_{\rm long}^{\text{Cheb}}$ is due to the truncation of Chebyshev expansion in free direction.
    \item $\mathcal{E}_{\rm long}^{\text{FFT}}$ is due to the possible involvement of FFTs for acceleration.
\end{itemize}

Let us analyze these parts individually. According to the range splitting and explicitly integrating along the z-axis, $\Phi_{\text{SOG}}^{\rm long}$ has a closed form expression:
\begin{equation}\label{eq::trun2}
\Phi_{\text{SOG}}^{\rm long}(\bm{r}_i)=\dfrac{\pi}{L_xL_y}\sum_{j=1}^{N}q_j\sum_{\ell=m+1}^{M}w_{\ell}s_{\ell}^2e^{-z_{ij}^2/s_{\ell}^2} \sum\limits_{\dot{\boldsymbol{k}}\in \mathcal{K}^{2}}e^{-s_{\ell}^2\dot{k}^2/4} e^{\mathrm{i}\dot{\bm{k}}\cdot\dot{\bm{r}}_{ij}},
\end{equation}
which is more convenient for the following error estimation. 

\begin{theorem}\label{thm::Fcheb}
Assume the Fourier summation in Eq.~\eqref{eq::trun2} is truncated at $|\dot{\bm{k}}|\leq K_{\max}^{\rm long}$. Under the DH theory, we have the estimate
\begin{equation}
\label{eq::M2Four}
\left\|\mathcal{E}_{\rm long}^{\emph{Four}}(\bm{r}_i)\right\|_{\infty}= O(w_{m+1}e^{-s_{m+1}^2(K_{\max}^{\rm long})^2}).
\end{equation}
Furthermore, if a $P$-terms truncated Chebyshev expansion is applied to the free direction, then the Chebyshev truncation error is bounded by
\begin{equation}
\label{eq::M2Cheb}
\left\|\mathcal{E}_{\rm long}^{\emph{Cheb}}(\bm{r}_i)\right\|_{\infty}= O(w_{m}(2\sqrt{2}\eta)^{-P}/\sqrt{P!}).
\end{equation}
\end{theorem}
\begin{proof}
By the DH theory outlined in Appendix~\ref{app::DH}, we have
\begin{equation}\label{eq::boundDH}
\left\|\sum_{j=1}^{N}q_je^{\mathrm{i}\dot{\bm{k}}\cdot\dot{\bm{r}}_{ij}}e^{-z_{ij}^2/s_{\ell}^2}\right\|_{\infty}\leq C q_{\max} 
\end{equation}
where $C$ is a constant. Moreover, if the factor $e^{-z_{ij}^2/s_{\ell}^2}$ is removed from Eq.~\eqref{eq::boundDH}, the inequality still holds. Substituting Eq.~\eqref{eq::boundDH} into Eq.~\eqref{eq::trun2} and replacing the Fourier sum by an integral yield that
\begin{equation}
\begin{split}
\left\|\mathcal{E}_{\rm long}^{\text{Four}}(\bm{r}_i)\right\|_{\infty}& \leq \frac{1}{2}\sum_{\ell=m+1}^{M}w_{\ell}s_{\ell}^2\int_{\widetilde{K}_{\max}}^{\infty}\dot{k}e^{-s_{\ell}^2\dot{k}^2/4}\left\|\sum_{j=1}^{N}q_{j}e^{i\dot{\bm{k}}\cdot\dot{\bm{r}}_{ij}}e^{-z_{ij}^2/s_{\ell}^2}\right\|_{\infty}d\dot{k}\\
& \leq \frac{Cq_{\max}}{4}\sum_{\ell=m+1}^{M}w_{\ell} e^{-s_{\ell}^2(K_{\max}^{\rm long})^2}\\
&= O(w_{m+1}e^{-s_{m+1}^2(K_{\max}^{\rm long})^2}).
\end{split}
\end{equation}
Recall that $s_{\ell}\leq \eta L_z$ for all $\ell\geq m+1$. Once a $P$-terms Chebyshev approximation is applied in the $z$-direction, the Chebyshev truncation error can be bounded by
\begin{equation}\label{eq::echeby}
\begin{split}
\left\|\mathcal{E}_{\rm long}^{\text{Cheb}}(\bm{r}_i)\right\|_{\infty}&\leq \frac{C\pi q_{\max}}{L_xL_y}\sum_{\ell=m+1}^{M}w_{\ell} s_{\ell} \sum_{\dot{\bm{k}}\in\mathcal{K}^2}e^{-s_{\ell}^2 \dot{k}^2}\left\|e^{-z_{ij}^2/s_{\ell}^2}-\sum_{n=0}^{P-1}{}^{\prime}a_{n}T_{n}(z_{ij})\right\|_{\infty}\\
&\leq \frac{C\pi q_{\max}}{\sqrt{P!}(2\sqrt{2}\eta)^PL_xL_y}\sum_{\ell=m+1}^{M}w_{\ell}s_{\ell}^2\sum_{\dot{\bm{k}}\in\mathcal{K}^2}e^{-s_{\ell}^2\dot{k}^2}\\
&\leq \frac{Cq_{\max}}{4\sqrt{P!}(2\sqrt{2}\eta)^P}\sum_{\ell=m+1}^{M}w_{\ell},
\end{split}
\end{equation}
where the first and second inequalities are derived by using Eq.~\eqref{eq::boundDH} and Theorem~\ref{thm::Che}, respectively, and the last formula is obtained by replacing the sum over Fourier modes with an integral. From Eq.~\eqref{eq::echeby}, we arrive at $\left\|\mathcal{E}_{\rm long}^{\text{Cheb}}(\bm{r}_i)\right\|_{\infty}= O(w_{m}(2\sqrt{2}\eta)^{-P}/\sqrt{P!})$.
\end{proof}

It is worth noting that the solver does not rely on the FFT to achieve its asymptotic complexity for systems without strong confinement, so that the mesh discretization error $\mathcal{E}_{\rm long}^{\text{FFT}}$ is not always present. %One can sum the Fourier modes and Chebyshev series directly to recover the long-range potential. 
Once the FFT is utilized,  $\mathcal{E}_{\rm long}^{\text{FFT}}$ is involved, and the estimation is similar to what we have done for $\mathcal{E}_{\rm mid}^{\text{scal}} + \mathcal{E}_{\rm mid}^{\text{gath}}$, i.e.
\begin{equation}\label{eq::Em2FFT}
\left\|\mathcal{E}_{\rm long}^{\text{FFT}}(\bm{r}_i)\right\|_{\infty}= O\left(\max_{d\in\{x,y\}}\left\{\sum_{\ell=m+1}^{M}w_{\ell}e^{-\frac{\pi^2\mu_{d}^{\ell}\left(\mathcal{P}_d^{\rm long}-1\right)^2}{4s_{\ell}^2\mathcal{S}^{\rm long}}}\right\}
+\erfc\left(\sqrt{\mathcal{S}^{\rm long}}\right)\right),
\end{equation}
where $\mathcal{S}^{\rm long}$ is the shape parameter and $\mathcal{P}_d^{\rm long}$ represents the number of grid points within the support of the window function in the Fourier-Chebyshev solver.
\subsection{Parameter selection}\label{subsec::parasele}
Several parameters are involved in the proposed spectral SOG solver. Assume  the number of particles $N$, the dimensions of the system $L_x$, $L_y$, $L_z$, and the error tolerance $\varepsilon$ are given, we outline a systematic approach for selecting these parameters.

%One selects $b\sim e^{-\pi^2/\log(\varepsilon^2/8)}$ based on the pointwise error Eq.~\eqref{eq::pointwiseerror} as well as the truncation error of erengy and force given in Theorem~\eqref{thm:SOG}. To achieve a $C^1$ decomposition, the parameters $\sigma$ and $\omega$ are solved through the continuity equations Eqs.~\eqref{Eq::2.20} and \eqref{eq::2.22}. Finer tuning of $w_{\ell}$ and $s_{\ell}$ is conducted if a higher order of continuity is required. 

First, the cutoff $r_c$ is selected to optimize the runtime for computing the first and third terms of Eq.~\eqref{eq::comp}, with a rough estimate of $r_c\sim \rho_r^{-1/3}$. Second, the parameters associated with the SOG decomposition, including $b$, $\sigma$, $\omega$ and $M$, are determined by following the discussions at the end of Section~\ref{subsec::truncateerror}. Third, we determine the maximum frequency of mid-range component
\begin{equation}
K_{\max}^{\text{mid}}\sim \sqrt{\log(1/\varepsilon)}/s_0
\end{equation} 
using the estimates Eqs.~\eqref{eq::M1trun} and \eqref{eq::M2Four} for the mid-range and long-range solvers, respectively. For the solver of mid-range part, this also sets the number of grids as
\begin{equation}
I_{d}=2\left\lceil\frac{K_{\text{max}}L_d}{2\pi}\right\rceil,\quad d\in\{x,y,z\},
\end{equation}
as well as the mesh size $h_d=L_d/I_d$ according to Theorem~\ref{thm::trunFour}. 

After both the type of window function and the domain extension factor $\delta_z$ are selected, we determine the shape parameter and the grid points within the window's support using corresponding error estimates. For instance, as per Theorems~\ref{thm::scal} and \ref{thm::pad}, achieving precision with the Gaussian window requires $\mathcal{S}\sim \text{erfcinv}(\varepsilon)^2$, where $\text{erfcinv}(\cdot)$ represents the inverse complementary error function, easily computed using scientific computing software. 
The number of grid points within the support of window, $\mathcal{P}_d$, can also be obtained through the matching Eq.~\eqref{eq::S}. The half-support of the window is then computed as $H_d=(\mathcal{P}_d-1)h_d/2$. Subsequently, we approximate the window with its polynomial expansions within the support as mentioned in Remark \ref{rmk::PWindow}, where the polynomial degree is chosen adaptively.

The remaining parameters only appear in the calculation of free direction, including the zero-padding ratio $\lambda_z$, the range splitting parameter $\eta$, and the terms of the truncated Chebyshev expansion $P$. These parameters are coupled, and their determination involves solving an optimization problem according to Theorems~\ref{thm::scal}, \ref{thm::pad} and \ref{thm::Fcheb}:
\begin{equation}\label{eq::optimiz}
    \begin{split}
        \min &  \left(\dfrac{\lambda_z(1+\delta_z/L_z)}{r_c^3\rho_r}N\log N+\dfrac{C_{\text{rat}}PL_xL_y}{\eta^2L_z^2}N\right),\\
        \text{s.t.} &~\erfc\left[\dfrac{\lambda_z(L_z+\delta_z)}{2\sqrt{\eta^2L_z^2-H_z^2/\mathcal{S}}}\right]\leq \varepsilon,~\dfrac{(2\sqrt{2}\eta)^{-P}}{\eta L_z\sqrt{P!}}\leq \varepsilon. \\
    \end{split}
\end{equation}
where the total complexity is minimized according to Eq.~\eqref{eq::comp}. The factor $C_{\text{rat}}$ quantifies the cost ratio between evaluating a complex-exponential basis and a Chebyshev basis, with the latter being significantly more efficient. Typically, $C_{\text{rat}}$ ranges from $20$ to $100$ and highly depends on the specific hardware architecture and implementation details~\cite{Trefethen2019App,HOSSEINY202018}. After Eq.~\eqref{eq::optimiz} is solved for $\lambda_z$, $\eta$, and $P$, the maximum frequencies, the number of grids, and box size along the free direction are extended to size $K_{\max}^{\rm long}\sim \sqrt{\log(1/\varepsilon)}/s_{m+1}$, $I_z^*=\lceil\lambda_z(L_z+\delta_{z})/h_z\rceil$ and $L_z^*=h_zI_z^*$, respectively.

Finally, when utilizing the FFTs to accelerate the solver for the long-range component, the number of grids in the periodic dimensions is determined by
\begin{equation}
I_d^{\rm long}=2\left\lceil\frac{K_{\max}^{\rm long}L_d}{2\pi}\right\rceil,\quad d\in\{x,y\},
\end{equation}
and the mesh size reads $h_d^{\rm long}=L_d/I_d^{\rm long}$. The Chebyshev grids are set using Eq.~\eqref{eq::Chebyshevnode}. Parameters related to the window, $\mathcal{S}^{\rm long}$ and $\mathcal{P}_d^{\rm long}$, can be determined analogously to the approach used for $\mathcal{S}$ and $\mathcal{P}_d$, respectively.

\section{Numerical results}\label{sec::numer}
In this section, we provide several numerical examples to examine the accuracy and performance of our new spectral SOG method, including tests for (i) cubic systems and (ii) strongly confined systems with aspect ratio up to about $3000:1$. Accuracy is measured using the relative $L^\infty$ error in 
potential, defined by $\varepsilon_r=\|\Phi_\text{SOG}(\bm{r}_i)-\Phi_\text{true}(\bm{r}_i)\|_{\infty}/\|\Phi_\text{true}(\bm{r}_i)\|_{\infty}$, where the ``true'' solution is computed using the Ewald2D method~\cite{parry1975electrostatic} with machine precision. 
Our software~\cite{FastSpecSoG} is developed based on the Julia Programming Language~\cite{Julia-2017} and the open-source package CellListMap.jl~\cite{celllistmap}.
% The calculations are conducted on a laptop equipped with an Intel i5-13400F CPU running at 4.6GHz and 32GB of RAM.
The calculations are conducted on a single core of a server equipped with an AMD Ryzen Threadripper PRO 3995WX CPU running at 2.2GHz and 256GB of RAM.

%For all accuracy presented in numerical examples, we discuss about the relative error with respect to the total system energy, defined by $\varepsilon_r=|E_\text{calc}-E_\text{true}|/|E_\text{total}|$. And all experiments are processed on a single core of a machine with an Intel Core i5-13400F central processing unit (CPU), running at 4.6 GHz with 32 GB of memory. The selection of experimental parameters will be specifically elaborated in each respective subsection. Computations and simulations were conducted using our self-developed Julia software package, which is now avaliable at \url{https://github.com/HPMolSim/FastSpecSoG.jl}.

\subsection{Accuracy test in a cubic system}\label{sec::acc}
In Sections~\ref{subsec::analysis3DFF}-\ref{subsec::analysFFCT}, we study the error estimates for each component of the proposed spectral SOG solver. In this subsection, we conduct several tests to validate these theoretical results for a cubic box with a side length of 20, containing 1000 randomly distributed monovalent particles satisfying the charge neutrality condition Eq.~\eqref{eq::chargeneutral}. By default, the parameters from the last row of Table~\ref{tabl:parameter} are used for the SOG decomposition, with an error level of $10^{-14}$. All other parameters are selected sufficiently large according to Section~\ref{subsec::parasele}, ensuring their impact is negligible. 

\begin{figure}[htbp] 
    \centering
    \includegraphics[width=1\textwidth]{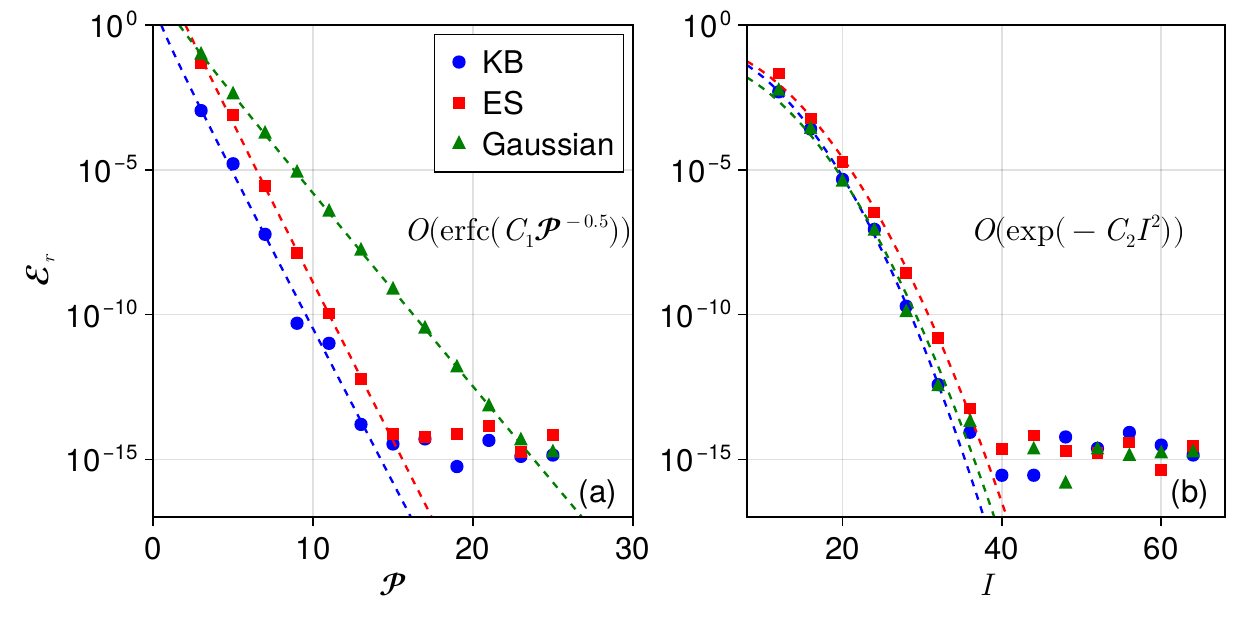}
%    \setcaptionwidth{1\textwidth}
    \caption{The relative error in evaluating $\Phi_{\text{SOG}}^{\rm mid}$ is plotted against (a) the window supports $\mathcal{P}$ and (b) the number of Fourier grids $I$ in each dimension. Data are shown for different type of window functions. The dashed lines in (a) and (b) represent the theoretical convergence rates provided by Theorems~\ref{thm::pad} (associated with Eq.~\eqref{eq::S}) and \ref{thm::trunFour}, respectively. The range splitting factor is fixed as $\eta=0.294$, which corresponds to $m=8$.}
    \label{fig:window}
\end{figure}

In our first numerical test, we compare the accuracy of Gaussian, KB, and ES window functions used in the gridding and gathering steps of Algorithm~\ref{al::mid}. The window supports and number of FFT grids are set to be the same along each dimension, i.e. $\mathcal{P}_x=\mathcal{P}_y=\mathcal{P}_z=\mathcal{P}$ and $I_x=I_y=I_z=I$. We use polynomial interpolation of order $\nu=10$ for all windows, as outlined in Remark~\ref{rmk::PWindow}. For the shape parameters, we follow suggestions from~\cite{shamshirgar2021fast}, setting $\mathcal{S}=0.455\pi\mathcal{P}$ (very close to the estimate Eq.~\eqref{eq::S}) for the Gaussian and 
$\beta=2.5\mathcal{P}$ for both the KB and ES. Figure~\ref{fig:window}(a) shows the relative error plotted as a function of the window support $\mathcal{P}$. The results indicate that the error decreases with a rate of 
$O(\erfc(C_1\sqrt{\mathcal{P}}))$, which aligns with our theoretical analysis for $\mathcal{E}_{\rm mid}^{\text{scal}}$ and $\mathcal{E}_{\rm mid}^{\text{gath}}$. The fitted values for the prefactors are around $C_1\approx 1.563$ for the Gaussian and 
$C_1\approx 1.226$ for both the KB and ES. Also, machine precision is achieved with $\mathcal{P}=24$ for the Gaussian and $\mathcal{P}=14$ for both the KB and ES. This suggests that the KB and ES can use a support size $1.7$ times smaller than the Gaussian in each dimension. While the decay rate remains similar, it is observed that ES requires $1-2$ more interpolation points than KB to achieve the same accuracy. This is partly because ES exhibits discontinuity near the truncation point, which reduces the accuracy of polynomial approximation. 

In Figure~\ref{fig:window}(b), we depict the error as a function of $I$, the number of FFT grids along each dimension, while maintaining the window support $\mathcal{P}=24$. An $O(e^{-C_2I^2})$ convergence rate is clearly illustrated, with the prefactor $C_2\approx 0.0244$ for all these windows. This demonstrates that the error introduced by the number of Fourier grids is insensitive to the choice of window function, a point also supported by our error estimate. Increasing the number of grid points primarily affects the truncation point of the Fourier mode, \( K_{\max}^{\mathrm{mid}} \sim \sqrt{2} \pi I / L \), as described in Theorem~\ref{thm::trunFour}, resulting in a Gaussian decay rate with respect to \( I \).

\begin{figure}[!ht] 
    \centering
    \includegraphics[width=1\textwidth]{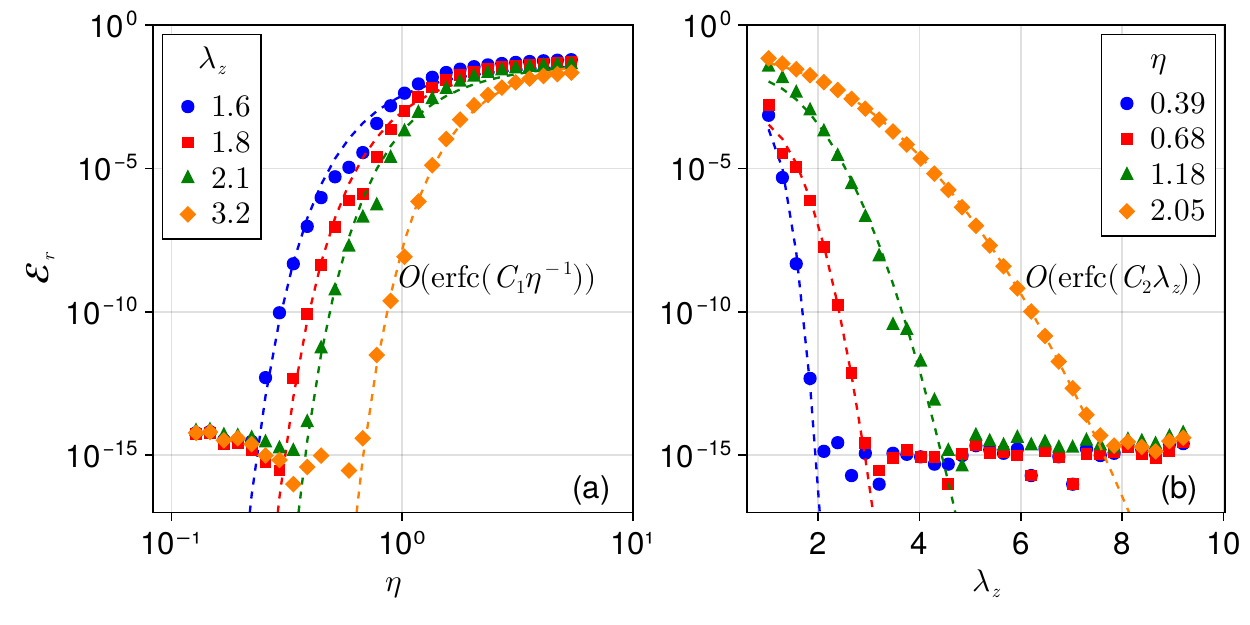}
    \caption{The relative error in evaluating $\Phi_{\text{SOG}}^{\rm mid}$ as a function of (a) the range splitting factor $\eta$ and (b) the zero-padding factor $\lambda_z$. The dashed line represents the theoretical convergence rates presented in Theorem~\ref{thm::pad}. In (b), the range splitting factor $\eta\approx0.39$, $0.68$, $1.18$, and $2.05$ corresponds to $\rm mid=10$, $14$, $18$, and $22$, respectively.}
    \label{fig:pad}
\end{figure}

In the forthcoming tests, we use a KB window with support $P=16$ and polynomial approximation order of $\nu=10$, while maintaining a fixed number of Fourier grids at $I=64$. In Figure~\ref{fig:pad}(a-b), we plot the error in evaluating $\Phi_{\text{SOG}}^{\rm mid}$ as a function of the range splitting factor $\eta$ and zero padding factor $\lambda_z$, respectively. According to Theorem~\ref{thm::pad}, these two parameters primarily affect the padding error $\mathcal{E}_{\rm mid}^{\text{padd}}$, expected to be scaled as $O(\erfc(C\lambda_z/\eta))$ under the condition $s_{m}\sim \eta L_z$. It is evident from both (a) and (b) that the error exhibits a rapid variation, scaling with rates of $\erfc(C_1\eta^{-1})$ and $\erfc(C_2\lambda_z)$, respectively, in good accordance with  Theorem~\ref{thm::pad}. Moreover, the fitting prefactors are around $C_1\approx1.27$, $1.68$, $2.06$ and $3.71$ for $\lambda_z\approx 1.3$, $1.5$, $1.7$, and $2.3$; and $C_2\approx 4.13$, $2.61$, $1.88$, and $1.13$ for $\eta=0.39$, $0.68$, $1.18$, and $2.05$ respectively. Thus, we have $C_1\sim O(\lambda_z)$ and $C_2\sim O(\eta^{-1})$ as observed. In practice, one can adjust the ratio between $\lambda_z$ and $\eta$ to ensure that $\mathcal{E}_{\rm mid}^{\text{padd}}$ achieves the desired accuracy. 

\begin{figure}[!ht] 
    \centering
    \includegraphics[width=1\textwidth]{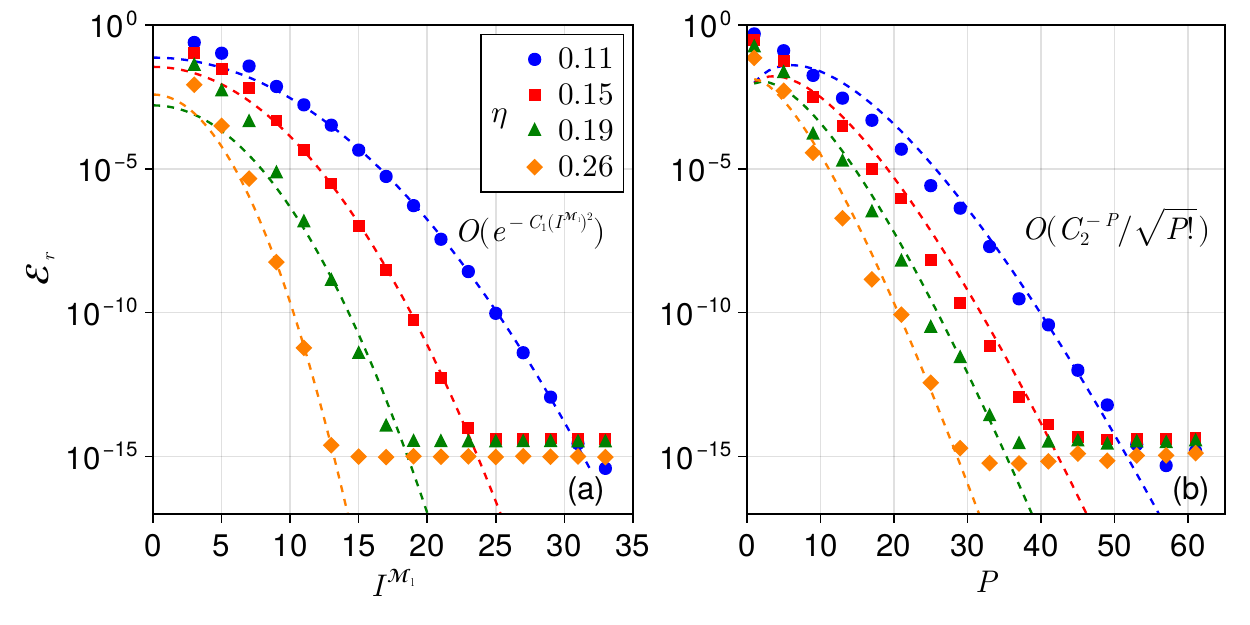}
    \caption{The relative error in evaluating $\Phi_{\text{SOG}}^{\rm long}$ as a function of (a) the number of Fourier modes $I^{\rm long}$ and (b) the number of Chebyshev basis $P$. Data are shown for different range splitting parameter $\eta$, where $\eta\approx0.11$, $0.15$, $0.19$, and $0.26$ correspond to $\rm mid=1$, $3$, $5$, and $7$, respectively. The dashed lines represent the theoretical convergence rate provided in Theorem~\ref{thm::Fcheb}.}
    \label{fig:long}
\end{figure}

Next, we examine the error when using the Fourier-Chebyshev solver to calculate $\Phi_{\text{SOG}}^{\rm long}$. In this test, we select its non-FFT version, Algorithm~\ref{al::long_nonfft}, as it is more efficient when $L_z\sim\sqrt{L_xL_y}$. We focus on two aspects: the number of Fourier modes ($I_x^{\rm long}=I_y^{\rm long}=I^{\rm long}$) required in the periodic directions, and the degree $P$ of the Chebyshev series needed in the free direction. Figures~\ref{fig:long}(a-b) illustrate how the error varies with these parameters, following rates of $O(e^{-C_1 (I^{\rm long})^2})$ and $O(C_2^{P}/\sqrt{P!})$, respectively, which are consistent with the estimates from Theorem~\ref{thm::Fcheb} for $\mathcal{E}_{\rm long}^{\text{Four}}$ and $\mathcal{E}_{\rm long}^{\text{Cheb}}$. As $\eta$ increases, the long-range potential $\Phi_{\text{SOG}}^{\rm long}$ becomes smoother in $[0,L_z]$, while its Fourier transform narrows in $\mathcal{K}^2$, requiring fewer Fourier modes and Chebyshev bases. This trend is quantified by the fitting prefactors, where $C_1\approx0.032$, $0.056$, $0.082$, $0.167$ and $C_2\approx0.39$, $0.49$, $0.61$, $0.84$ correspond to $\eta\approx0.11$, $0.15$, $0.19$, $0.26$, respectively. Consequently, compared to Figure~\ref{fig:window}(b), $I^{\rm long}$ is much smaller than the number of grids required for evaluating $\Phi_{\text{SOG}}^{\rm mid}$. 

%Regarding the long-range component, the parameters of primary interest to us are the number of Fourier grid points required in the periodic direction and the number of Chebyshev interpolation nodes needed in the nonperiodic direction. Figure \ref{fig:long} reflects the decay of relative error with respect to (a) the number of Fourier grid points in the $xy$ direction and (b) the number of Chebyshev interpolation nodes in the $z$ direction, under different critical term $\rm mid$. In contrast to the mid-range component, as the critical bandwidth increases, the corresponding Gaussian interaction becomes smoother in real space, requiring fewer Fourier grid points and Chebyshev bases. One point to note is that for the long-range Gaussian in a cubic box system, the number of Fourier grid points in the periodic direction only needs to be of $O(1)$ magnitude, thus linear time complexity can be achieved without the need for window function interpolation and FFT techniques. In the method of direct summation, increasing the number of Fourier grid points $N_x$ and $N_y$ is equivalent to increasing the Fourier truncation $K_\text{max}$, hence Figure \ref{fig:long}(a) exhibits exponential quadratic convergence rather than relying on the spectral convergence properties of window functions and FFT, as stated in Theorem \ref{thm::Fcheb}.

\subsection{Accuracy test in a strongly confined system}\label{sec::strongconf}

\begin{figure}[tb] 
\centering
\includegraphics[width=1\textwidth]{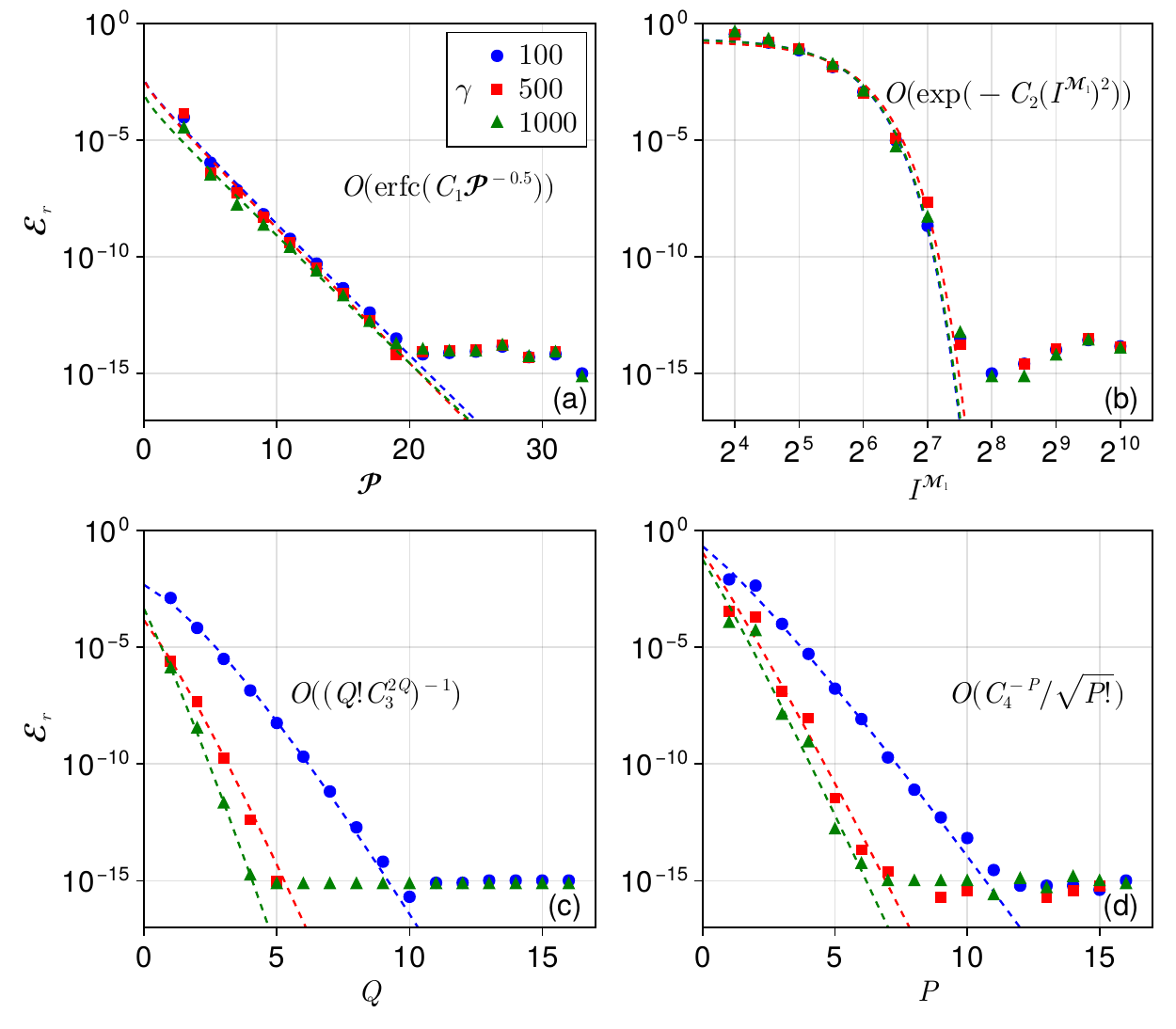}
\caption{The relative error in computing~$\Phi_{\text{SOG}}^{\rm long}$ against (a) window supports $\mathcal{P}$, (b) the number of Fourier grids $I^{\rm long}$ per dimension, (c) the number of Taylor expansion terms $Q$, and (d) the number of Chebyshev bases $P$. Data are shown for various aspect ratios $\gamma$. The dashed lines represent the theoretical convergence rate provided in this work.}
\label{fig:thin}
\end{figure}

Accurately handling strongly confined systems poses a significant challenge in quasi-2D electrostatic summations. Our proposed spectral SOG method provides an effective solution to this problem. To demonstrate its capability, we set \(L_x = L_y = 100\) and vary \(\gamma = L_x/L_z\) from 100 to 1000, creating extremely thin boxes containing 1000 randomly distributed monovalent particles. We choose the range-splitting factor \(\eta = 0.294\), ensuring that all Gaussians in the SOG decomposition satisfy the condition \(s_{\ell} > \eta L_z\), which eliminates the need for the mid-range solver, i.e., \(\Phi^{\rm mid}_{\text{SOG}}(\bm{r}_i) \equiv 0\). Furthermore, based on the complexity analysis in Section~\ref{subsec::complexity}, the long-range solver requires \(O(N)\) Fourier modes per particle. To achieve an \(O(N \log N)\) complexity, we employ the FFT-accelerated version, Algorithm~\ref{al::long}. Additionally, we utilize the Taylor expansion-based technique described in Appendix~\ref{subsec::Gridding} to accelerate the computation of Fourier-Chebyshev coefficients.

Figure~\ref{fig:thin} shows how the error in evaluating $\Phi_{\text{SOG}}^{\rm long}$ varies with different parameters as $\gamma$ varies. We use the KB window with polynomial approximation order $\nu=8$ and a shape parameter $\beta=7.5\mathcal{P}$. Note that this $\beta$ value is larger than that used for cubic systems and is determined through numerical experiments to ensure a near-optimal convergence rate of error with $\mathcal{P}$. Other parameters are selected according to Section~\ref{subsec::parasele}. Figure~\ref{fig:thin}(a-b) shows that the error scales with rates of $\erfc(C_1\sqrt{\mathcal{P}})$ and $e^{-C_2(I^{\rm long})^2}$, respectively. The prefactors $C_1\approx 1.12$ and $C_2\approx 1.101\times 10^{-3}$ remain constant regardless of the aspect ratio $\gamma$, aligning well with our estimate Eq.~\eqref{eq::Em2FFT} for $\mathcal{E}_{\rm long}^{\text{FFT}}$. 

In Figure~\ref{fig:thin}(c-d), it is shown that the error scales as $(C_3^{2Q}Q!)^{-1}$ and $C_4^{-P}/\sqrt{P!}$ with the number of Taylor terms $Q$ and the Chebyshev degree $P$, respectively, consistent with Lemma~\ref{lemma::Taylor} and Theorem~\ref{thm::Fcheb}. Moreover, the prefactors $C_3$ and $C_4$ are significantly influenced by the aspect ratio $\gamma$, with values approximately $2.38$, $6.90$, and $17.90$ for $C_3$, and $10.01$, $59.01$, and $96.40$ for $C_4$ when $\gamma=100$, $500$, and $1000$, respectively. These findings indicate that as $\gamma$ increases, the cost of our SOG spectral solver \emph{decreases gradually}. Remarkably, machine precision can be attained with $Q=P=6$ when $\gamma>1000$. Since existing algorithms generally incur higher costs as $\gamma$ increases to accommodate stronger anisotropy~\cite{yeh1999ewald,lindbo2012fast}, this underscores the advantage of our spectral SOG method in effectively addressing strongly confined quasi-2D systems.

\subsection{Scaling to large-scale systems}\label{sec::total}

\begin{figure}[!ht] 
    \centering
\includegraphics[width=1\textwidth]{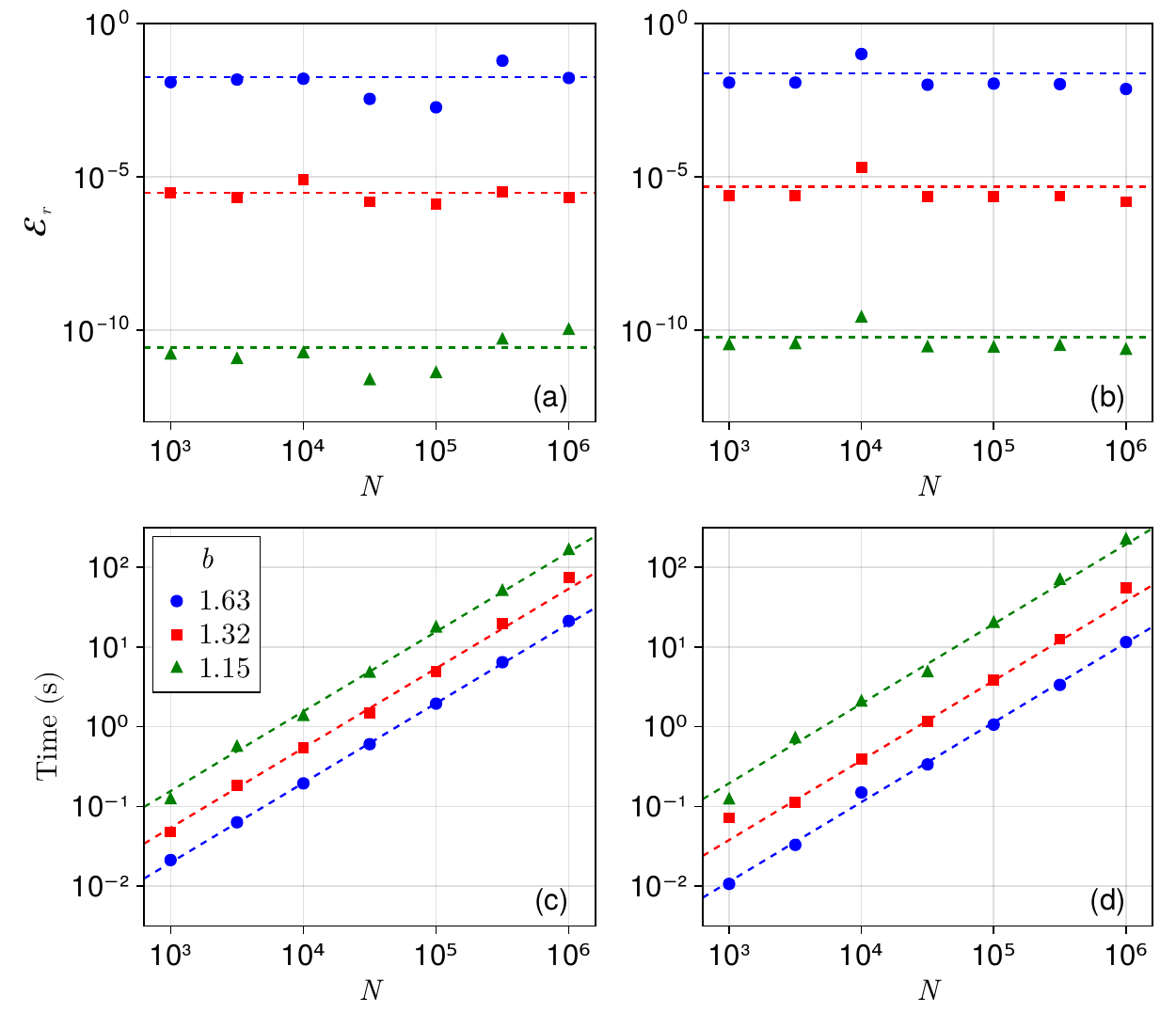}
    \caption{Total relative error is shown for (a) cubic systems with a fixed bulk density $N/V=0.125$ and (b) strongly confined systems with fixed surface density $N/(L_xL_y)=1.1$ and height $L_z = 0.3$, plotted against the number of particles $N$. The data represent different values of $b$ as detailed in Table~\ref{tabl:parameter}, where $b=1.63$, $1.32$, and $1.15$ correspond to precision levels of approximately $10^{-4}$, $10^{-6}$, and $10^{-14}$, respectively. In (c) and (d), we display the corresponding total CPU time for the scenarios in (a) and (b), respectively. The dashed lines in (a-b) and (c-d) indicate the average over seven data points and linear fitting, respectively.
    }
    \label{fig:total}
\end{figure}

In this section, we study the total error and CPU time performance for evaluating $\Phi(\bm{r}_i)$ with our proposed SOG spectral method. We explore two types of systems: cubic shapes with $\gamma=1$ and strongly confined cases with $\gamma\gg 1$. We vary $N$ from $10^3$ to $10^6$, maintaining a constant bulk density for cubic systems and surface density (i.e. $N/(L_xL_y)$) for strongly confined cases. Consequently, for the strongly confined cases, $\gamma$ ranges from $10^2$ to $10^{3.5}$. Our parameter selection process involves 1) determining suitable parameters for $N=10^3$ as explained in Section~\ref{subsec::parasele}; 2) scaling these parameters for larger systems while keeping the FFT mesh size $h_d$, Chebyshev degree $P$, range splitting factor $\eta$, and zero-padding factor $\lambda_z$ constant. The parameters we used for $N=10^3$ cases are provided in Table~\ref{tabl:CubicTime2}, Appendix \ref{app::ParaSelection}. 

\renewcommand\arraystretch{1.11}
\begin{table}[!ht]
    \caption{CPU time performance of the spectral SOG method for cubic systems for cubic systems of varying sizes and required precisions $\varepsilon_r$. In the table, $t_{\text{total}}$ is the total computational time; $t_{\text{short}}$, $t_{\text{mid}}$ and $t_{\text{long}}$ are the time required by the near-field, mid-range, and long-range component of the solver, respectively; $\zeta$ represents the ratio $t_{\text{long}}/t_{\text{total}}$.}
    \centering
    \begin{tabular}{ccccccc}
    \\
        \toprule
        $\varepsilon_r$ & $N$ & $t_\text{total}(s)$ & $t_\text{short}(s)$ & $t_\text{mid}(s)$ & $t_\text{long}(s)$ &  $\zeta(\%)$ \\
        \midrule
        $10^{-3}$ & $10^{3}$ & $2.12\times 10^{-2}$ & $6.95\times 10^{-3}$ & $1.40\times 10^{-2}$ & $2.49\times 10^{-4}$ & $1.17$\\
        $10^{-3}$ & $10^{4}$ & $1.95\times 10^{-1}$ & $8.13\times 10^{-2}$ & $1.11\times 10^{-1}$ & $2.57\times 10^{-3}$ & $1.32$\\
        $10^{-3}$ & $10^{5}$ & $1.95\times 10^{0}$ & $8.83\times 10^{-1}$ & $1.04\times 10^{0}$ & $2.60\times 10^{-2}$ & $1.34$\\
        $10^{-3}$ & $10^{6}$ & $2.13\times 10^{1}$ & $9.90\times 10^{0}$ & $1.11\times 10^{1}$ & $2.64\times 10^{-1}$ & $1.24$\\\\
        $10^{-6}$ & $10^{3}$ & $4.77\times 10^{-2}$ & $1.08\times 10^{-2}$ & $3.46\times 10^{-2}$ & $2.27\times 10^{-3}$ & $4.77$\\
        $10^{-6}$ & $10^{4}$ & $5.47\times 10^{-1}$ & $1.26\times 10^{-1}$ & $3.98\times 10^{-1}$ & $2.30\times 10^{-2}$ & $4.20$\\
        $10^{-6}$ & $10^{5}$ & $4.95\times 10^{0}$ & $1.36\times 10^{0}$ & $3.35\times 10^{0}$ & $2.48\times 10^{-1}$ & $5.01$\\
        $10^{-6}$ & $10^{6}$ & $7.55\times 10^{1}$ & $1.46\times 10^{1}$ & $5.86\times 10^{1}$ & $2.31\times 10^{0}$ & $3.05$\\\\
        $10^{-12}$ & $10^{3}$ & $1.20\times 10^{-1}$ & $1.74\times 10^{-2}$ & $9.35\times 10^{-2}$ & $9.41\times 10^{-3}$ & $7.82$\\
        $10^{-12}$ & $10^{4}$ & $1.34\times 10^{0}$ & $1.98\times 10^{-1}$ & $1.04\times 10^{0}$ & $9.89\times 10^{-2}$ & $7.40$\\
        $10^{-12}$ & $10^{5}$ & $1.72\times 10^{1}$ & $2.11\times 10^{0}$ & $1.41\times 10^{1}$ & $9.88\times 10^{-1}$ & $5.74$\\
        $10^{-12}$ & $10^{6}$ & $1.61\times 10^{2}$ & $2.25\times 10^{1}$ & $1.29\times 10^{2}$ & $9.70\times 10^{0}$ & $6.01$\\
        \bottomrule
    \end{tabular}
	\label{tabl:CubicTime}
\end{table}

\renewcommand\arraystretch{1.11}
\begin{table}[!ht]
	\caption{CPU time performance of the spectral SOG method for strongly confined systems with different system sizes, aspect ratio $\gamma$, and required precision $\varepsilon_r$. The explanations for other parameters are consistent with Table~\ref{tabl:CubicTime}.
    %$t_{\text{total}}$ represents the total computational time, while $t_{\text{short}}$ and $t_{\text{long}}$ denote the time required by the near-field and long-range components of the solver, respectively.
 }
	\centering
	\begin{tabular}{cccccc}
        \\
        \toprule
        $\varepsilon_r$ & $N$ & $\gamma$ & $t_\text{total}(s)$ & $t_\text{short}(s)$ & $t_\text{long}(s)$  \\
        \midrule
            $10^{-3}$ & $10^{3}$ & $100$ & $1.06\times 10^{-2}$ & $5.21\times 10^{-3}$ & $5.43\times 10^{-3}$\\
            $10^{-3}$ & $10^{4}$ & $100\sqrt{10}$ & $1.50\times 10^{-1}$ & $5.46\times 10^{-2}$ & $9.49\times 10^{-2}$\\
            $10^{-3}$ & $10^{5}$ & $1000$ & $1.06\times 10^{0}$ & $5.61\times 10^{-1}$ & $5.02\times 10^{-1}$\\
            $10^{-3}$ & $10^{6}$ & $1000\sqrt{10}$ & $1.15\times 10^{1}$ & $5.97\times 10^{0}$ & $5.57\times 10^{0}$\\
            \\
            $10^{-6}$ & $10^{3}$ & $100$ & $7.17\times 10^{-2}$ & $8.92\times 10^{-3}$ & $6.28\times 10^{-2}$\\
            $10^{-6}$ & $10^{4}$ & $100\sqrt{10}$ & $3.89\times 10^{-1}$ & $9.37\times 10^{-2}$ & $2.96\times 10^{-1}$\\
            $10^{-6}$ & $10^{5}$ & $1000$ & $3.84\times 10^{0}$ & $9.59\times 10^{-1}$ & $2.88\times 10^{0}$\\
            $10^{-6}$ & $10^{6}$ & $1000\sqrt{10}$ & $5.50\times 10^{1}$ & $9.89\times 10^{0}$ & $4.52\times 10^{1}$\\
            \\
            $10^{-12}$ & $10^{3}$ & $100$ & $1.19\times 10^{-1}$ & $1.42\times 10^{-2}$ & $1.05\times 10^{-1}$\\
            $10^{-12}$ & $10^{4}$ & $100\sqrt{10}$ & $2.02\times 10^{0}$ & $1.46\times 10^{-1}$ & $1.87\times 10^{0}$\\
            $10^{-12}$ & $10^{5}$ & $1000$ & $1.97\times 10^{1}$ & $1.48\times 10^{0}$ & $1.82\times 10^{1}$\\
            $10^{-12}$ & $10^{6}$ & $1000\sqrt{10}$ & $2.19\times 10^{2}$ & $1.53\times 10^{1}$ & $2.04\times 10^{2}$\\
        \bottomrule
        \end{tabular}
	\label{tabl:SConfinedTime}
\end{table}
 
Figures~\ref{fig:total}(a-b) show the total error as a function of $N$ for cubic and strongly confined systems, respectively, with variations in the decomposition parameter $b$ and precision level. The results indicate that the error remains stable as $N$ increases, suggesting that the parameters investigated in Sections~\ref{sec::acc} and~\ref{sec::strongconf} can be readily scaled to larger-scale systems. Figures~\ref{fig:total}(c-d) show the total runtime for $N$ ranging from $10^3$ to $10^6$, using the system settings from (a-b). These results confirm the nearly linear scaling of our spectral SOG method, even for $\gamma>1000$. The expected $O(N \log N)$ scaling due to the FFT is not observed in our tests, as the main costs are associated with the gridding and gathering steps, which scales as $O(N)$. 

Our spectral SOG method demonstrates superior efficiency for strongly confined systems. Tables~\ref{tabl:CubicTime} and \ref{tabl:SConfinedTime} present detailed CPU time results for each component of the solver while generating Figures~\ref{fig:total}(c-d). Comparing with the cubic case, the strongly confined case takes about half the total CPU time at $\varepsilon_r=10^{-3}$, slightly less at $10^{-6}$, and no more than twice the time at $10^{-12}$. This difference is due to the increasing proportion of time allocated to the long-range part, characterized by $\zeta$, as precision requirements rise. However, $\zeta$ remains below $10\%$ even at $10^{-12}$ precision, where only $8$ Chebyshev polynomials are required in the $z$-direction. 
Since Ewald-based spectral methods require $O(\gamma)$ zero-padding and thereby suffers from performance loss under strong confinement, this further underscores the efficiency of our algorithm.

\subsection{Performance test for a high-density system}\label{sec::highdense}
In this section, we evaluate the computational performance of our spectral SOG method in evaluating $\Phi(\bm{r}_i)$ for high-density systems. The system comprises of $10^5$ particles with dimensions $L_x=L_y=10$ in periodic directions. We consider two aspect ratios: $\gamma=1$ for the cubic case and $\gamma=100$ for the strongly confined case. Note that the cubic case was previously studied using the SE method~\cite{shamshirgar2021fast}.
We compare our newly proposed spectral SOG method with Ewald-based spectral methods. For a fair comparison, it is crucial to account for factors such as parallelization, vectorization, running environments, and FFT libraries~\cite{Gholami2016SISC,Pippig2013SISC}. Thus, our focus here is not on direct CPU time comparisons but on demonstrating the potential advantages of our method in reducing total FFT grids and efficiently handling strongly confined systems. 

\renewcommand\arraystretch{1.2}
\begin{table}[!ht]
	\caption{CPU time performance of the spectral SOG method for high-density systems with cubic shape ($\gamma=1$) and strongly confined shape ($\gamma=100$), with $L_x=L_y=10$, $\varepsilon_r=10^{-12}$, $b=1.17$, and $N=10^5$. The explanations for other parameters in the table are consistent with Table~\ref{tabl:CubicTime}. To produce these results, we use $r_c=1.8$, $\eta=0.37$, $I=94$, $\mathcal{P}=15$, $\lambda_z=1.99$, $I^{\rm long}=7$, and $P=20$ for $\gamma=1$ case; and $r_c=1.3$, $\mathcal{P}=19$, $I^{\rm long}=130$, $Q=6$, and $P=8$ for $\gamma=100$ case. The symbol ``$-$'' indicates the mid-range solver is not called.
}
	\centering
	\begin{tabular}{cccccc}
        \\
         \toprule
         $\gamma$ & $t_\text{total}(s)$ & $t_\text{short}(s)$ & $t_\text{mid}(s)$ & $t_\text{long}(s)$ &  $\zeta(\%)$ \\
         \midrule
        %   $1$ & $8.98$ & $4.15$ & $3.42$ & $1.40$ & $15.6$\\
        %   $100$ & $2.20\times 10^{1}$ & $1.15\times 10^{1}$ & $-$ & $8.73$ & $39.7$ \\
            $1$ & $1.60 \times 10^{1}$ & $7.87 \times 10^{0}$ & $5.64 \times 10^{0}$ & $2.51 \times 10^{0}$ & $15.7$\\
            $100$ & $3.53\times 10^{1}$ & $2.20\times 10^{1}$ & $-$ & $1.33 \times 10^{1}$ & $37.6$ \\
         \bottomrule
        \end{tabular}
	\label{tabl:CubicHighdense}
\end{table}

For the case $\gamma=1$, Table~\ref{tabl:CubicHighdense} presents CPU times for each component of the spectral SOG solver. Notably, the state-of-the-art TKM-based SE method requires an optimal zero-padding factor of $2$ and additional adaptive upsampling~\cite{shamshirgar2021fast}. We set the zero-padding factor of our method to a comparable $\lambda_z=1.99$ but without any upsampling. Following the error estimate of Ewald decomposition~\cite{kolafa1992cutoff}, initially proposed for 3DPBC but also applicable to quasi-2D cases~\cite{gan2024fast,lindbo2012fast}, the decomposition error of Ewald splitting is $\sim e^{-r_c^2/(\sqrt{2}\sigma_{\text{Ewald}})^2}$. Referring to Eq.~\eqref{eq::convrate}, the error for SOG decomposition is $\sim e^{-r_c^2/s_{-1}^2}=e^{-r_c^2/(\sqrt{2}b^{-1}\sigma_{\text{SOG}})^2}$ when $\omega$ approaches $1$, with $\sigma_{\text{SOG}}/\sigma_{\text{Ewald}}\approx b=1.17$ in this scenario. As discussed in~\cite{shan2005gaussian,DEShaw2020JCP}, the required number of FFT grids is approximately $1/\sigma^3$, hence Ewald-based calculations require about $1.17^3=1.6$ times the total number of FFT grids for equivalent $10^{-12}$ accuracy. Please note that this estimation assumes that the cost associated with each Fourier mode is equal for both methods.

Next, we study the $\gamma=100$ case, where all Gaussians are categorized as long-range. As the bulk density increases compared to the $\gamma=1$ case, the required FFT grids in the periodic directions of the long-range solver increases to $I^{\rm long}=130$. Notably, only $P=8$ Chebyshev grids are needed for the free direction calculation. In contrast, for Ewald-based methods, the number of FFT grids required in the free direction is proportional to $I^{\rm long}$ after zero-padding (as analyzed in Section~\ref{subsec::rangesplitting}). This suggests that Ewald-based methods would require approximately $130/8\approx 16.25$ times the total number of grids. Thus, our proposed spectral SOG method demonstrates significant promise for efficiently handling highly anisotropic quasi-2D systems.

\section{Conclusions}\label{sec::con}
We have introduced a fast spectral SOG solver for electrostatic summation in quasi-2D systems. Unlike previous Ewald-based spectral methods, our approach applies an SOG decomposition to the Laplace kernel, eliminating singular Fourier integrals and partitioning the interactions into near-field, mid-range, and long-range components. The near-field component is calculated directly, while the mid-range component is handled by a pure Fourier spectral solver, similar to the NUFFT. The long-range component, consisting of Gaussians with narrow Fourier space support, is addressed by an efficient Fourier-Chebyshev solver, where polynomial interpolation/anterpolation replaces the Fourier method in the free dimension. By leveraging FFT acceleration, our solver achieves \(O(N \log N)\) complexity per time step with a small prefactor. We have provided detailed discussions on grid discretization, error analysis, and parameter selection, with systematic tests confirming the accuracy and efficiency of our proposed solver.

Compared to previous spectral solvers, the newly proposed spectral SOG solver not only inherits super-algebraic convergence and  decouples approximation errors from truncation errors, but also addresses challenges arising from singularities and strong confinement without any upsampling. Furthermore, this solver offers the flexibility to incorporate other kernel functions easily through kernel-independent SOG methods. Future work will focus on CPU/GPU parallelization, extending the method to quasi-1D and free space electrostatic problems, and applying this approach to practical simulations.
In a certain sense, our scheme can be viewed as a simplified $O(N\log N)$ two-level nonadaptive version of the general DMK framework, since it also uses the SOG approximation as a dual-space kernel-splitting for the Laplace kernel. Thus, it is possible to extend our scheme to the fully adaptive case using the full machinery of the DMK framework.
%Future work: Extending to free space problem, compare with multilevel methods~\cite{greengard2023dual} and other SE methods	
	
\appendix
\section{Mathematical Preliminaries}
In this appendix, we present fundamental mathematical tools that underpin our work. These tools are standard in scientific computing and are extensively covered in foundational textbooks.
\subsection{Chebyshev Expansion}\label{subsec::ChebyshevExpansion}
Let the Chebyshev polynomial of degree $n$ on $[-1,1]$ be defined by the formula
		\begin{equation}
			T_n(\cos(\theta))=\cos(n\theta),\quad\theta\in[0,\pi].
		\end{equation}
		It follows $T_0(x)=1$, $T_{1}(x)=x$, and
		\begin{equation}
			T_{n+1}(x)=2xT_{n}(x)-T_{n-1}(x)\quad \text{for}\quad n\geq 1.
		\end{equation} 
The functions $T_n$ constitute an orthonormal basis with respect to the inner product 
\begin{equation}
\langle f,g\rangle=\int_{-1}^{1}\frac{f(x)g(x)}{\sqrt{1-x^2}}dx.
\end{equation}
Let $f(x)$ be a smooth function on the interval $[-1,1]$. Then it is well-known to have a rapidly converging Chebyshev series
	\begin{equation}
		 f(x)=\sum_{n=0}^{\infty}{}^{\prime}a_nT_n(x),
	\end{equation}
    where the prime indicates that there a factor of $1/2$ is multiplied in front of $a_0$, and the coefficients are strictly given by
    \begin{equation}
    	a_{n}=\frac{2}{\pi}\int_{0}^{\pi}f(\cos(\theta))\cos(k\theta)d\theta.
    \end{equation}

\subsection{Chebyshev Interpolation and Evaluation}\label{subsec::ChebyshevInterpolation}
The Chebyshev nodes of the first kind are the zeros of $T_n(x)$, given by
\begin{equation}\label{eq::Chebyshevnode}
    \left\{x_{i}=\cos\left[\frac{(i-1/2)\pi}{P}\right],\,\,i=1,\cdots,P\right\}.
\end{equation}
We define the $P\times P$ basis matrix $\bm{V}$ by $V(i,j)=T_{j}(x_i)$. Given
the vector of function values $\bm{f}= (f(x_1),\cdots,f(x_P))$, $f(x)$ can be approximated by a $P$-term truncated Chebyshev series where the coefficients $\bm{a}=(a_1,\cdots,a_P)$ can be obtained as
\begin{equation}\label{eq::ChebyshevInt}
    \bm{a}=\bm{V}^{-1}\bm{f}.
\end{equation}
We refer this procedure as ``Chebyshev Interpolation''.

%Assume that $\bm{a}$ is the coefficients of a $P$-term truncated Chebyshev series representation of function $f(x)$ so that Eq.~\eqref{eq::2.35} holds.         
Given a set of $N$ additional points $\{y_j\in[-1,1],\,\,j=1,\cdots,N\}$, we define the $N\times P$ evaluation matrix $\bm{E}$ by $E(j,n)=T_{n}(y_j)$, so that 
\begin{equation}\label{eq::ChebyshevEva}
    \bm{f}^*=\bm{E}\bm{a}
\end{equation}
is the value of the interpolant at the additional points. This procedure is referred as ``Chebyshev Evaluation''. In general, the value of $P$ depends on the desired precision of the calculation and the smoothness of the kernel $f$. 

% \subsection{Chebyshev Integration}\label{subsec::cheint}
% 	Let $f(x)$ be a smooth function on $[-1,1]$ given by a Chebyshev series
% 	\begin{equation}
% 		f(x)=\sum_{n=0}^{\infty}{}^{\prime}a_{n}T_{n}(x).
% 	\end{equation}
%     Then the integral of $f$ has a Chebyshev expansion of the form
%     \begin{equation}
%     	\int_{-1}^{x} f(t)dt=\sum_{n=1}^{\infty}b_{n}T_{n}(x)+b_0,
%     \end{equation}
%     where $b_0$ is a constant term and other coefficients are given by
%     \begin{equation}
%     	b_{n}=\frac{1}{2n}(a_{n-1}-a_{n+1}).
%     \end{equation}
% 	The double integral of $f$ also has an expansion of the form
% 	\begin{equation}
% 		\int_{-1}^{x}\int_{-1}^{\tau} f(t)dtd\tau=\sum_{n=2}^{\infty}c_{n}T_{n}(x)+c_1x+c_0,
% 	\end{equation}
%     where $c_1=b_0$, $c_0$ is a constant term and
%     \begin{equation}
%     	c_n=\frac{1}{2n}\left[\frac{a_{n-2}-a_n}{2(n-1)}-\frac{a_n-a_{n+2}}{2(n+1)}\right].
%     \end{equation}

\subsection{Convolution Theorem}\label{subsec::conv}
 Let $f([\dot{\bm{r}},z])$, $g([\dot{\bm{r}}, z])$ be two functions that are periodic in $\dot{\bm{r}}$ and non-periodic in $z$. Suppose that $f$ and $g$ have Fourier transform $\widehat{f}$ and $\widehat{g}$, respectively. Their convolution is defined by
\begin{align}
u([\dot{\boldsymbol{r}}, z]):=(f * g)([\dot{\boldsymbol{r}}, z])=\int_{\mathcal{R}^2} \int_{\mathbb{R}} f\left([\dot{\boldsymbol{r}}-\dot{\boldsymbol{r}}^{\prime}, z-z^{\prime}]\right) g\left(\dot{\boldsymbol{r}}^{\prime}, z^{\prime}\right) d z^{\prime} d \dot{\boldsymbol{r}}^{\prime},
\end{align}
satisfying
\begin{align}
\widetilde{u}([\dot{\boldsymbol{k}}, k_z])=\widetilde{f}([\dot{\boldsymbol{k}}, k_z]) \widetilde{g}([\dot{\boldsymbol{k}}, k_z]).
\end{align}

\subsection{Plancherel's Theorem}\label{subsec::Plancherel}
Let $f([\dot{\bm{r}},z])$, $g([\dot{\bm{r}}, z])$ be two functions that are periodic in $\dot{\bm{r}}$ and non-periodic in $z$, and have Fourier transform $\widehat{f}([\dot{\bm{k}},k_z])$ and $\widehat{g}([\dot{\bm{k}},k_z])$, respectively. Then we have
\begin{equation}
    \int_{\mathbb{R}} \int_{\mathcal{R}^2} f([\dot{\boldsymbol{r}}, z]) \overline{g([\dot{\boldsymbol{r}}, z])} \mathrm{d} \dot{\boldsymbol{r}} \mathrm{d} z=\frac{1}{2 \pi L_xL_y} \int_{\mathbb{R}} \sum_{\dot{\boldsymbol{k}}\in\mathcal{K}^2} \hat{f}([\dot{\boldsymbol{k}}, k_z]) \overline{\hat{g}([\dot{\boldsymbol{k}}, k_z])} \mathrm{d} k_z,
\end{equation}
where the overline indicates taking conjugation. A special case of above identity is when $f([\dot{\bm{r}},z])\equiv g([\dot{\bm{r}},z])$.

\subsection{Poisson summation formula}\label{subsec::PoissonSummation}
Let $f(\bm{r})$ and $\hat{f}(\bm{k})$ be a Fourier transform pair with $\bm{r},\bm{k}\in\mathbb{R}^3$. Suppose $\bm{h}=[h_x,h_y,h_z]$ and $\bm{h}_{\text{inv}}=[1/h_x,1/h_y,1/h_z]$ with $h_d>0$, we have
\begin{equation}        \sum_{\bm{n}\in\mathbb{Z}^3}f(\bm{t}_0+\bm{n}\circ\bm{h})=\frac{1}{h_xh_yh_z}\sum_{\bm{n}\in\mathbb{Z}^3}\hat{f}(2\pi\bm{n}\circ\bm{h}_{\text{inv}})e^{2\pi i \bm{t}_0\cdot (\bm{n}\circ \bm{h}_{\text{inv}})},
\end{equation}
where $\bm{t}_0\in\mathbb{R}^3$ is a constant vector.

\section{A simplified version of the Fourier-Chebyshev solver}\label{appendix::simplified}
When the system is not very strongly confined, i.e. $L_z\not\ll\min\{L_x,L_y\}$, the Fourier-Chebyshev solver proposed in Section~\ref{sec::Chebyshev} can be simplified. One directly samples 
\begin{equation}\label{eq::gridnonfft}
\widehat{S}_{\text{grid}}(\left[\dot{\bm{k}},z\right])=\sum_{\ell=m+1}^{M}w_{\ell}s_{\ell}^2e^{-s_{\ell}^2\dot{k}^2/4}\sum_{j=1}^{N}q_{j}e^{-i\dot{\bm{k}}\cdot\dot{\bm{r}}_j}e^{-(z-z_{j})^2/s_{\ell}^2}
\end{equation}
on a tensor product grid $\mathscr{S}_{\text{F-C}}\subset\left\{(\dot{\bm{k}},z)\in\mathscr{F}^2\times \mathscr{C}\right\}$ where $\mathscr{F}$ and $\mathscr{C}$ denote the set of Fourier nodes on periodic domains and Chebyshev nodes scaled on $[-L_z/2,L_z/2]$, respectively. This step is referred as the gridding step. Next, one applies the Chebyshev interpolation to obtain the Fourier-Chebyshev coefficients $\widetilde{S}_{\text{grid}}([\dot{\bm{k}},n])$ from $\widehat{S}_{\text{grid}}([\dot{\bm{k}},z])$ for each $\dot{\bm{k}}$ and $n\in\{0,\cdots,P-1\}$. Finally, one gathers all the contributions from $\widetilde{S}_{\text{grid}}$, 
\begin{equation}\label{eq::3.13A}
\Phi_{\mathrm{SOG}}^{\rm long}\left([\dot{\bm{r}}_i,z_i]\right)=\frac{2\pi}{L_x L_y L_z}\sum_{\dot{\bm{k}} \in \mathcal{K}^2} \sum_{n=0}^{P-1}{}^{\prime}\, \widetilde{S}_{\text{grid}}([\dot{\boldsymbol{k}},n]) e^{i \dot{\boldsymbol{k}} \cdot \dot{\bm{r}}_i} T_n\left(\frac{2z_i}{L_z}\right),
\end{equation}
by applying the backward 2D Fourier transform and Chebyshev evaluation in periodic and non-periodic dimensions, respectively, which is denoted as the gathering step. The procedure of this simplified Fourier-Chebyshev spectral solver
is summarized in Algorithm~\ref{al::long_nonfft}.
\begin{algorithm}[H]
	\caption{~Fast spectral solver for long-range potential $\Phi_{\text{SOG}}^{\rm long}$ (without FFT)}\label{al::long_nonfft}
	\begin{algorithmic}[1]
	\State (Gridding) Sample $\widehat{S}_{\text{grid}}([\dot{\bm{k}},z])$ on a tensor product grid $\mathscr{S}_{\text{F-C}}$ via Eq.~\eqref{eq::gridnonfft}.
	\State (Interpolation) Obtain $\widetilde{S}_{\text{grid}}([\dot{\boldsymbol{k}},n])$ from $\widehat{S}_{\text{grid}}([\dot{\bm{k}},z])$ by the Chebyshev interpolation.
	\State (Gathering) Recover $\Phi^{\rm long}_{\text{SOG}}([\dot{\bm{r}}_i,z_i])$ by gathering contributions from $\widetilde{S}_{\text{grid}}([\dot{\boldsymbol{k}},n])$ via Eq.~\eqref{eq::3.13A}.
	\end{algorithmic}
\end{algorithm}

\section{Efficient computation of Fourier-Chebyshev coefficients}\label{subsec::Gridding}
In Step $1-3$ of Algorithm~\ref{al::long}, a complexity proportional to $\rm long$, the number of long-range Gaussians, is required to obtain the Fourier-Chebyshev coefficients $\widetilde{S}_{\text{scal}}([\dot{\bm{k}},n])$ which could become costly if $\rm long$ is substantial. Fortunately, we have observed that the long-range Gaussians can be effectively approximated by a Taylor series, which helps simplify Algorithm~\ref{al::long}. 

\begin{lemma}\label{lemma::Taylor}
Each long-range Gaussian satisfying $s_\ell>\eta L_z$ has a $Q$-term truncated Taylor expansion on the interval $[-L_z/2,L_z/2]$ with a remainder of Lagrange type,
\begin{equation}
	\label{eq::Taylor}
	e^{-z^2/s_\ell^2}=\sum_{p=0}^{Q-1}\frac{1}{p!}\left(-\frac{z^2}{s_\ell^2}\right)^{p}+\frac{(-1)^Qe^{-\xi}}{Q!}\cdot \left(-\frac{z^2}{s_\ell^2}\right)^{Q},\quad \xi \in \left[0,\frac{1}{4\eta^2}\right],
\end{equation}
where the remainder can be bounded by 
\begin{equation}\label{eq::Taylor2}
    \left|e^{-z^2/s_\ell^2}-\sum_{p=0}^{Q-1}\frac{1}{p!}\left(-\frac{z^2}{s_\ell^2}\right)^{p}\right|\leq \frac{1}{Q!\,(2\eta)^{2Q}}.
\end{equation}
\end{lemma}

Lemma~\ref{lemma::Taylor} states that the Taylor expansion exhibits rapid convergence, with the estimate rate in Eq.~\eqref{eq::Taylor2} depending on $\eta$. For instance, when $\eta=1$, choosing $Q=4$, $7$, and $12$ provides approximately $4$, $8$, and $15$ decimal digits of accuracy, respectively. Consequently, Steps $1-3$ of Algorithm~\ref{al::long} can be adjusted as follows. First, we sample 
\begin{equation}
S_{\text{grid}}^{p}([\dot{\bm{r}},z])=\sum_{j=1}^{N}q_jW(\dot{\bm{r}}-\dot{\bm{r}}_j)\left(\frac{z-z_j}{L_z}\right)^{2p},\quad p=0,1,\cdots,Q-1,
\end{equation}
on a tensor product grid $\mathscr{S}_{\text{R-C}}\subset\left\{(\dot{\bm{r}},z)\in\mathscr{R}^2\times \mathscr{C}\right\}$, where the factor $L_z$ is added to eliminate ``catastrophic cancellation'' errors. Next, applying 2D FFTs and Chebyshev interpolation to obtain Fourier-Chebyshev coefficients $\widetilde{S}_{\text{grid}}^{p}([\dot{\bm{k}},n])$. Finally, we collect contributions from the Taylor series to obtain 
\begin{equation}
\widehat{S}_{\text{scal}}([\dot{\bm{k}},n]):=\sum_{p=0}^{Q-1}A_{p}(\dot{\bm{k}})\widetilde{S}_{\text{grid}}^{p}([\dot{\bm{k}},n]),
\end{equation}
where the coefficients are given via
\begin{equation}
A_{p}(\dot{\bm{k}})=\frac{(-1)^p\left|\widehat{W}(\dot{\bm{k}})\right|^{-2}}{p!}\sum_{\ell=m+1}^{M}\frac{ w_{\ell}L_z^{2p}}{s_{\ell}^{2p-2}}e^{-s_{\ell}^2\dot{\bm{k}}/4}.
\end{equation}
Importantly, the computational complexity of the above procedure is independent of $\rm long$, as the summation over $\ell$ is solely engaged in $\widehat{A}_n(\dot{\bm{k}})$, which can be precomputed and stored.

\section{The Debye-H\"uckel approximation}\label{app::DH}
We apply the DH theory~\cite{hansen2013theory} to estimate quantities in the form of
\begin{equation}
\mathcal{G}(\bm{k},\bm{r}_i)=\sum_{j=1}^{N}q_{j}e^{i\bm{k}\cdot\bm{r}_{ij}}G(r_{ij}),
\end{equation}
where $G(r)$ is a radial symmetry function with polynomial growth with respect to $r$ at most. The DH theory considers the simplest model of an electrolyte solution confined to the simulation cell, where all $N$
ions are idealized as hard spheres of diameter $a$ carrying charge $\pm q$ at their centers. 

Let us fix one ion of charge $q$ at the origin and consider the distribution of the other ions around it. In the region $\mathring{B}(a):=\{0<r<a\}$, the potential is governed by the Laplace equation $\nabla^2 \phi=0$. Outside $\mathring{B}$, the charge distribution is characterized by the Boltzmann distribution $\rho_{\pm}(\boldsymbol{r})=\pm q \rho_{\infty,\pm} e^{\mp\beta q \phi}$ with $\rho_{\infty, \pm}=N/(2V)$ and $\beta=1/k_{\text{B}}T$. The corresponding Poisson's equation can be linearized as
\begin{equation}
    - \nabla^2 \phi=q \rho_{+} e^{-\beta q \phi}-q \rho_{-} e^{\beta q \phi} \approx -\beta q^2 \rho \phi, \quad r\geq a,
\end{equation}
and its solution is given by
\begin{equation}
    \phi(r)= \begin{cases}\dfrac{q}{4 \pi  r}-\dfrac{q \kappa}{4 \pi (1+\kappa a)}, & r<a, \\ \dfrac{q e^{\kappa a} e^{-\kappa r}}{4 \pi r(1+\kappa a)}, & r\geq a ,\end{cases}
\end{equation}
where $\kappa=\sqrt{\beta q^2\rho}$ denotes the inverse of Debye length $\lambda_{\text{D}}$. As a result, the net charge density for $r\geq a$ reads $\rho(r)=-\kappa^2 \phi(r)$. In the context of quasi-2D systems, ions is confined in $z$-direction. Given these considerations, for $r \geq a$, we obtain the following estimate:
\begin{equation}\label{eq::C}
    \begin{aligned}
        \left|\sum_{j\neq i}q_je^{i\bm{k}\cdot \bm{r}_{ij}}G(r_{ij})\right|&\approx \left|\int_{(\mathbb{R}^3\backslash B(\bm{r}_i,a))\cap \{|z|\le L_z/2\}} \rho(\bm{r})e^{i\bm{k}\cdot \bm{r}}G(r)d\bm{r}\right|\\
        &\leq \frac{q\kappa^2e^{\kappa a}}{(1+\kappa a)}\int_{a}^{\infty}re^{-\kappa r}G(r)dr
    \end{aligned}
\end{equation}
Eq.~\eqref{eq::C} indicates that $|\mathcal{G}(\bm{k},\bm{r}_i)|\leq qC_{\text{DH}} $, where $C_{\text{DH}}$ is a constant depending on the expression of $G(r)$. For instance, if $G(r)\equiv 1$, we have $C_{\text{DH}}=1$; if $G(r)=1/r$, we have $C_{\text{DH}}=\kappa/(1+\kappa a)$.

The results above can be readily extended to multi-component systems using the generalized DH theory~\cite{levin2002electrostatic}. When the charge distribution along the $z$-direction lacks spatial homogeneity, the DH theory may not be accurate. However, one can still expect an upper bound in the form of $|\mathcal{G}(\bm{k},\bm{r}_i)|\leq Cq$. This is attributed to the confinements, ensuring that the integral in Eq.~\eqref{eq::C} along the $z$-direction remains bounded.

\section{Reference parameters}\label{app::ParaSelection}

\text{    }
%Table \ref{tabl:CubicTime2} showcase the experimental parameters for the corresponding time examples mentioned in Section \ref{sec::total}, as shown in .

\renewcommand\arraystretch{1.1}
\begin{table}[H]
\caption{Reference parameters for test systems used in Section \ref{sec::total}, listed for different aspect ratios $\gamma$ and error tolerances $\varepsilon_r$. Note that $I_x=I_y=I_z=I$ and $\mathcal{P}_x=\mathcal{P}_y=\mathcal{P}_z=\mathcal{P}$. 
% For $\gamma=1$, the system has dimensions $L_x=L_y=L_z=20$; for $\gamma=100$, it is strongly confined with dimensions $L_x=L_y=100$ and $L_z=1$. 
In this table, ``$-$'' denotes that the corresponding parameter is not used in the solver.}
	\centering
	\begin{tabular}{ccccccccc}
        \\
         \toprule
         $L$ & $\varepsilon_r$ & $\eta$ & $I$ & $\mathcal{P}$ & $\lambda_z$ &  $I^{\rm long}$ & $Q$ & $P$  \\
         \midrule
         $(20, 20, 20)$ & $10^{-3}$ & $0.68$ &  $(8,8,8)$ & $(9,9,9)$ & $1.47$ & $(1,1)$ & $-$ & $1$  \\
         $(20, 20, 20)$ & $10^{-6}$ & $0.68$ & $(16,16,16)$ & $(13,13,13)$ & $1.64$ & $(3,3)$ & $-$ & $9$  \\
         $(20, 20, 20)$ &  $10^{-12}$ & $0.68$& $(36,36,36)$ & $(17,17,17)$ & $2.23$ & $(5,5)$ & $-$ & $14$  \\
         \\
         $(30, 30, 0.3)$ & $10^{-3}$ & $-$ & $-$ & $(9,9)$ & $-$ & $(16,16)$ & $2$& $2$  \\
         $(30, 30, 0.3)$ & $10^{-6}$ & $-$ & $-$ & $(13,13)$ & $-$ & $(24,24)$ & $4$& $4$  \\
         $(30, 30, 0.3)$ &  $10^{-12}$ & $-$ & $-$ & $(17,17)$ & $-$ & $(64,64)$ & $6$& $8$  \\
         \bottomrule
        \end{tabular}
	\label{tabl:CubicTime2}
\end{table}

\providecommand{\bysame}{\leavevmode\hbox to3em{\hrulefill}\thinspace}
\providecommand{\MR}{\relax\ifhmode\unskip\space\fi MR }
\providecommand{\MRhref}[2]{%
  \href{http://www.ams.org/mathscinet-getitem?mr=#1}{#2}
}
\providecommand{\href}[2]{#2}

%\bibliography{refer,SE}
	

\bibliographystyle{amsplain}

\end{document}